\documentclass{article}
\usepackage[numbers,sort&compress]{natbib}
\usepackage[letterpaper,margin=1.0in]{geometry}
\usepackage[utf8]{inputenc} 
\usepackage[T1]{fontenc}    
\usepackage{hyperref}       
\usepackage{booktabs}       
\usepackage{amsfonts}       
\usepackage{xcolor}         
\usepackage{algorithm}
\usepackage{algorithmic}
\usepackage{graphicx}
\usepackage{enumitem}

\usepackage{amsmath}\allowdisplaybreaks
\usepackage{amsfonts,bm}
\usepackage{amssymb}

\def\eqref#1{equation~(\ref{#1})}

\def\1{\bf{1}}

\newcommand{\Norm}[1]{\left\| #1 \right\|}
\newcommand{\norm}[1]{\left\| #1 \right\|_2}
\def\inner#1#2{\langle #1, #2 \rangle}

\def\eps{{\varepsilon}}

\def\vzero{{\bf{0}}}
\def\vone{{\bf{1}}}

\def\vg{{\bf{g}}}

\def\vr{{\bf{r}}}
\def\vs{{\bf{s}}}

\def\vu{{\bf{u}}}
\def\vv{{\bf{v}}}
\def\vw{{\bf{w}}}
\def\vx{{\bf{x}}}
\def\vy{{\bf{y}}}
\def\vz{{\bf{z}}}

\def\fO{{\mathcal{O}}}
\def\fP{{\mathcal{P}}}

\def\fX{{\mathcal{X}}}
\def\fY{{\mathcal{Y}}}
\def\fZ{{\mathcal{Z}}}

\def\BE{{\mathbb{E}}}

\def\BI{{\mathbb{I}}}

\def\BR{{\mathbb{R}}}

\def\BT{{\mathbb{T}}}

\newcommand{\E}{\mathbb{E}}

\DeclareMathOperator*{\argmin}{arg\,min}

\usepackage{amsthm}
\theoremstyle{plain}
\newtheorem{thm}{Theorem}[section]

\newtheorem{lem}{Lemma}[section]
\newtheorem{asm}{Assumption}[section]
\newtheorem{remark}{Remark}[section]

\makeatletter
\def\Ddots{\mathinner{\mkern1mu\raise\p@
\vbox{\kern7\p@\hbox{.}}\mkern2mu
\raise4\p@\hbox{.}\mkern2mu\raise7\p@\hbox{.}\mkern1mu}}
\makeatother

\makeatletter
\newcommand*{\rom}[1]{\expandafter\@slowromancap\romannumeral #1@}
\makeatother

\def\hDelta {{{\hat \Delta}}}
\def\hdelta {{{\hat \delta}}}
\def\hzeta {{{\hat \zeta}}}

\def\hC {{{\hat C}}}
\def\hr {{{\hat r}}}
\def\bw {{{\bar w}}}
\def\bu {{{\bar u}}}
\def\bs {{{\bar s}}}
\def\bg {{{\bar g}}}

\def\bv {{{\bar v}}}

\def\bu {{{\bar u}}}
\def\bv {{{\bar v}}}
\def\bw {{{\bar w}}}
\def\bx {{{\bar x}}}
\def\by {{{\bar y}}}
\def\bz {{{\bar z}}}

\def\FZ{{\bm \fZ}}
\newcommand{\FM}[1]{{\rm FastMix}\left(#1 \right)}

\title{Decentralized Stochastic Variance Reduced Extragradient Method}
\author{Luo Luo\thanks{School of Data Science, Fudan University; luoluo@fudan.edu.cn} \qquad\qquad\qquad
Haishan Ye\thanks{School of Management, Xi'an Jiaotong University; hsye\_cs@outlook.com} 
}
\date{}

\begin{document}

\maketitle

\renewenvironment{abstract}
 {\small
  \begin{center}
  \bfseries \abstractname\vspace{-.5em}\vspace{0pt}
  \end{center}
  \list{}{
    \setlength{\leftmargin}{1.5cm}%
    \setlength{\rightmargin}{\leftmargin}%
  }%
  \item\relax}
 {\endlist}

\begin{abstract}
This paper studies decentralized convex-concave minimax optimization problems of the form $\min_x\max_y f(x,y) \triangleq\frac{1}{m}\sum_{i=1}^m f_i(x,y)$, where $m$ is the number of agents and each local function can be written as $f_i(x,y)=\frac{1}{n}\sum_{j=1}^n f_{i,j}(x,y)$.
We propose a novel decentralized optimization algorithm, called multi-consensus stochastic variance reduced extragradient, which achieves the best known stochastic first-order oracle (SFO) complexity for this problem. Specifically, each agent requires $\mathcal O((n+\kappa\sqrt{n})\log(1/\varepsilon))$ SFO calls for strongly-convex-strongly-concave problem and $\mathcal O ((n+\sqrt{n}L/\varepsilon)\log(1/\varepsilon))$ SFO call for general convex-concave problem to achieve $\varepsilon$-accurate solution in expectation, where $\kappa$ is the condition number and $L$ is the smoothness parameter. The numerical experiments show the proposed method performs better than baselines.
\end{abstract}

\section{Introduction}

The minimax optimization problem has received increasing attention recently because of it contains lot of popular applications such as empirical risk minimization~\cite{zhang2017stochastic,wang2017exploiting},
robust optimization~\cite{ben2009robust,gao2016distributionally,shafieezadeh2015distributionally,duchi2019variance}, AUC maximization~\cite{hanley1982meaning,ying2016stochastic}, fairness-aware machine learning~\cite{zhang2018mitigating}, policy evaluation~\cite{du2017stochastic}, PID control~\cite{hast2013pid}, game theory~\cite{bacsar1998dynamic,von2007theory} and etc.

We focus on decentralized first-order methods for convex-concave minimax problem where agents are connected by an undirected network. 
Decentralized optimization algorithm is more communication efficient than centralized one because of each agent is only allowed to access its neighbors \cite{lian2017can}, which reduces the communication cost on the busiest agent.
For large-scale machine learning model, the number of training samples could be very large and we are interested in designing stochastic first-order (SFO) algorithms to reduce the computational cost.
Such ideas have been widely used for solving minimization problem \cite{li2020variance,xin2020decentralized,ye2020pmgt,hendrikx2020dual,hendrikx2021optimal,tang2018d,pan2020d,lian2017can}.
Recently, researchers studied decentralized first-order algorithms for minimax optimization. 
\citet{mukherjee2020decentralized} proposed GT-ExtraGradient (GT-EG) algorithm for decentralized minimax problem, which is based on the classical extragradient method~\cite{korpelevich1977extragradient,tseng1995linear} and gradient tracking technique~\cite{pu2021distributed,qu2017harnessing,nedic2017achieving,qu2019accelerated}. GT-EG has linear convergence rate for unconstrained strongly-convex-strongly-concave minimax problem, but its iterations depend on expensive full gradient oracles. 
\citet{beznosikov2020distributed} considered more general convex-concave minimax problem and proposed decentralized extra step method (DESM), which iterates with stochastic gradient oracle and achieves sub-linear convergence rate under bounded variance assumption. Additionally, \citet{liu2020decentralized,xian2021faster} studied minimax problem without convex-concave assumption.

In this paper, we propose a novel stochastic optimization algorithm, called multi-consensus stochastic variance reduced extragradient (MC-SVRE), which is designed by integrating the ideas of extragradient method \cite{korpelevich1977extragradient,tseng1995linear}, gradient tracking \cite{qu2017harnessing,nedic2017achieving,qu2019accelerated}, variance reduction \cite{johnson2013accelerating,zhang2013linear,hofmann2015variance,schmidt2017minimizing,defazio2014saga,allen2017katyusha,fang2018spider,chavdarova2019reducing,palaniappan2016stochastic,luo2019stochastic,luo2020stochastic,luo2021near,yang2020global,kovalev2020don,alacaoglu2021stochastic,xian2021faster} and multi-consensus~\cite{ye2020multi,ye2020pmgt,DBLP:conf/nips/YeZL020}.
MC-SVRE achieves the best known SFO complexity for decentralized convex-concave minimax optimization problem.
It is the first stochastic decentralized algorithm that enjoys linear convergence for strongly-convex-strongly-concave minimax problem and also works for general convex-concave case.
Furthermore, we can derive a deterministic decentralized algorithm, called multi-consensus extragradient (MC-EG), by replacing the stochastic variance reduced gradient in MC-SVRE by full gradient. MC-EG requires more computational cost than MC-SVRE, while it (nearly) matches the communication and computational lower bound of deterministic first-order algorithms. We also conduct the numerical experiments on several machine learning applications to show the outperformance of proposed algorithms.

\paragraph{Paper Organization} 
In Section \ref{sec:setting}, we give the problem setting and preliminaries. 
In Section \ref{sec:related}, we we survey the previous work for decentralized minimax optimization.
In Section \ref{sec:MC-SVRE}, we propose MC-SVRE method and provide the complexity analysis.  
convex-concave assumption.
In Section \ref{sec:MC-EG}, we derive the deterministic algorithm MC-EG from MC-SVRE and provide the related theoretical results.
In Section \ref{sec:experiments}, we provide the empirical studies to show the effectiveness of proposed algorithms. 
Finally, we conclude this work in Section \ref{conclusion}.
All proofs are deferred to the appendix.

\section{Problem Setting and Preliminaries}\label{sec:setting}

We study the following decentralized minimax optimization problem of the form
\begin{align}\label{prob:main}
\min_{x\in\fX}\max_{y\in\fY} f(x,y) \triangleq \frac{1}{m}\sum_{i=1}^m f_i(x,y),
\end{align}
where $m$ is the number of agents and $f_i(x,y)$ is the local function on $i$-th agent with $n$ components as follows
\begin{align}\label{eq:finite-sum}
    f_i(x,y) = \frac{1}{n}\sum_{i=1}^n f_{i,j}(x,y).
\end{align}
Our task is finding the saddle point of $(x^*,y^*)\in\fX\times\fY$ of $f(x,y)$ which satisfies 
\begin{align*}
f(x^*,y) \leq f(x^*,y^*) \leq f(x,y^*)    
\end{align*}
for any $x\in\fX$ and $y\in\fY$. For the ease of presentation, we also denote $z=[x;y]\in\BR^d$ and $z^*=[x^*;y^*]\in\BR^d$, where $d=d_x+d_y$.

We consider the following assumptions for problem (\ref{prob:main}).
\begin{asm}\label{asm:smooth}
We suppose each $f_{i,j}(x,y)$ is $L$-smooth for $L>0$, that is, we have
\begin{align*}
\Norm{\nabla f_{i,j}(x,y)-\nabla f_{i,j}(x',y')}^2 \leq L^2\big(\Norm{x-x'}^2+\Norm{y-y'}^2\big).
\end{align*}
for any $x,x'\in\BR^{d_x}$ and $y,y'\in\BR^{d_y}$, 
\end{asm}
\begin{asm}\label{asm:cc}
We suppose $f(x,y)$ is convex-concave, that is, we have
\begin{align*}
f(x',y) \geq f(x,y) + \inner{\nabla_x f(x,y)}{x'-x}    
\end{align*}
and 
\begin{align*}
f(x,y') \leq f(x,y) + \inner{\nabla_y f(x,y)}{y'-y}    
\end{align*}
for any $x,x'\in\BR^{d_x}$ and $y,y'\in\BR^{d_y}$.
\end{asm}
\begin{asm}\label{asm:scsc}
We suppose $f(x,y)$ is $\mu$-strongly-convex-$\mu$-strongly-concave for $\mu>0$, that is, the function defined as
$f(x,y)-\frac{\mu}{2}\Norm{x}^2+\frac{\mu}{2}\Norm{y}^2$
is convex-concave.
\end{asm}
\begin{asm}\label{asm:unconstrained}
We suppose $\fX=\BR^{d_x}$ and $\fY=\BR^{d_y}$.
\end{asm}
\begin{asm}\label{asm:constrained}
We suppose $\fX\subseteq\BR^{d_x}$ and $\fY\subseteq\BR^{d_y}$ are convex and compact with diameter $D$, that is, we have
$\Norm{x\!-\!x'}^2+\Norm{y\!-\!y'}^2 \leq\!D^2$ 
for any $x, x'\!\in\fX$ and $y, y'\!\in\fY$.
\end{asm}
\noindent
Based on above assumptions, we will study decentralized minimax optimization for the following three cases\footnote{Note that if the problem is unconstrained and we only suppose $f(x,y)$ is convex-concave, the saddle point may be not well-defined on $\BR^{d_x}\times \BR^{d_y}$, such as $f(x,y)=x+y$ for $x,y\in\BR$.}: 
\begin{enumerate}[label=(\alph*),topsep=0pt,leftmargin=0.7cm]
\item The objective function $f(x,y)$ is $\mu$-strongly-convex-$\mu$-strongly-concave and the problem is unconstrained (Assumption \ref{asm:smooth}, \ref{asm:scsc} and \ref{asm:unconstrained}).\label{case:a}
\item The objective function $f(x,y)$ is $\mu$-strongly-convex-$\mu$-strongly-concave and the constrained sets $\fX$ and $\fY$ are convex and compact (Assumption \ref{asm:smooth}, \ref{asm:scsc} and \ref{asm:constrained}).\label{case:b}
\item The objective function $f(x,y)$ is convex-concave and the constrained sets $\fX$ and $\fY$ are convex and compact (Assumption \ref{asm:smooth}, \ref{asm:cc} and \ref{asm:constrained}).\label{case:c}
\end{enumerate}
For cases \ref{case:a} and \ref{case:b}, we denote the condition number of the problem as $\kappa\triangleq L/\mu$.

We define aggregate variables
\begin{align*}
\vx=[\vx(1),\cdots,\vx(m)]^\top\in\BR^{m\times d_x}    
\end{align*}
and
\begin{align*}
\vy=[\vy(1),\cdots,\vy(m)]^\top\in\BR^{m\times d_y},
\end{align*}
where $\vx(i)\in\BR^{d_x}$ and $\vy(i)\in\BR^{d_y}$ are local variables of the $i$-th agent. 
We also denote $\vz=[\vx,\vy]\in\BR^{m\times d}$, where $\vz(i)=[\vx(i);\vy(i)]\in\BR^d$ and $d=d_x+d_y$.

We define the gradient operator of $f(x,y)$ as
\begin{align*}
    g(x,y) = [\nabla_x f(x,y); \nabla_y f(x,y)]\in\BR^d.
\end{align*}
We also introduce the aggregate gradient operator as 
\[\vg(\vz)=[g_1(\vz(1)),\cdots,g_m(\vz(m))]^\top\in\BR^{m\times d},\] 
where 
$g_i(\vz(i))=[\nabla_x f_i(\vx(i),\vy(i)); -\nabla_y f_i(\vx(i),\vy(i))]$.

We use the lower case with a bar to represents the mean vector for the rows of corresponding aggregate variable, such as $\bz=\frac{1}{m}\vone^\top\vz\in\BR^{1\times d}$. 
We denote $\bz^*$ and $\bg(\bz)$ as the transpose of $z^*$ and $g(\bz^\top)$ respectively.

In decentralized optimization, the communication step is typically written as the matrix multiplication 
\[\vz^{\rm new} = W \vz^{\rm old},\] 
where $W\in\BR^{m\times m}$ is the gossip matrix associate with $m$ agents and it holds the following assumption.
\begin{asm}\label{asm:W}
We suppose the gossip matrix $W$ satisfies:
\begin{enumerate}[label=(\alph*),topsep=0pt,leftmargin=0.7cm]
\item $W\in\BR^{m\times m}$ is symmetric with $W_{i,j}\neq 0$ if and only if $i$ and $j$ are connected or $i\neq j$;
\item $\vzero \preceq W \preceq I$, $W\vone = \vone$ and ${\rm null}(I-W)={\rm span}(\vone)$;
\end{enumerate}
where $I$ is identity matrix and $\vone=[1,\cdots,1]^\top\in\BR^m$.
\end{asm}

\begin{algorithm}[t]
\caption{${\rm FastMix}(\vz^{(0)})$} \label{alg:fm}
\begin{algorithmic}[1]
    \STATE \textbf{Initialize:} $\vz^{(-1)}=\vz^{(0)}$, $\eta_u=\frac{1-\sqrt{1-\lambda_2^2(W)}}{1+\sqrt{1-\lambda_2^2(W)}}$. \\[0.15cm]
    \STATE \textbf{for} $k = 0, 1, \dots, K$ \textbf{do}\\[0.15cm]
    \STATE\quad $\vz^{(k+1)}=(1+\eta_u)W\vz^{(k)}-\eta_u\vz^{(k-1)}$ \\[0.15cm] 
    \STATE\textbf{end for} \\[0.15cm]
    \STATE \textbf{Output:} $\vz^{(K)}$.
\end{algorithmic}
\end{algorithm}

We denote $\lambda_2(W)$ as the second largest eigenvalue of $W$.
Note that Assumption \ref{asm:W} means $\lambda_2(W)$ is strictly less than one. Hence, we define $\chi \triangleq 1/(1-\lambda_2(W))$.

\citet{liu2011accelerated} proposed an efficient way to achieve average of local variables described in Algorithm \ref{alg:fm} and it holds the following convergence result.

\begin{lem}[{\citet[Lemma 2]{liu2011accelerated}}]\label{lem:FM}
Under Assumption \ref{asm:W}, Algorithm~\ref{alg:fm} holds 
$\frac{1}{m}\vone^\top\vz^{(K)} = \bz^{(0)}$ and
\begin{align*}
\big\|\vz^{(K)}-\vone\bu\big\| \leq  \left(1-\sqrt{1-\lambda_2(W)}\right)^K \big\|\vz^{(0)}-\vone\bz^{(0)}\big\|,
\end{align*}
where $\bz^{(0)} = \frac{1}{m}\vone^\top\vz^{(0)}$ and $\Norm{\cdot}$ is the Frobenius norm.
\end{lem}

We also denote $\fZ=\{z=[x;y]: x\in\fX ~\text{and}~ y\in\fY\}$ and define the projection from $\BR^{d_x+d_y}$ to $\fZ$ as
\begin{align*}
    \fP_\fZ(z) = \argmin_{z'\in\fZ} \Norm{z-z'}.
\end{align*}
Similarly, the corresponding projection operator for aggregate variable is defined as 
\begin{align*}
\fP_{\FZ}(\vz) = [\fP_{\fZ}(\vz(1)), \cdots, \fP_{\fZ}(\vz(m))]^\top.
\end{align*}

\section{Related Work}\label{sec:related}

The extragradient (EG) method \cite{korpelevich1977extragradient,tseng1995linear} is the optimal \cite{zhang2019lower} batch gradient algorithm for convex-concave minimax problem on single machine, whose iteration is based on
\begin{align*}
\begin{cases}
z_{t+1/2} = \fP_\fZ(z_t - \eta_1 g(z_t)), \\
z_{t+1} = \fP_\fZ(z_t - \eta_1 g(z_{t+1/2})),
\end{cases}
\end{align*}
where $\eta_1=\Theta(1/L)$ is the stepsize.

\citet{mukherjee2020decentralized} combined the idea of EG gradient with gradient tracking and proposed GT-Extragradient (GT-EG) algorithm for unconstrained decentralized minimax optimization problem (\ref{prob:main}). The update rule of GE-EG has the following compact form
\begin{align*}
\begin{cases}
\vz_{t+1/2} = \vz_t - \eta_2 \vs_t, \\
\vs_{t+1/2} = \vs_t +  \vg(\vz_{t+1/2}) - \vg(\vz_t) \\
\vz_{t+1} = W \vz_t - \eta_2 \vs_{t+1/2}, \\
\vs_{t+1} = W \vs_t + \vg(z_{t+1}) - \vg(\vz_t),
\end{cases}
\end{align*}
where $\eta_2=\Theta\big(\big(\frac{1}{\kappa\chi^2 L} \vee \frac{1}{\kappa^{1/3}\chi^{4/3}L}\big) \wedge \frac{1}{\chi^2L}\big)$.
The analysis shows GT-EG has linear convergence rate under Assumption \ref{asm:scsc} and \ref{asm:unconstrained}, but it converges slower than EG on single machine since $\eta_2$ could be much smaller than $\eta_1$ when $\kappa$ or $\chi$ is large. Additionally, the iteration of GT-EG requires each agent to compute the full gradient of local function, which is very expensive for large-scale problems.
\citet{beznosikov2020distributed} proposed a stochastic algorithm for decentralized minimax optimization, called decentralized extra step method (DESM), which iterates with stochastic gradient to reduce the computational cost. DESM has sub-linear convergence rate under bounded variance assumption\footnote{The settings for the analysis of DESM~\cite{beznosikov2020distributed} is different from ours. It depends on the assumption that we can access the stochastic gradient operator $g_i(z;\xi)$ such that $\BE_\xi\Norm{g_i(z)-g_i(z;\xi)}^2\leq\sigma^2$, while this paper suppose each $f_i$ has the finite-sum form of (\ref{eq:finite-sum}).}.

Variance reduction is a popular way to improve the convergence of stochastic optimization algorithms and it has been successfully used in large-scale minimax optimization~\cite{palaniappan2016stochastic,chavdarova2019reducing,luo2019stochastic,luo2020stochastic,luo2021near,yang2020global,alacaoglu2021stochastic,xian2021faster}. 
\citet{chavdarova2019reducing} first incorporated variance reduction technique with EG algorithm. Later, \citet{alacaoglu2021stochastic} proposed another variance reduced algorithm based on the framework of loopless SVRG~\cite{hofmann2015variance,kovalev2020don}, which achieves the optimal SFO complexity for single machine algorithms~\cite{luo2021near}.

\section{Multi-Consensus Stochastic Variance Reduced Extragradient}\label{sec:MC-SVRE}

We present our multi-consensus stochastic variance reduced extragradient (MC-SVRE) method in Algorithm \ref{alg:dsvre}. Each agent iterate with local variance reduced gradient estimator 
\[\vv_{t+1/2}(i) = g_i(\vw_t(i)) + g_{i,j_i}(\vz_{t+1/2}(i)) - g_{i,j_i}(\vw_t(i))\]
where $\vw_t$ is the most recent iteration with $g_i(\cdot)$ being exactly evaluated. We use multi-consensus (Algorithm \ref{alg:fm}) to achieve gradient estimators $\vs_t$ and $\vs_{t+1/2}$ for updating $\vz_{t+1/2}$ and $\vz_{t+1}$ respectively, resulting the following lemma.

\begin{algorithm*}[t]
\caption{Multi-Consensus Stochastic Variance Reduced Extragradient (MC-SVRE)} \label{alg:dsvre}
\begin{algorithmic}[1]
    \STATE \textbf{Initialize:} $\vw_0=\vz_0=[z_0^\top;\dots;z_0^\top]$ with $z_0\in\fX\times\fY$, $\vv_{-1}=\vs_{-1}=\vzero$, $p=1/(2n)$, $\alpha=1-p$. \\[0.15cm]
    \STATE $\vv_0=\vs_0=\FM{\vg(\vz_0), K_0}$ \\[0.15cm]
    \STATE \textbf{for} $t = 0, 1, \dots, T-1$ \textbf{do}\\[0.15cm]
    \STATE\quad $\vz'_t = \alpha \vz_t + (1 - \alpha) \vw_t$ \\[0.15cm]
    \STATE\quad $\vs_{t} = \FM{\vs_{t-1}+\vv_t-\vv_{t-1}, K}$ \\[0.15cm]
    \STATE\quad $\vz_{t+1/2} = \FM{\fP_\FZ\left(\vz'_t - \eta \vs_t\right),K}$ \\[0.15cm]
    \STATE\quad \textbf{parallel for} $i = 1, \dots, m$ \textbf{do}\\[0.15cm]
    \STATE\quad\quad Draw an index $j_i \in [n]$ uniformly at random. \\[0.15cm]
    \STATE\quad\quad $\vv_{t+1/2}(i) = g_i(\vw_t(i)) + g_{i,j_i}(\vz_{t+1/2}(i)) - g_{i,j_i}(\vw_t(i))$ \\[0.15cm]
    \STATE\quad\textbf{end parallel for} \\[0.1cm]
    \STATE\quad $\vs_{t+1/2} = \FM{\vs_{t}+\vv_{t+1/2}-\vv_t, K}$ \\[0.15cm]
    \STATE\quad $\vz_{t+1} = \FM{\fP_\FZ\left(\vz'_t - \eta\vs_{t+1/2}\right), K}$ \\[0.15cm]
    \STATE\quad $\begin{cases} 
    \vw_{t+1} = \vz_{t+1},~~~\vv_{t+1}=\vg(\vz_{t+1}),~~\vs_{t+1} = \FM{\vs_{t}+\vv_{t+1}-\vv_{t}, K}  &  \text{ with probability } p, \\[0.1cm]
    \vw_{t+1} = \vw_t,~~~~~~\vv_{t+1}=\vv_t,~~~~~~~~~~~\vs_{t+1}=\vs_t & \text{ with probability } 1 - p. \end{cases}$ \\[0.2cm]
\STATE\textbf{end for} \\[0.15cm]
\end{algorithmic}
\end{algorithm*}

\begin{lem}\label{lem:grad-avg}
We suppose Assumption~\ref{asm:smooth} holds.
For Algorithm \ref{alg:dsvre}, it holds that
$\Norm{\bs_t - \bg(\bw_t)} \leq \frac{L}{\sqrt{m}}\Norm{\vw_t - \vone\bw_t}$
and
$\Norm{\BE_t[\bs_{t+1/2}] - \bg(\bz_{t+1/2})} \leq \frac{L}{\sqrt{m}}\Norm{\vz_{t+1/2} - \vone\bz_{t+1/2}}$.
\end{lem}

Lemma \ref{lem:grad-avg} indicates $\bs_t$ and $\bs_{t+1/2}$ are good gradient estimators at $\bz_t$ and $\bw_t$ respectively when the consensus error $\Norm{\vw_t - \vone\bw_t}$ and $\Norm{\vz_{t+1/2} - \vone\bz_{t+1/2}}$ are small. Hence, we can characterize the convergence of Algorithm \ref{alg:dsvre} by showing how the mean variables $\bz_t$ and $\bw_t$ converge to $\bz^*$. 

We also introduce the vector
\begin{align*}
r_t^2 =\big[\Norm{\vz_t-\vone\bz_t}^2, \Norm{\vw_t-\vone\bw_t}^2, \eta^2\Norm{\vs_t-\vone\bs_t}^2\big]^\top.
\end{align*}
for analyzing the consensus error.
We always denote that
$\rho \triangleq \big(1-\sqrt{1-\lambda_2(W)}\big)^K < 1$, where $K$ is the number of iterations in FastMix (Algorithm \ref{alg:fm}).
For MC-SVRE (Algorithm \ref{alg:dsvre}), we use $\BE_t[\cdot]$ to present the expectation by fixing $\vz_0,\dots,\vz_t, \vw_0,\dots,\vw_t$.

The remainder of this section provides the convergence analysis for proposed MC-SVRE (Algorithm \ref{alg:dsvre}) under different kinds of assumptions. 
Table \ref{table:scsc-uc}, \ref{table:scsc-c} and \ref{table:cc} present the detailed comparison of proposed methods with existing algorithms for decentralized minimax optimization.

\subsection{Unconstrained Case}\label{sec:mc-svre-unconstrained}

We first consider the unconstrained minimax problem with  $\mu$-strongly-convex-$\mu$-strongly-concave assumption. In such case, projection steps in Algorithm \ref{alg:dsvre} are unnecessary and Lemma \ref{lem:FM} implies the update of mean variables can be written as
\begin{align}\label{eq:mean-update}
\begin{cases}
\bz'_t = \alpha \bz_t + (1 - \alpha) \bw_t, \\
\bz_{t+1/2} = \bz'_t - \eta\bs_t, \\
\bz_{t+1}  = \bz'_t - \eta\bs_{t+1/2}.
\end{cases}
\end{align}
By imitating the analysis of centralized algorithms and considering the consensus error from decentralized setting, we obtain the following lemma.

\begin{lem}\label{lem:zw}
Suppose Assumption~\ref{asm:smooth}, \ref{asm:scsc} and \ref{asm:unconstrained} hold. For Algorithm \ref{alg:dsvre} with $\eta = 1/\left(6\sqrt{n} L\right)$ and $c_1>0$, we have
\begin{align}\label{ieq:say-delta}
\begin{split}
& \left(1 + \frac{\eta \mu}{2} - c_1p\right) \BE_t\Norm{\bz_{t+1} - \bz^*}^2 + c_1 \BE_t\Norm{\bw_{t+1} - \bz^*}^2  \\
\leq & (1 - p) \Norm{\bz_t - \bz^*}^2 + (p + c_1(1 - p)) \Norm{\bw_t - \bz^*}^2 - \frac{\delta_t}{3n} 
  + \frac{8 L\eta\left(2\kappa + 3L\eta\right)\vone^\top r_t^2}{m}
\end{split}
\end{align}
where 
$\delta_t=\BE_t\left[\Norm{\bz_{t+1/2} - \bz_{t+1}}^2 + \Norm{\bz_{t+1/2} - \bw_t}^2\right]$.
\end{lem}

We also have the recursion of consensus error as follows.
\begin{lem}\label{lem:error}
Under the settings of Lemma \ref{lem:zw} and suppose $\rho\leq 1/\sqrt{291}$. Then for Algorithm \ref{alg:dsvre} holds that
\begin{align}\label{ieq:say-delta2}
    \BE_t\left[\vone^\top r_{t+1}^2\right]
\leq  \left(1-\frac{1}{3n}\right) \vone^\top r_t^2 +
4\rho^2m \delta_t.
\end{align}
\end{lem}

\begin{table*}[t]
\centering
\caption{\small Complexities to obtain $\BE\Norm{\bz-\bz^*}^2\leq\eps$ for unconstrained and strongly-convex-strongly-concave case} \vskip 0.15cm
\label{table:scsc-uc}
{\small\begin{tabular}{cccc}
\hline 
Algorithm & SFO Complexity & Communication Complexity & Reference  \\\hline 
\addlinespace
MC-SVRE & $\fO\left(\left(n+\kappa\sqrt{n}\right)\log\left(\frac{1}{\eps}\right)\right)$ & $\fO\left((n+\kappa\sqrt{n})\sqrt{\chi}\log(\kappa n)\log\left(\frac{1}{\eps}\right)\right)$  & Algorithm \ref{alg:dsvre} (Theorem \ref{thm:scsc-unconstrained}) \\ \addlinespace
MC-EG & $\fO\left(\kappa n\log\left(\frac{1}{\eps}\right)\right)$ & $\fO\left(\kappa\sqrt{\chi}\log\kappa\log\left(\frac{1}{\eps}\right)\right)$  & Algorithm \ref{alg:deg} (Theorem \ref{thm:eg-unconstrained}) \\ \addlinespace
GTEG & $\fO\big(\big(\kappa^{4/3}\chi^{4/3}\vee \kappa\chi^2\big)n\log\big(\frac{1}{\eps}\big)\big)$  & $\fO\big(\big(\kappa^{4/3}\chi^{4/3}\vee\kappa\chi^2\big)\log\big(\frac{1}{\eps}\big)\big)$  &  \citet{mukherjee2020decentralized} \\\addlinespace \hline
\end{tabular}}\vskip-0.25cm
\end{table*}

\begin{table*}[t]
\centering
\caption{Complexities to obtain $\BE\Norm{\bz-\bz^*}^2\leq\eps$ for constrained and strongly-convex-strongly-concave case}\vskip 0.15cm
\begin{tabular}{cccc}
\hline 
Algorithm & SFO Complexity & Communication Complexity & Reference  \\\hline 
\addlinespace
MC-SVRE & $\fO\left(\left(n+\kappa\sqrt{n}\right)\log\left(\frac{1}{\eps}\right)\right)$ & $\fO\left((n+\kappa\sqrt{n})\sqrt{\chi}\log(\frac{\kappa n}{\eps})\log\left(\frac{1}{\eps}\right)\right)$  & Algorithm~\ref{alg:dsvre} (Theorem \ref{thm:scsc-constrained})  \\ \addlinespace
MC-EG & $\fO\left(\kappa n\log\left(\frac{1}{\eps}\right)\right)$ & $\fO\left(\kappa\sqrt{\chi}\log\left(\frac{\kappa }{\eps}\right)\log\left(\frac{1}{\eps}\right)\right)$  & Algorithm \ref{alg:deg} (Theorem \ref{thm:eg-scsc-constarined})  \\ \addlinespace
DESM & $\fO\left(\frac{\sigma^2}{\mu^2\eps}\right)$ &  $\fO\left(\kappa\sqrt{\chi}\log(\frac{\kappa }{\eps})\log\left(\frac{1}{\eps}\right)\right)$ & \citet{beznosikov2020distributed} \\ \addlinespace \hline
\end{tabular}
\label{table:scsc-c}
\end{table*}

\begin{table*}[t]
\centering
\caption{Complexities to obtain $\BE\left[f(\hat x,y^*)-f(x^*,\hat y)\right]\leq\eps$ for constrained and convex-concave case}\vskip 0.15cm
\begin{tabular}{cccc}
\hline 
Algorithm & SFO Complexity & Communication Complexity & Reference\\\hline
\addlinespace
MC-SVRE & $\fO\left(n + \frac{\sqrt{n}L}{\eps}\right)$  & $\fO\left(\frac{\sqrt{n\chi}L}{\eps}\log\left(\frac{nL}{\eps}\right)\right)$  & Algorithm \ref{alg:dsvre} (Theorem \ref{thm:cc-constrained})   \\ \addlinespace 
MC-EG & $\fO\left(\frac{nL}{\eps}\right)$  & $\fO\left(\frac{L\sqrt{\chi}}{\eps}\log\left(\frac{L}{\eps}\right)\right)$  & Algorithm \ref{alg:deg} (Theorem \ref{thm:eg-cc-constrained})    \\ \addlinespace 
DESM & $\fO\left(\frac{\sigma^2}{\eps^2}\right)$  & $\fO\left(\frac{L\sqrt{\chi}}{\eps}\right)$  & \citet{beznosikov2020distributed}  \\ \addlinespace
\hline
\end{tabular}
\label{table:cc}
\end{table*}

Note that the term related to $\delta_t$ in (\ref{ieq:say-delta}) is typically ignored in the study of single machine convex-concave minimax optimization~\cite{mokhtari2020unified,alacaoglu2021stochastic}, while it is useful to the analysis of decentralized algorithms.
Specifically, connecting Lemma \ref{lem:zw} and \ref{lem:error} with sufficient small $\rho$ and appropriate $c_1$, we can cancel the last term of (\ref{ieq:say-delta2}) and establish the linear convergence of the weighed sum of $\Norm{\bz_{t} - \bz^*}^2$, $\Norm{\bw_{t} - \bz^*}^2$ and $\vone^\top r_{t}$.
\begin{lem}\label{lem:scsc-unconstrained}
Under the settings of Lemma \ref{lem:error}, define
\begin{align*}
V_t = \Norm{\bz_{t} - \bz^*}^2 + \frac{c_1\Norm{\bw_{t} - \bz^*}^2}{1 + \frac{\eta \mu}{2} - c_1p}  + \frac{c_2C\vone^\top r_{t}}{1 + \frac{\eta \mu}{2} - c_1p}
\end{align*}
where 
$C = \frac{8 L\eta\left(2\kappa + 3L\eta\right)}{m}$, 
$c_1 =\frac{2\eta \mu + 4 p}{\eta \mu + 4 p}$
and $c_2=\frac{3}{p}$. 
Then Algorithm \ref{alg:dsvre}  holds that
\begin{align*}
\BE_t\left[V_{t+1}\right] 
\leq \left(1-\frac{1}{6(n+4\kappa\sqrt{n})}\right)  V_t
\end{align*}
by taking 
$K = \left\lceil\sqrt{\chi}\log\left(\max\left\{\sqrt{12mnc_2C},~\sqrt{291}\right\}\right)\right\rceil$.
\end{lem}

The update rule of $\vv_{t+1}$ indicates the full gradient $\vg(\vz_{t+1})$ is computed with probability $p=1/2n$ at each iteration. Hence, MC-SVRE has $\fO(Tnp)=\fO(T)$ SFO complexity totally in expectation.
Note that Algorithm \ref{alg:dsvre} runs FastMix with $K_0$ iterations on $\vg(\vz_0)$ in line 2, which is used to upper bound the term $\vone^\top r_0$ in $V_0$. Finally, we obtain our main result for the unconstrained and strongly-convex-strongly-concave case.

\begin{thm}\label{thm:scsc-unconstrained}
Under the settings of Lemma \ref{lem:scsc-unconstrained},
let
\begin{align*}
T = \left\lceil 6(n+4\kappa\sqrt{n})\log\left(\frac{6\Norm{\bz_0-z^*}^2 + \eps}{\eps}\right)\right\rceil, \quad
K_0 = \left\lceil\sqrt{\chi}\log\left(\frac{10\kappa}{3m\sqrt{n}L^2\eps}\Norm{\vg(\vz_0)-\frac{1}{m}\vone\vone^\top\vg(\vz_0)}^2 \right)\right\rceil
\end{align*}
and suppose $n\geq 2$ for Algorithm \ref{alg:dsvre}. Then it requires at most $\fO\left(\left(n+\kappa\sqrt{n}\right)\log\left(1/\eps\right)\right)$
SFO calls on each agent in expectation and 
$\fO\left(\left(n+\kappa\sqrt{n}\right)\sqrt{\chi}\log(\kappa n)\log\left(1/\eps\right)\right)$
communication rounds to obtain 
$\BE_T\Norm{\bz_{T}-\bz^*}^2 \leq \eps$.
\end{thm}

\subsection{The Constrained Case}\label{sec:mc-svre-constrained}

For the constrained minimax problem, the updates rule of mean vector can be written as
\begin{align*}
\begin{cases}
\bz'_t = \alpha \bz_t + (1 - \alpha) \bw_t, \\
\bz_{t+1/2} =  \fP(\bz'_t - \eta\bs_t)  + \Delta_t,  \\
\bz_{t+1}  =  \fP(\bz'_t - \eta\bs_{t+1/2}) + \Delta_{t+1/2}. \\
\end{cases}
\end{align*}
where $\Delta_{t+1/2}=\frac{1}{m}\vone^\top\fP(\vz'_t-\eta\vs_{t+1/2}) - \fP(\bz'_t - \eta\bs_{t+1/2})$ and $\Delta_t=\frac{1}{m}\vone^\top\fP(\vz_t'-\eta\vs_t) - \fP(\bz'_t - \eta\bs_t)$. 
Compared with update rule (\ref{eq:mean-update}) for unconstrained case, we need to deal with the additional term $\Delta_t$ and $\Delta_{t+1/2}$.

For the strongly-convex-strongly-concave case, we establish the recursion relationship of $\bz_t$ and $\bw_t$ as follows.
\begin{lem}\label{lem:rel-constrained}
Suppose Assumption~\ref{asm:smooth}, \ref{asm:scsc} and \ref{asm:constrained} hold. For Algorithm \ref{alg:dsvre} with $\eta = 1/\left(6\sqrt{n} L\right)$, we have
\begin{align*}
\begin{split}
& \!\left(1 + \frac{\eta \mu}{2} - c_1p\right) \BE_t \Norm{\bz_{t+1} - \bz^*}^2 + c_1 \E_{t}\Norm{\bw_{t+1} - \bz^*}^2  \\
\leq & \left(1 - p\right) \Norm{\bz_t - \bz^*}^2 + \left(p + c_1\left(1 - p\right)\right) \Norm{\bw_t - \bz^*}^2 + \zeta_t
\end{split}
\end{align*}
where
\begin{align*}
\zeta_t = & \frac{2L\eta\left(2\kappa+
3L\eta\right)}{m}\left(\BE_t\Norm{\vz_{t+1/2}-\vone\bz_{t+1/2}}^2+ \Norm{\vw_t-\vone\bw_t}^2\right) \\
& + \BE_t\left[\Norm{\bz_{t+1}-\bz'_t+\eta \bs_{t+1/2}+\bz_{t+1}-\bz^*}\Norm{\Delta_{t+1/2}}\right]  
  + \BE_t\left[\Norm{\bz_{t+1/2} - \bz'_t + \eta\bs_t+\bz_{t+1/2} -\bz_{t+1}}\Norm{\Delta_t}\right]. 
\end{align*}
\end{lem}

Using the bounded assumption on $\fX\times\fY$ and Lemma \ref{lem:FM}, we  bound the term $\zeta_t$ as follows.

\begin{lem}\label{lem:zeta}
Under the settings of Lemma \ref{lem:rel-constrained},
for Algorithm \ref{alg:dsvre} with 
$K\geq\lceil\sqrt{\chi}\log\left(2\left(\sqrt{m}LD+\delta'\right)/\delta'\right)\rceil$ and $\Norm{\vs_0 - \vone\bs_0}\leq\delta'$ for some $\delta'>0$, it holds that
\begin{align*}
\zeta_t \leq \frac{4L\eta\left(2\kappa + 3L\eta\right)\delta'^2}{m} + \frac{6(C_1+C_{1/2})\delta'}{\sqrt{m}}
\end{align*}
where 
$C_1 = 2D + \frac{1}{6\sqrt{n}L}\sqrt{2L^2D^2 + \frac{2}{m}\sum_{i=1}^m\Norm{g_i(z^*)}^2}$ and
$C_{1/2} = D + \frac{1}{6\sqrt{n}L}\sqrt{6L^2D^2+\frac{3}{m}\sum_{i=1}^m\Norm{g_i(z^*)}^2}$.
\end{lem}

By appropriate setting of $K_0$, we can make $\delta'$ and $\zeta_t$ be sufficient small and derive the complexity of Algorithm \ref{alg:dsvre} for constrained strongly-convex-strongly-concave case.

\begin{thm}\label{thm:scsc-constrained}
Under the settings of Lemma \ref{lem:zeta}, let
\begin{align*}
T = \left\lceil6\left(n+4\kappa\sqrt{n}\right)\log\left(\frac{10\Norm{\bz_0-\bz^*}^2}{\eps}\right)\right\rceil
\quad\text{and}\quad 
K_0 = \left\lceil\sqrt{\chi}\log\left(\frac{1}{\delta'}\Norm{\vg(\vz_0)-\frac{1}{m}\vone\vone^\top\vg(\vz_0)}^2\right)\right\rceil.
\end{align*}
and suppose $n\geq 2$ for Algorithm \ref{alg:dsvre}. 
Then it requires at most
$\fO\left(\left(n+\kappa\sqrt{n}\right)\log\left(1/\eps\right)\right)$
SFO calls on each agent in expectation and
{\small$\fO\left((n+\kappa\sqrt{n})\sqrt{\chi}\log\left(n\kappa/\eps\right)\log\left(1/\eps\right)\right)$}
communication rounds to obtain $\BE_T\Norm{\bz_T - \bz^*}^2\leq\eps$, where $\delta'=\min\Big\{\sqrt{\frac{m\eps/(192L\eta)}{(n+4\kappa\sqrt{n})\left(2\kappa + 3L\eta\right)}}, \frac{m\eps/48(C_1+C_{1/2})}{n+4\kappa\sqrt{n}}\Big\}$.
\end{thm}

For the convex-concave case, we can upper bound the consensus error in the similar way to Lemma \ref{lem:zeta} and obtain the following result.

\begin{thm}\label{thm:cc-constrained}
Suppose Assumption~\ref{asm:smooth}, \ref{asm:cc} and \ref{asm:constrained} hold. We let $\eta = 1/(6\sqrt{n} L)$, 
$T=\left\lceil 12\sqrt{n}L\Norm{\bz_0 - z^*}^2/\eps \right\rceil$,
$K = \big\lceil\sqrt{\chi}\log\left(2\left(\sqrt{m}LD+\delta'\right)/\delta'\right)\big\rceil$ and
\begin{align*}
K_0 = \left\lceil{\sqrt{\chi}}\log\left(\frac{1}{\delta'}\Norm{\vg(\vz_0)-\frac{1}{m}\vone\vone^\top\vg(\vz_0)}^2 \right)\right\rceil    
\end{align*}
for Algorithm \ref{alg:dsvre}, where  
\begin{align*}
\mbox{\small
$\displaystyle{
\delta' = \min\left\{\sqrt{\frac{m\eta\eps}{24(2\eta^2+\beta\eta)L^2}},~ \frac{\sqrt{m}\eta\eps}{4(C_1+C_{1/2})}\right\},  \beta=\frac{2D^2}{\eps}
}$}
\end{align*}
and $C_1$, $C_2$ follow definitions in Lemma \ref{lem:zeta}. We let 
\begin{align*}
(\hat x, \hat y)=\left(\frac{1}{T}\sum_{t=0}^{T-1}\bx_{t+1/2}^\top,~ \frac{1}{T}\sum_{t=0}^{T-1}\by_{t+1/2}^\top\right)
\end{align*}
as the output. Then it requires at most
$\fO\left(n + \sqrt{n}L/\eps\right)$
SFO complexity on each agent in expectation and
$\fO\left(\left(\sqrt{n\chi}L/\eps\right)\log\left(nL/\eps\right)\right)$
communication rounds to obtain $\BE_T\left[f(\hat x,y^*)-f(x^*,\hat y)\right]\leq \eps$.
\end{thm}

\begin{algorithm}[t]
\caption{Multi-Consensus Extragradient (MC-EG)} \label{alg:deg}
\begin{algorithmic}[1]
    \STATE \textbf{Initialize:} $\vz_0=[z_0^\top;\dots;z_0^\top]$ with $z_0\in\fX\times\fY$  \\[0.15cm]
    \STATE $\vs_0=\FM{\vg(\vz_0), K_0}$  \\[0.15cm]
    \STATE \textbf{for} $t = 0, 1, \dots, T-1$ \textbf{do}\\[0.15cm]
    \STATE\quad $\vz_{t+1/2} = \FM{\fP_\FZ\left(\vz_t - \eta \vs_t\right),K}$ \\[0.15cm]
    \STATE\quad $\vs_{t+1/2} = \FM{\vs_{t}+\vg(\vz_{t+1/2})-\vg(\vz_t), K}$ \\[0.15cm]
    \STATE\quad $\vz_{t+1} = \FM{\fP_\FZ\left(\vz_t - \eta\vs_{t+1/2}\right), K}$ \\[0.15cm]
    \STATE\quad $\vs_{t+1} = \FM{\vs_{t}+\vg(\vz_{t+1})-\vg(\vz_{t}), K}$ \\[0.15cm]
    \STATE\textbf{end for} \\[0.15cm]
\end{algorithmic}
\end{algorithm}

\section{The Deterministic Algorithm}\label{sec:MC-EG}

By eliminating all randomness, MC-SVRE (Algorithm \ref{alg:dsvre}) reduces to the deterministic method shown in Algorithm~\ref{alg:deg}. 
We name it as multi-consensus extragradient (MC-EG). 

The analysis of MG-EG can directly follows the one of MC-SVRE with $n=1$. 
Since the variable $\vw_t$ is unnecessary, we only needs to analyze the consensus by the two-dimensional vector
\begin{align*}
\hr_t^2 = \big[\norm{\vz_t-\vone\bz_t}^2, \eta^2\norm{\vs_t-\vone\bs_t}^2\big]^\top.
\end{align*}

Both of the computation and communication complexities of MC-EG (nearly) match the lower bound for decentralized convex-concave minimax optimization~\cite{beznosikov2020distributed}.
The reminder of this section formally present the theoretical results for MC-EG in difference cases.

\subsection{Unconstrained Case}

Following the ideas in Section \ref{sec:mc-svre-unconstrained}, we first establish the recursive relationship for $\Norm{\bz_t-\bz^*}^2$ and $\vone^\top\hr_t$ as follows.

\begin{lem}\label{lem:eg-relation}
Suppose Assumption~\ref{asm:smooth}, \ref{asm:scsc} and \ref{asm:unconstrained} hold. For Algorithm \ref{alg:deg} with $\eta = 1/\left(6L\right)$, we have
{\small\begin{align*}
 \Norm{\bz_{t+1}-\bz^*}^2  
\leq  \left(1- \frac{\mu\eta}{2}\right)\Norm{\bz_t-\bz^*}^2 
-\left(\frac{5}{6}-\frac{5\mu\eta}{2}\right)\Norm{\bz_{t+1/2}-\bz_t}^2  
 - \frac{1}{2}\Norm{\bz_{t+1/2}-\bz_{t+1}}
+ \frac{4L\eta\left(3L\eta+ 2\kappa\right)\rho}{m}\vone^\top\hr_t^2.
\end{align*}}
\end{lem}

\begin{lem}\label{lem:errorb}
Under the settings of Lemma \ref{lem:eg-relation}, we have
\begin{align*}
    \vone^\top\hr_{t+1}^2
\leq  52\rho^2\vone^\top\hr_t^2 +
2\rho^2m\hdelta_t,
\end{align*}
where $\hdelta_t=\Norm{\bz_{t+1} - \bz_{t+1/2}}^2 +  \Norm{\bz_{t+1/2} - \bz_t}^2$.
\end{lem}

Connecting Lemma \ref{lem:eg-relation} and Lemma \ref{lem:errorb}, we have the following convergence result.

\begin{lem}\label{lem:deg-unconstrained}
Under the settings of Lemma \ref{lem:errorb}.
We let $\eta=1/(6L)$ and $K = \big\lceil\sqrt{\chi}/\tilde\delta\big\rceil$ for Algorithm \ref{alg:deg} where
\begin{align*}
\mbox{
$\displaystyle{\tilde\delta=\min\left\{
\frac{1}{2},
\sqrt{\frac{1}{2}\left(\frac{5}{6}-\frac{5\mu\eta}{2}\right)},
\frac{3\left(2-\mu\eta\right)}{2(4\kappa+157)}\right\}}$}.
\end{align*}
Then it holds that
{\begin{align*}
 \Norm{\bz_{t+1}\!-\!\bz^*}^2 + \frac{\vone^\top\hr_{t+1}}{m}    
\leq  \left(1\!-\!\frac{1}{12\kappa}\right)\left(\Norm{\bz_t\!-\!\bz^*}^2 
+ \frac{\vone^\top r_t^2}{m}\right)
\end{align*}} 
\end{lem}

By appropriate setting of $K_0$, we obtain the main result for MC-EG (Algorithm \ref{alg:deg}) in unconstrained strongly-convex-strongly-concave case, which indicates the algorithm has lower communication cost than MC-SVRE.
\begin{thm}\label{thm:eg-unconstrained}
Under the settings of Lemma~\ref{lem:deg-unconstrained}, Algorithm~\ref{alg:deg} with 
\begin{align*}
K_0 = \left\lceil\frac{\sqrt{\chi}}{2}
\log\left(\frac{1}{m\eps}\Norm{\vg(\vz_0)-\frac{1}{m}\vone\vone^\top\vg(\vz_0)}^2\right)\right\rceil \quad\text{and}\quad    
T = \left\lceil12\kappa\log\left(\frac{\Norm{\bz_0 - \bz^*}^2}{\eps}+1\right)\right\rceil
\end{align*}
requires at most 
$\fO\left(\kappa n \log\left(1/\eps\right)\right)$
SFO complexity on each agent and
$\fO\left(\kappa\sqrt{\chi}\log\kappa\log\left(1/\eps\right)\right)$ communication rounds to obtain $\Norm{\bz_T-\bz^*}^2\leq \eps$.
\end{thm}

\subsection{Constrained Case}

The analysis of constrained case follows the one in Section \ref{sec:mc-svre-constrained}.
We present the results for the strongly-convex-strongly-concave case and the convex-concave case respectively in Theorem \ref{thm:eg-scsc-constarined} and \ref{thm:eg-cc-constrained} respectively.

\begin{thm}\label{thm:eg-scsc-constarined}
Suppose Assumption \ref{asm:smooth}, \ref{asm:scsc} and \ref{asm:constrained} hold. We let $\eta = 1/\left(6L\right)$,
$T = \big\lceil 12\kappa\log\big(2\Norm{\bz_0-\bz^*}^2/\eps\big)\big\rceil$,
$K_0 = \big\lceil\sqrt{\chi}\log\big(\Norm{\vg(\vz_0)-\frac{1}{m}\vone\vone^\top\vg(\vz_0)}^2/\delta'\big)\big\rceil$ and
$K = \left\lceil\sqrt{\chi}\log\left(2\left(\sqrt{m}LD+\delta'\right)/\delta'\right)\right\rceil$ for Algorithm \ref{alg:deg}, where
$\hdelta' = \min\big\{\sqrt{m\eps/(16\kappa \left(4\kappa + 1\right))},~\sqrt{m}\eps/(384\kappa\hC_1)\big\}$
then we require at most 
$\fO\left(\kappa n\log\left(1/\eps\right)\right)$
SFO complexity on each agent and
$\fO\left(\kappa\sqrt{\chi}\log\log\left(\kappa /\eps\right)\left(1/\eps\right)\right)$
communication rounds to obtain $\Norm{\bz_T-\bz^*}^2\leq\eps$.
\end{thm}

\begin{thm}\label{thm:eg-cc-constrained}
Suppose Assumption~\ref{asm:smooth}, \ref{asm:cc} and \ref{asm:constrained} hold. We let $\eta = 1/(6L)$, $K = \big\lceil\sqrt{\chi}\log\big(2\big(\sqrt{m}LD+\hdelta'\big)/\hdelta'\big)\big\rceil$, 
$K_0 = \big\lceil\sqrt{\chi}\log\big(\Norm{\vg(\vz_0)-\frac{1}{m}\vone\vone^\top\vg(\vz_0)}^2/\hdelta' \big)\big\rceil$ and
$T = \big\lceil12L\Norm{\bz_0 - \bz^*}^2/\eps\big\rceil$ where
\begin{align*}
\hdelta'=\min\left\{\sqrt{\frac{m\eta\eps/48}{(\eta^2+\beta\eta)L^2}},\frac{\sqrt{m}\eta\eps}{32\hC_1}\right\}, \quad \beta = \frac{2D^2}{\eps}
\end{align*}
and $\hat C_1 = 2D + \frac{1}{6L}\sqrt{2L^2D^2 + \frac{2}{m}\sum_{i=1}^m\Norm{g_i(z^*)}^2}$. 
Then we require 
$\fO\left(nL/\eps\right)$ SFO complexity on each agent and
$\fO\left((L\sqrt{\chi}/\eps)\log\left(LD/\eps\right)\right)$
communication rounds to obtain $f(\hat x,y^*)-f(x^*,\hat y)\leq \eps$, where 
$\hat x=\frac{1}{T}\sum_{t=0}^{T-1}\bx_{t+1/2}^\top$ and
$\hat y=\frac{1}{T}\sum_{t=0}^{T-1}\by_{t+1/2}^\top$.
\end{thm}

\begin{remark}\label{remark:trade-off}
Based on results summarized in Table~\ref{table:scsc-uc}-\ref{table:cc}, we observe that MC-EG is more communication efficient than MC-SVRE, but it has more computational cost. 
This comparison implies we can introduce a trade-off between computation and communication for MC-SVRE in practice. Specifically, we can replace $\vv_{t+1/2}(i)$ by the mini-batch variance reduced estimator as follows
\begin{align*}
g_i(\vw_t(i)) + \frac{1}{l}\sum_{k=1}^l\left(g_{i,j_{i,k}}(\vz_{t+1/2}(i)) - g_{i,j_{i,k}}(\vw_t(i))\right),
\end{align*}
where $\ell$ is the mini-batch size and each $j_{i,k}$ is sampled from $\{1,\dots,n\}$ uniformly. The probability of computing the full gradient should be scaled into $p=l/(2n)$ for such strategy.
\end{remark}

\begin{figure}[t]\centering
\begin{tabular}{cccc}
\includegraphics[scale=0.29]{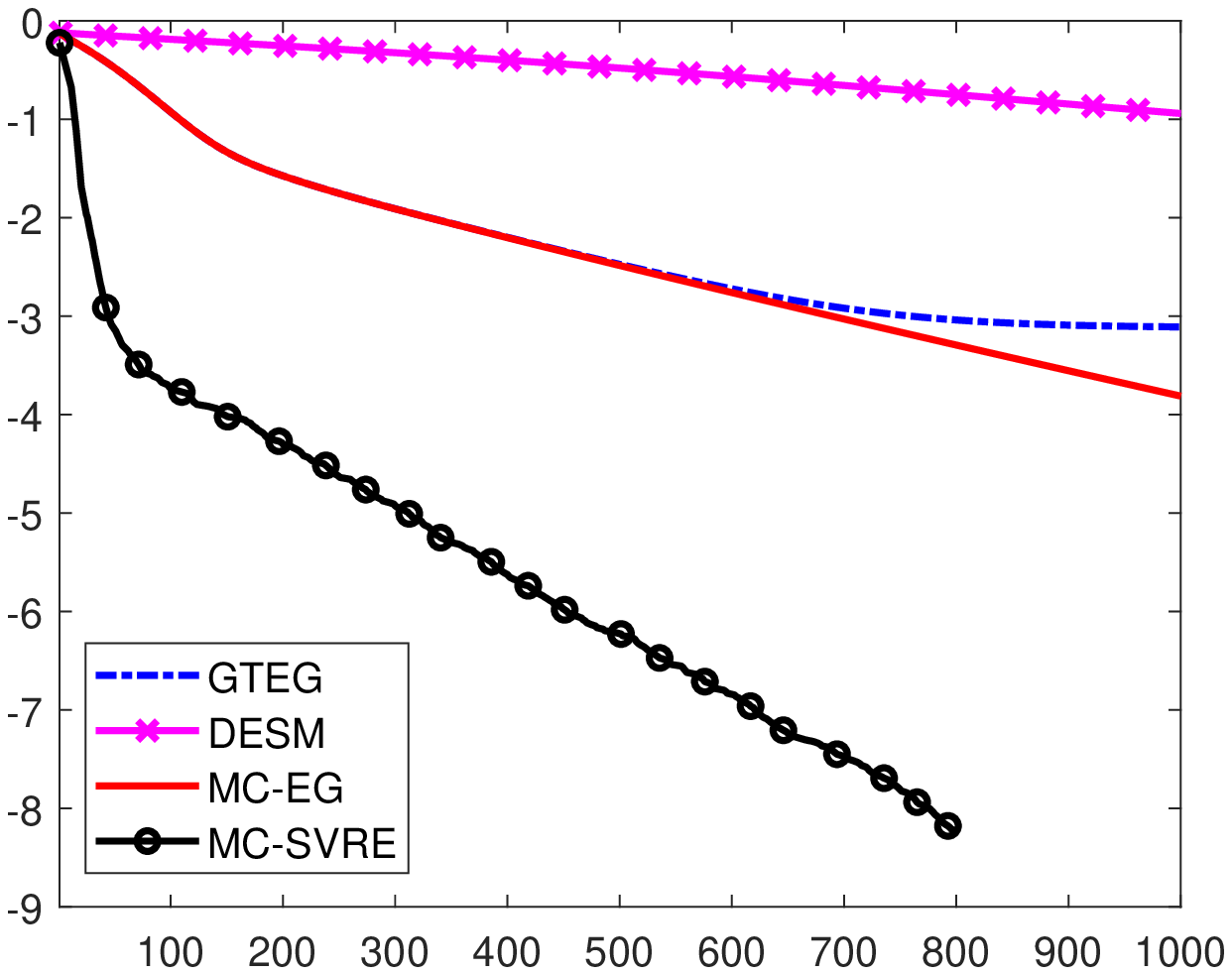} &
\includegraphics[scale=0.29]{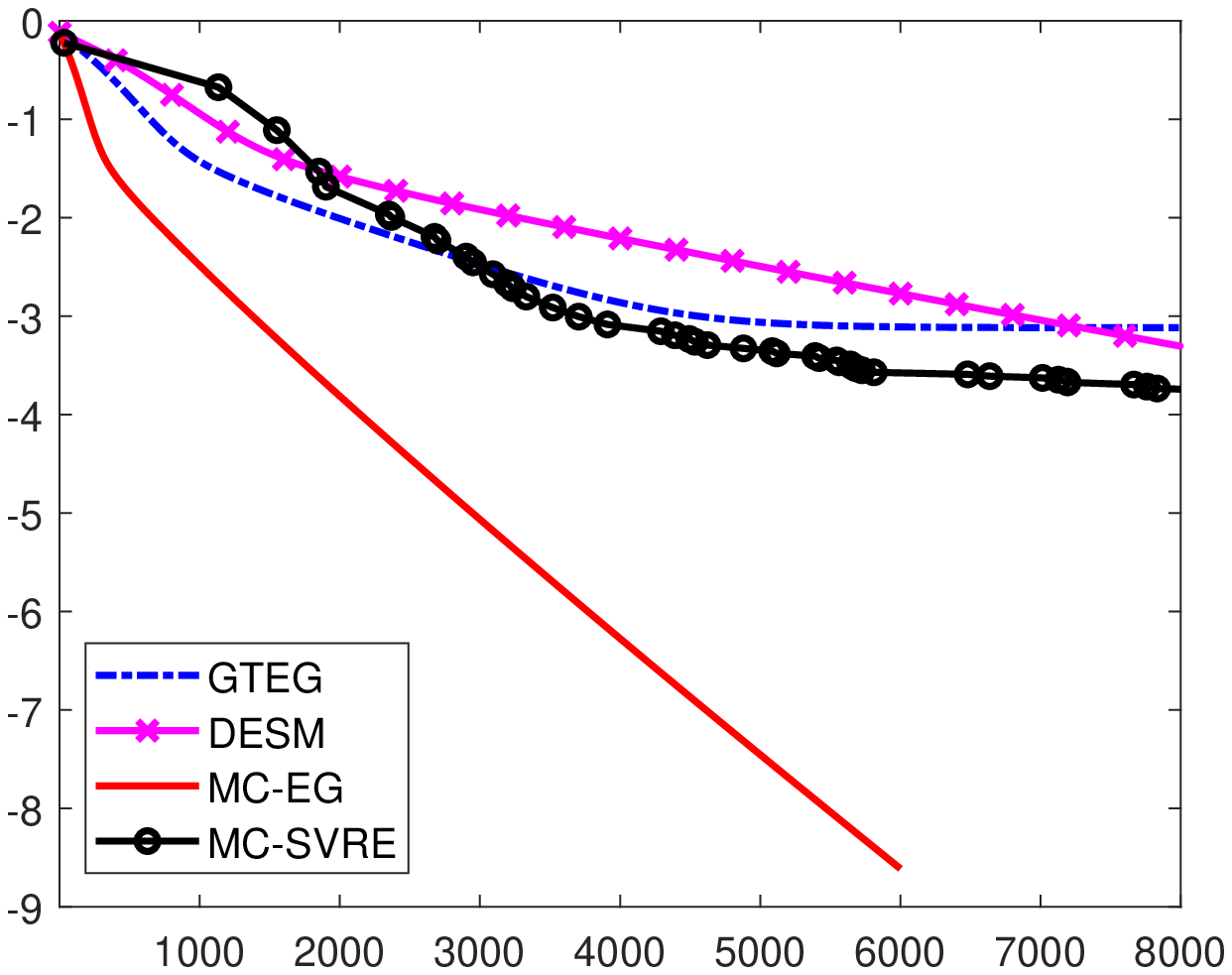} &
\includegraphics[scale=0.29]{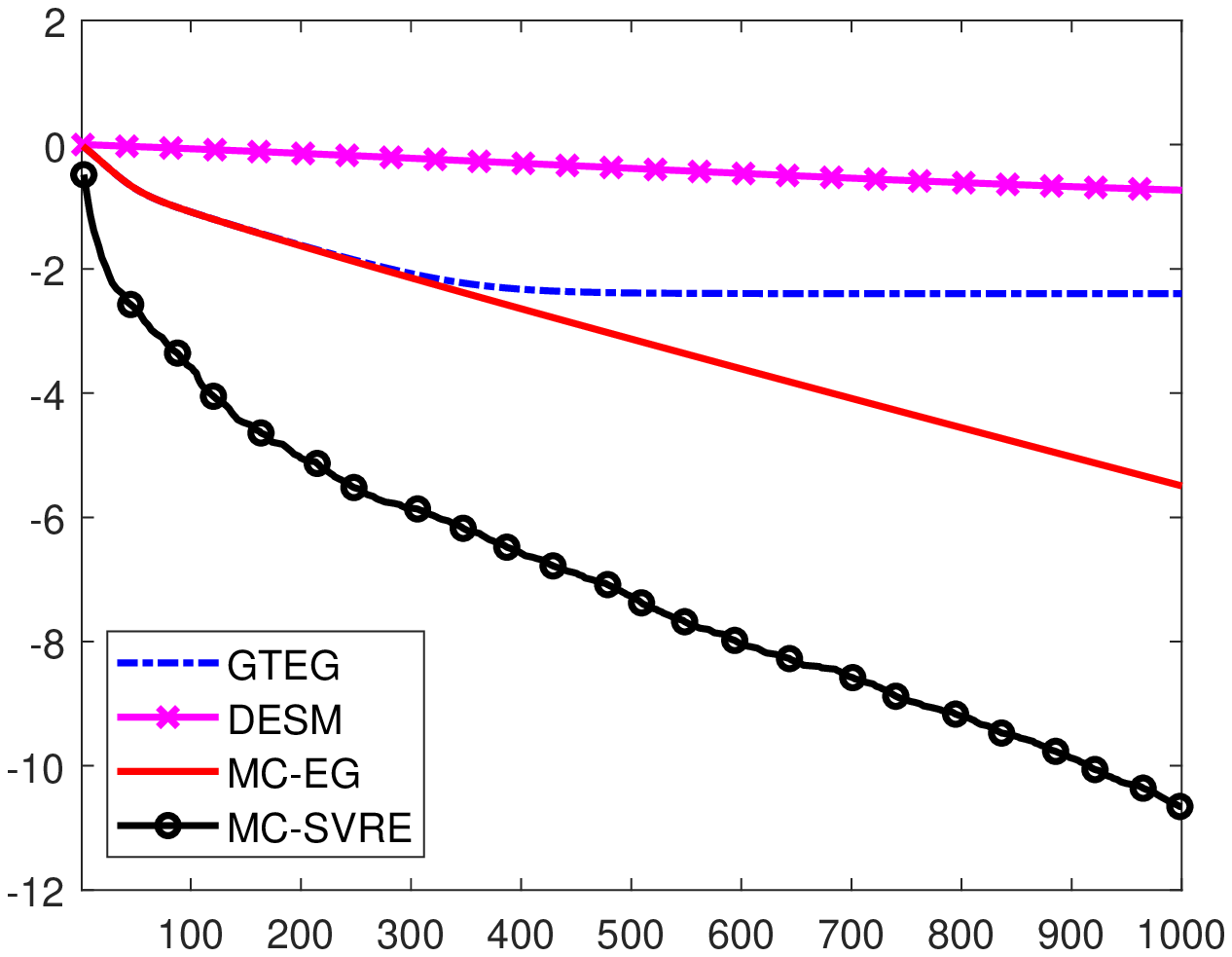} &
\includegraphics[scale=0.29]{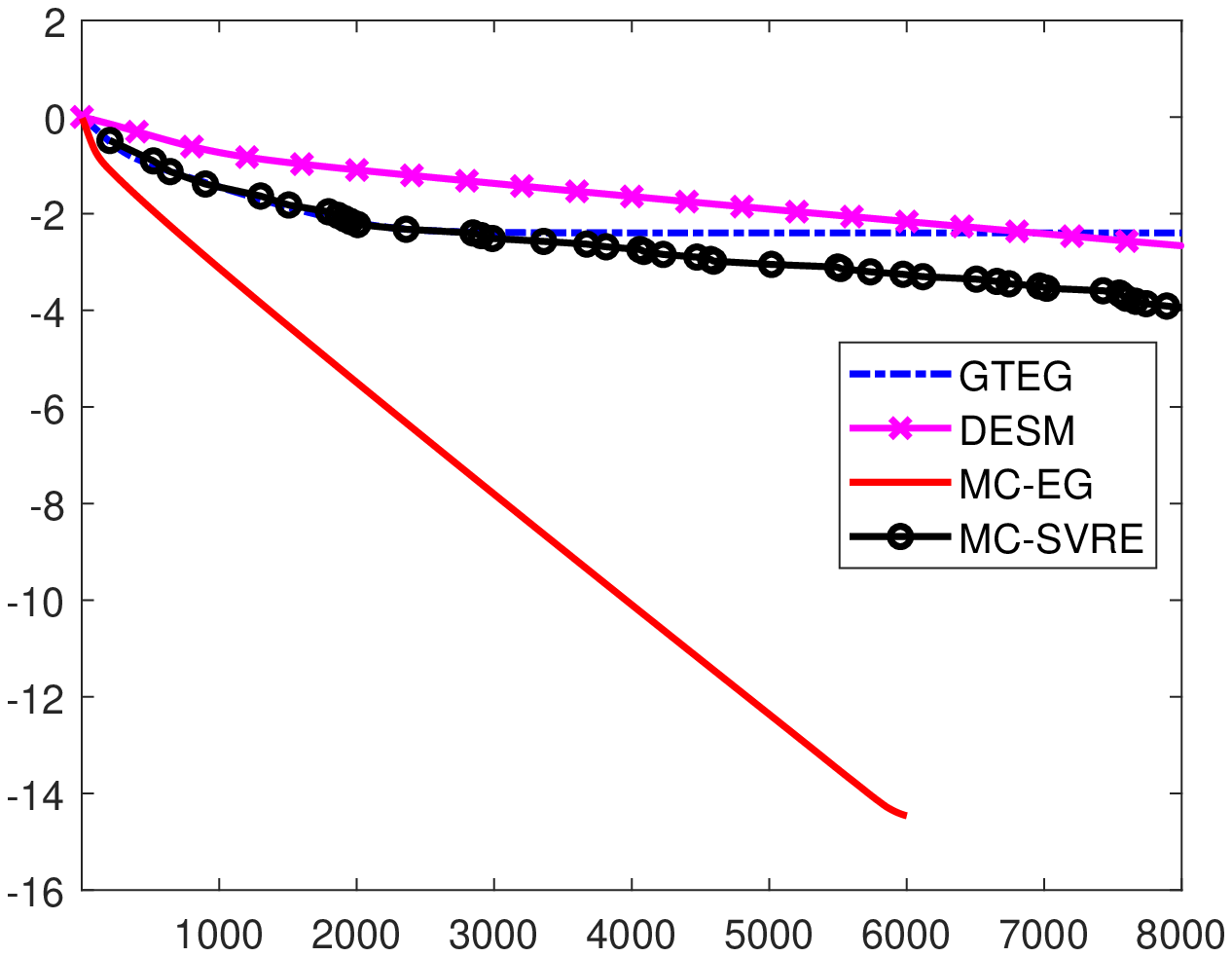} \\
\footnotesize (a) a9a, epoch vs $\log\Norm{g(z)}$ & 
\footnotesize (b) a9a, comm vs $\log\Norm{g(z)}$ &
\footnotesize (c) w8a, epoch vs $\log\Norm{g(z)}$ & 
\footnotesize (d) w8a, comm vs $\log\Norm{g(z)}$ \\[-0.1cm]
\end{tabular}
\caption{We show the results of epochs vs. $\log\Norm{g(z)}$ and rounds of communication vs. $\log\Norm{g(z)}$ for AUC maximization on datasets ``a9a''  and ``w8a''.}\label{figure:auc}
\end{figure}

\section{Numerical Experiments}\label{sec:experiments}
 
In this section, we provide numerical experiments on the applications of AUC maximization and distributionally robust optimization (DRO). We evaluation the performance by the norm of gradient operator $g(z)$ and the norm of gradient mapping $h(z)=\Norm{z-\fP_\fZ(z - \tau g(z))}/\tau$ for unconstrained and constrained problems respectively, where we set $\tau = 0.5$.
All experiments are conducted in a random graph of ten nodes and the $10 \times 10$ gossip matrix related to this graph satisfies Assumption~\ref{asm:W}.
Furthermore, our experiments of AUC maximization and distributionally robust optimization are conducted on two datasets:  ``a9a'' and ``w8a''. 
The samples of dataset are uniformly distributed into the ten agents.
We compare the proposed MC-SVRE (Algorithm \ref{alg:dsvre}) and MC-EG (Algorithm \ref{alg:deg}) with existing algorithms GTEG \cite{mukherjee2020decentralized} and DESM \cite{beznosikov2020distributed}.
To achieve the fair comparison, we tune the parameters for all algorithms properly.

\subsection{AUC Maximization}

AUC maximization~\cite{hanley1982meaning,ying2016stochastic} is targeted to find the binary classifier $\theta\in\BR^d$ on training set $\{(a_{i,j}, b_{i,j})\}_{i,j}$ with $mn$ samples, where $a_{i,j}\in\BR^d$, $b_{i,j}\in\{+1,-1\}$, $i=1,\dots,m$ and $j=1,\dots,n$. 
We denote $N^+$ be the numbers of positive and negative instances and let $q=N^+/(mn)$. The unconstrained minimax formulation for this model is
\begin{align*}
    \min_{x\in\BR^{d+2}}\max_{y\in\BR} f(x,y)   \triangleq \frac{1}{mn}\sum_{i=1}^m\sum_{j=1}^n f_{i,j}(x,y;a_{i,j},b_{i,j},\lambda),
\end{align*}
where $x=[\theta;u;v]\in\BR^{d+2}$, $\lambda$ is the regularization parameter and each component function $f_{i,j}(x,y;a_{i,j},b_{i,j},\lambda)$ is defined as
\begin{align*}
f_{i,j}(x,y;a_{i,j},b_{i,j},\lambda)  
\triangleq & \frac{\lambda}{2}\norm{x}^2 -q(1-q)y^2    + (1-q)\left((\theta^\top a_{i,j}-u)^2 - 2(1+y)\theta^\top a_{i,j} \right)\BI_{\{b_{i,j}=1\}} \\
&  + q\left((\theta^\top a_{i,j}-v)^2 +  2(1+y)\theta^\top a_{i,j}\right)\BI_{\{b_{i,j}=-1\}}. \end{align*}
We set $\lambda=0.01$ for our experiment.

\subsection{Distributionally Robust Optimization}

We consider the distributionally robust optimization with logistic loss and $\ell_1$-ball constraint~\cite{duchi2019variance,yan2019stochastic}. Given a training set $\left\{(a_{i,j}, b_{i,j})\right\}_{i,j}$ with $mn$ samples where $a_{i,j}\in\BR^d$, $b_{i,j}\in\{1,-1\}$, $i=1,\dots,m$ and $j=1,\dots,n$. 
The constrained minimax formulation for this model is
\begin{align*}
    \min_{x\in\fX} \max_{y\in{\fY}} f(x,y)\triangleq \frac{1}{mn}\sum_{i=1}^m\sum_{j=1}^n f_{i,j}(x,y),
\end{align*}
where 
$f_{i,j}(x,y)=y_{i,j} l_{i,j}(x) + \frac{\lambda_2}{2}\Norm{x}^2 -V(y)$,
$l_{i,j}(x)=\log\left(1+\exp(-b_{i,j} a_{i,j}^\top x)\right)$ and
$V(y)=\frac{\lambda_3}{2}\norm{mny\!-\!\bf 1}^2$.
The constrained sets are $\fX=\{x\in\BR^d: \Norm{x}_1\leq \lambda_1\}$ and
$\fY=\{y\in{\mathbb R}^{mn}: 0 \leq y_i \leq 1, \sum_{i=1}^{mn} y_i=1\}$. We set $\lambda_1=1$, $\lambda_2=0.1$ and $\lambda_3=1/n^2$ for our experiment. 

\begin{figure}[t]\centering
\begin{tabular}{cccc}
\includegraphics[scale=0.288]{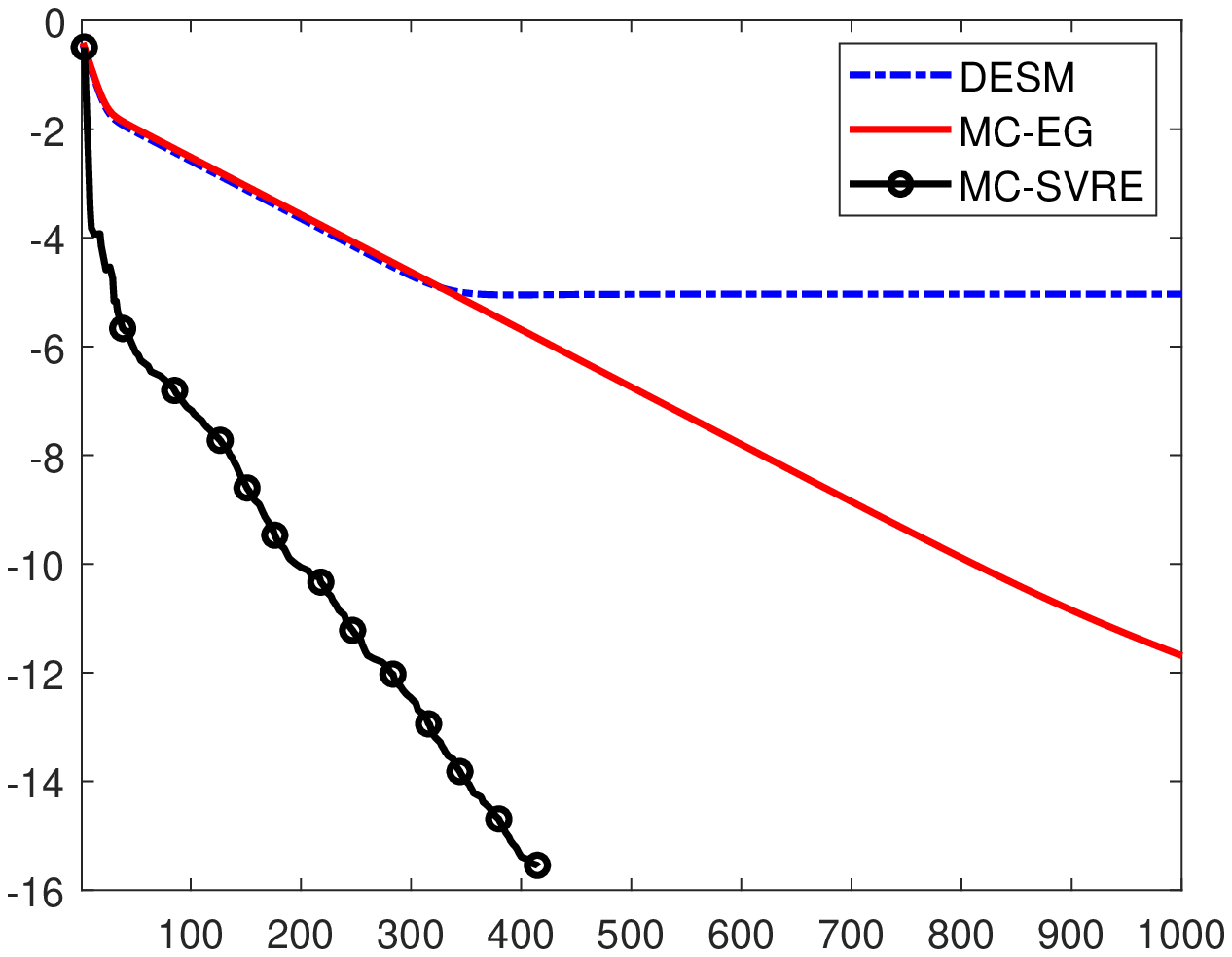} &
\includegraphics[scale=0.288]{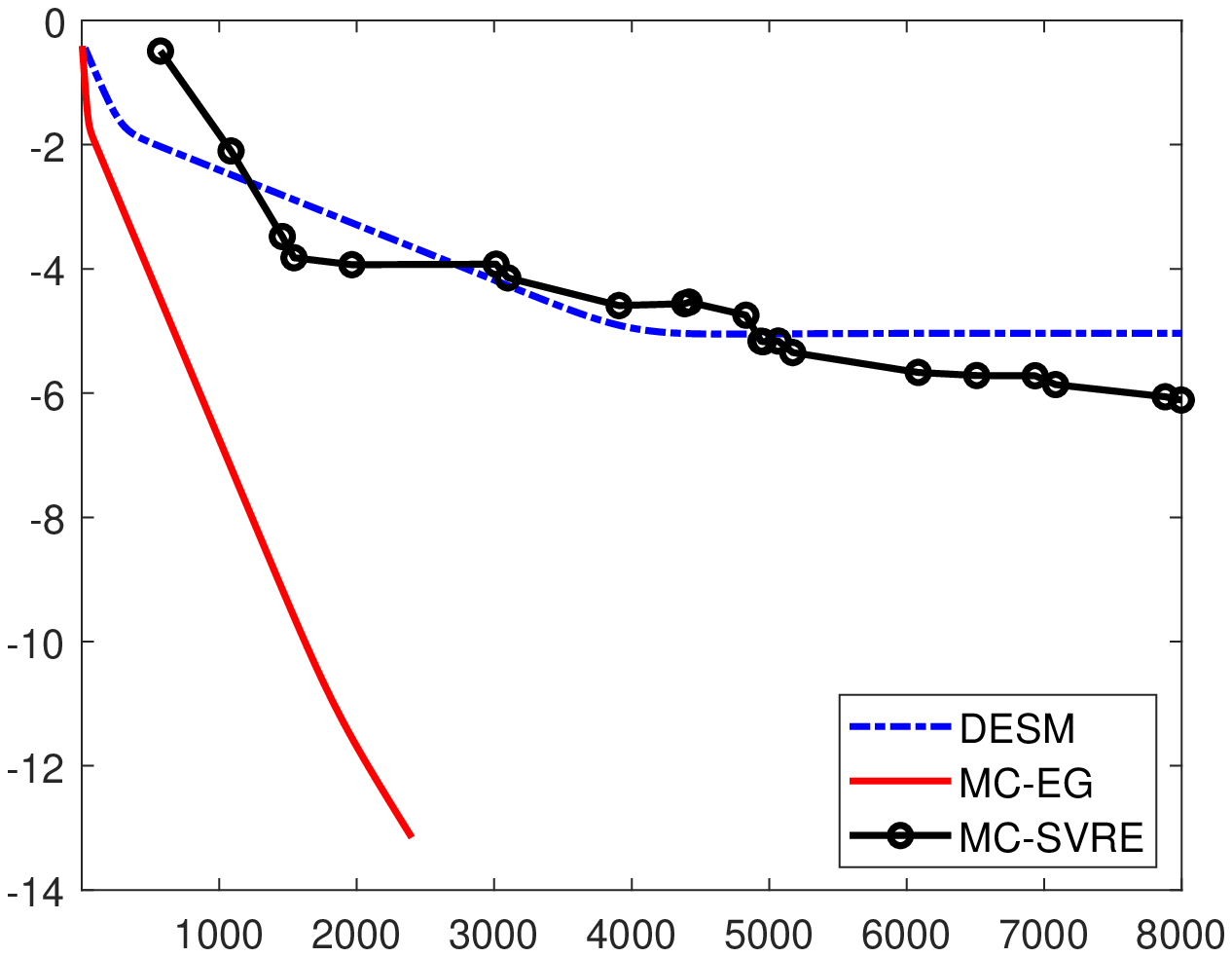} &
\includegraphics[scale=0.288]{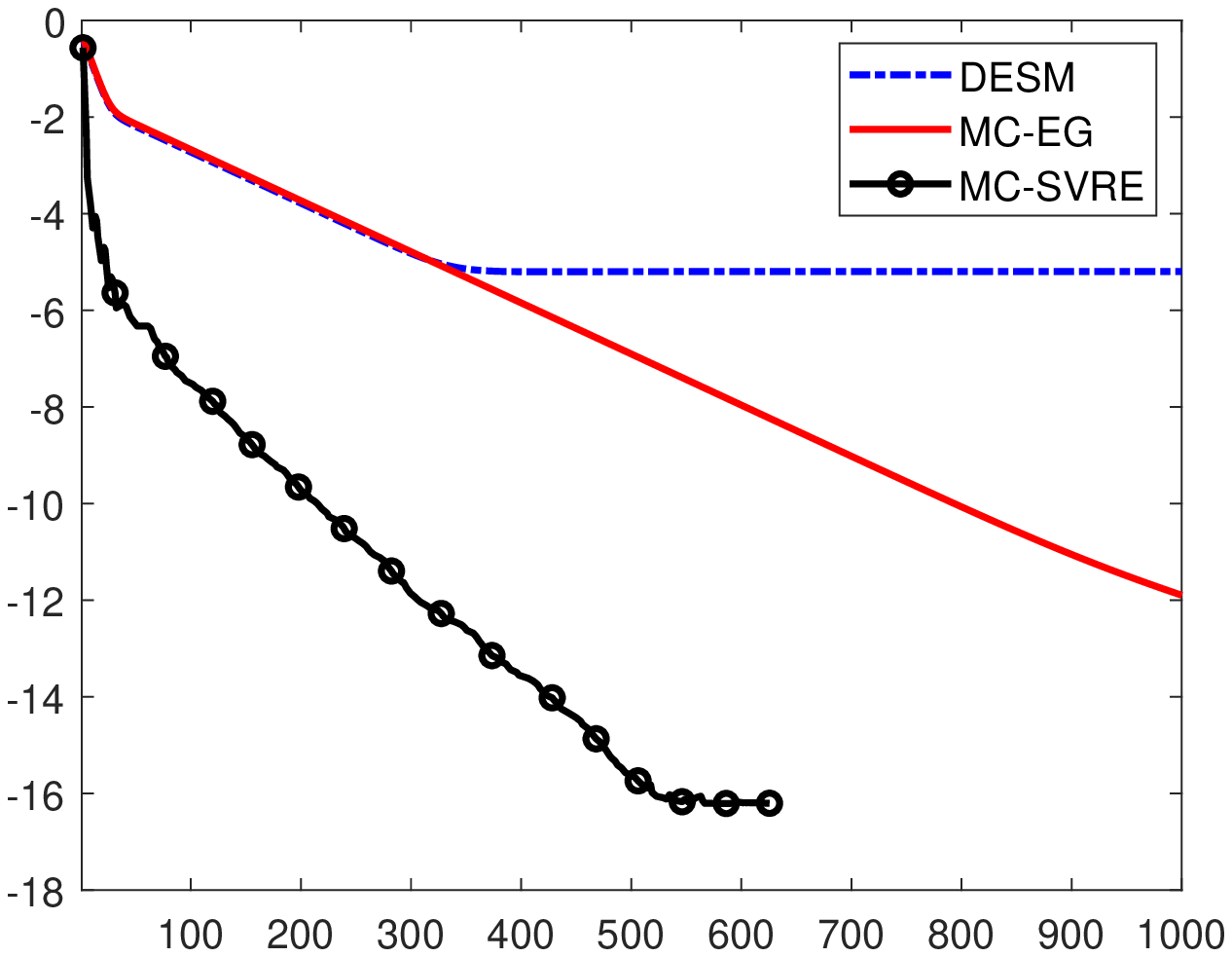} &
\includegraphics[scale=0.288]{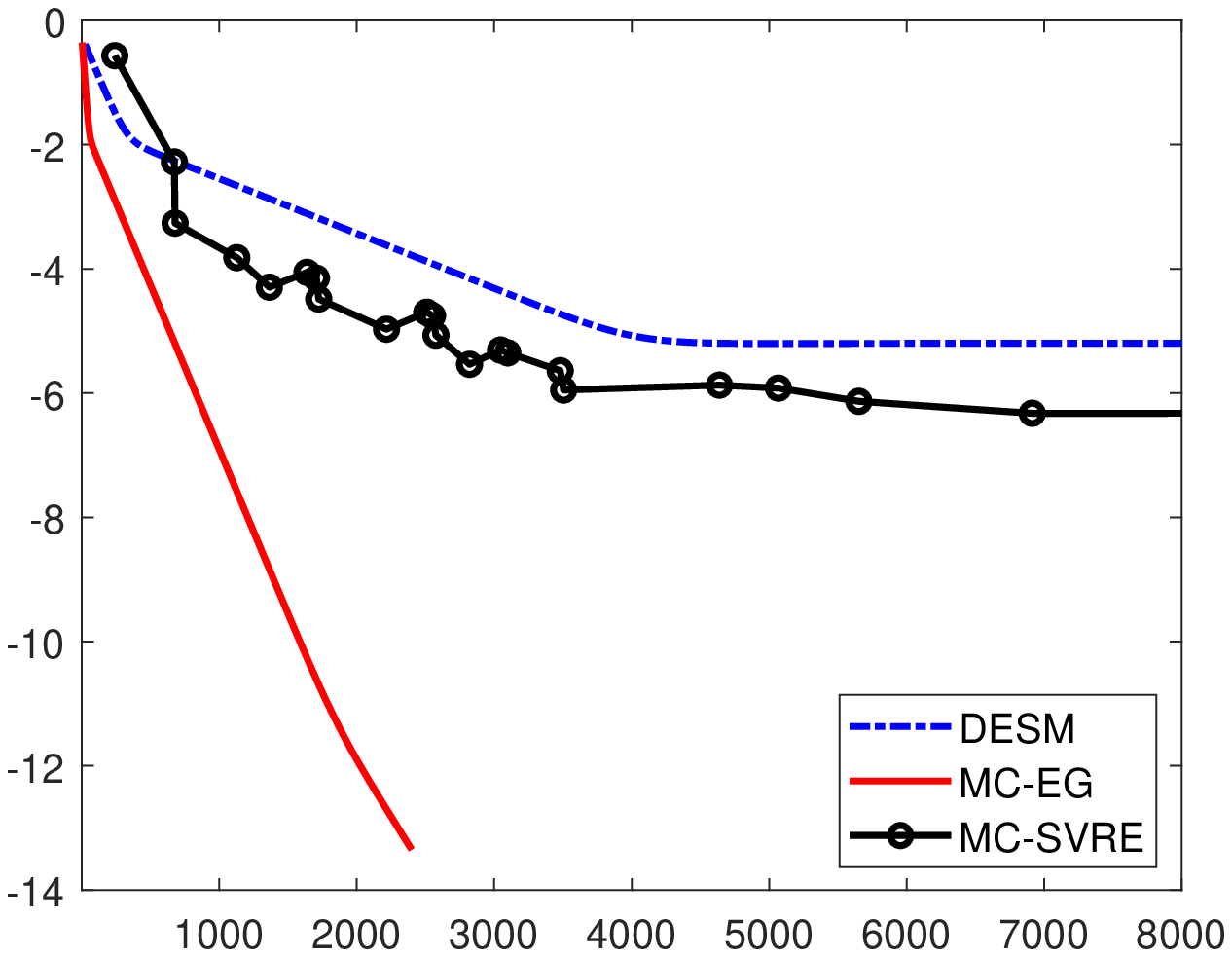} \\
\footnotesize (a) a9a, epoch vs $\log\Norm{h(z)}$ & 
\footnotesize (b) a9a, comm vs $\log\Norm{h(z)}$ &
\footnotesize (c) w8a, epoch vs $\log\Norm{h(z)}$ & 
\footnotesize (d) w8a, comm vs $\log\Norm{h(z)}$ \\[-0.1cm]
\end{tabular}
\caption{We show the results of epochs vs. $\log\Norm{h(z)}$ and rounds of communication vs. $\log\Norm{h(z)}$ for DRO problem on datasets ``a9a''  and ``w8a''.}\label{figure:dro}
\end{figure}

\subsection{Experimental Results and Discussion}

We report our experimental results in Figure \ref{figure:auc} and \ref{figure:dro}.
We observe that MC-EG outperforms GTEG since GTEG lacks the multi-consensus step and each agent only communicates with its neighbors only once for per update of the variable.
At the same time, the results show that MC-EG is both more communication efficient and more computation efficient than DESM.
This is because of DESM does not include the gradient-tracking step, leading to the algorithm requires more communication rounds when it tries to obtain a high precision solution.
On the other hand, the comparisons show that MC-SVRE achieves the lowest computation cost among all algorithms, but it commonly requires more communication cost than MC-EG, which implies MC-SVRE is suitable for computation sensitive cases. All of these empirical results validates our theoretical analysis. 

\section{Conclusion}\label{conclusion}

In this paper, we have proposed a variance reduced extragradient method for decentralized convex-concave minimax optimization.  
We prove the algorithm achieves best known SFO complexity in theoretical. 
We also provide a deterministic variant for this method, which is more communication efficient. 
The numerical experiments on machine learning applications of AUC maximization and distributionally robust optimization validate the effectiveness of proposed algorithms.
It would be interesting to extend our algorithms to more general case, such as nonconvex-concave or nonconvex-nonconcave minimax optimization problems.

\bibliographystyle{plainnat}
\bibliography{reference}

\begin{thebibliography}{53}
\providecommand{\natexlab}[1]{#1}
\providecommand{\url}[1]{\texttt{#1}}
\expandafter\ifx\csname urlstyle\endcsname\relax
  \providecommand{\doi}[1]{doi: #1}\else
  \providecommand{\doi}{doi: \begingroup \urlstyle{rm}\Url}\fi

\bibitem[Alacaoglu and Malitsky(2021)]{alacaoglu2021stochastic}
Ahmet Alacaoglu and Yura Malitsky.
\newblock Stochastic variance reduction for variational inequality methods.
\newblock \emph{arXiv preprint arXiv:2102.08352}, 2021.

\bibitem[Allen-Zhu(2017)]{allen2017katyusha}
Zeyuan Allen-Zhu.
\newblock Katyusha: The first direct acceleration of stochastic gradient
  methods.
\newblock \emph{Journal of Machine Learning Research}, 18\penalty0
  (1):\penalty0 8194--8244, 2017.

\bibitem[Ba{\c{s}}ar and Olsder(1998)]{bacsar1998dynamic}
Tamer Ba{\c{s}}ar and Geert~Jan Olsder.
\newblock \emph{Dynamic noncooperative game theory}.
\newblock SIAM, 1998.

\bibitem[Ben-Tal et~al.(2009)Ben-Tal, El~Ghaoui, and Nemirovski]{ben2009robust}
Aharon Ben-Tal, Laurent El~Ghaoui, and Arkadi Nemirovski.
\newblock \emph{Robust optimization}.
\newblock Princeton university press, 2009.

\bibitem[Beznosikov et~al.(2020)Beznosikov, Samokhin, and
  Gasnikov]{beznosikov2020distributed}
Aleksandr Beznosikov, Valentin Samokhin, and Alexander Gasnikov.
\newblock Distributed saddle-point problems: Lower bounds, optimal and robust
  algorithms.
\newblock \emph{arXiv preprint arXiv:2010.13112}, 2020.

\bibitem[Chavdarova et~al.(2019)Chavdarova, Gidel, Fleuret, and
  Lacoste-Julien]{chavdarova2019reducing}
Tatjana Chavdarova, Gauthier Gidel, Fran{\c{c}}ois Fleuret, and Simon
  Lacoste-Julien.
\newblock Reducing noise in {GAN} training with variance reduced extragradient.
\newblock In \emph{NeurIPS}, 2019.

\bibitem[Defazio et~al.(2014)Defazio, Bach, and
  Lacoste-Julien]{defazio2014saga}
Aaron Defazio, Francis Bach, and Simon Lacoste-Julien.
\newblock {SAGA}: A fast incremental gradient method with support for
  non-strongly convex composite objectives.
\newblock In \emph{NIPS}, 2014.

\bibitem[Du et~al.(2017)Du, Chen, Li, Xiao, and Zhou]{du2017stochastic}
Simon~S. Du, Jianshu Chen, Lihong Li, Lin Xiao, and Dengyong Zhou.
\newblock Stochastic variance reduction methods for policy evaluation.
\newblock In \emph{ICML}, 2017.

\bibitem[Duchi and Namkoong(2019)]{duchi2019variance}
John~C. Duchi and Hongseok Namkoong.
\newblock Variance-based regularization with convex objectives.
\newblock \emph{Journal of Machine Learning Research}, 20\penalty0
  (68):\penalty0 1--55, 2019.

\bibitem[Fang et~al.(2018)Fang, Li, Lin, and Zhang]{fang2018spider}
Cong Fang, Chris~Junchi Li, Zhouchen Lin, and Tong Zhang.
\newblock {SPIDER}: Near-optimal non-convex optimization via stochastic
  path-integrated differential estimator.
\newblock In \emph{NeurIPS}, 2018.

\bibitem[Gao and Kleywegt(2016)]{gao2016distributionally}
Rui Gao and Anton~J. Kleywegt.
\newblock Distributionally robust stochastic optimization with wasserstein
  distance.
\newblock \emph{arXiv preprint arXiv:1604.02199}, 2016.

\bibitem[Hanley and McNeil(1982)]{hanley1982meaning}
James~A Hanley and Barbara~J McNeil.
\newblock The meaning and use of the area under a receiver operating
  characteristic ({ROC}) curve.
\newblock \emph{Radiology}, 143\penalty0 (1):\penalty0 29--36, 1982.

\bibitem[Hast et~al.(2013)Hast, {\AA}str{\"o}m, Bernhardsson, and
  Boyd]{hast2013pid}
Martin Hast, Karl~Johan {\AA}str{\"o}m, Bo~Bernhardsson, and Stephen Boyd.
\newblock Pid design by convex-concave optimization.
\newblock In \emph{ECC}, 2013.

\bibitem[Hendrikx et~al.(2020)Hendrikx, Bach, and
  Massouli{\'e}]{hendrikx2020dual}
Hadrien Hendrikx, Francis Bach, and Laurent Massouli{\'e}.
\newblock Dual-free stochastic decentralized optimization with variance
  reduction.
\newblock In \emph{NeurIPS}, 2020.

\bibitem[Hendrikx et~al.(2021)Hendrikx, Bach, and
  Massoulie]{hendrikx2021optimal}
Hadrien Hendrikx, Francis Bach, and Laurent Massoulie.
\newblock An optimal algorithm for decentralized finite-sum optimization.
\newblock \emph{SIAM Journal on Optimization}, 31\penalty0 (4):\penalty0
  2753--2783, 2021.

\bibitem[Hofmann et~al.(2015)Hofmann, Lucchi, Lacoste-Julien, and
  McWilliams]{hofmann2015variance}
Thomas Hofmann, Aurelien Lucchi, Simon Lacoste-Julien, and Brian McWilliams.
\newblock Variance reduced stochastic gradient descent with neighbors.
\newblock In \emph{NIPS}, 2015.

\bibitem[Johnson and Zhang(2013)]{johnson2013accelerating}
Rie Johnson and Tong Zhang.
\newblock Accelerating stochastic gradient descent using predictive variance
  reduction.
\newblock In \emph{NIPS}, 2013.

\bibitem[Korpelevich(1977)]{korpelevich1977extragradient}
GM~Korpelevich.
\newblock Extragradient method for finding saddle points and other problems.
\newblock \emph{Matekon}, 13\penalty0 (4):\penalty0 35--49, 1977.

\bibitem[Kovalev et~al.(2020)Kovalev, Horv{\'a}th, and
  Richt{\'a}rik]{kovalev2020don}
Dmitry Kovalev, Samuel Horv{\'a}th, and Peter Richt{\'a}rik.
\newblock Don’t jump through hoops and remove those loops: {SVRG} and
  {K}atyusha are better without the outer loop.
\newblock In \emph{ALT}, 2020.

\bibitem[Li et~al.(2020)Li, Lin, and Fang]{li2020variance}
Huan Li, Zhouchen Lin, and Yongchun Fang.
\newblock Variance reduced {EXTRA} and {DIGing} and their optimal acceleration
  for strongly convex decentralized optimization.
\newblock \emph{arXiv preprint arXiv:2009.04373}, 2020.

\bibitem[Lian et~al.(2017)Lian, Zhang, Zhang, Hsieh, Zhang, and
  Liu]{lian2017can}
Xiangru Lian, Ce~Zhang, Huan Zhang, Cho-Jui Hsieh, Wei Zhang, and Ji~Liu.
\newblock Can decentralized algorithms outperform centralized algorithms? a
  case study for decentralized parallel stochastic gradient descent.
\newblock \emph{arXiv preprint arXiv:1705.09056}, 2017.

\bibitem[Liu and Morse(2011)]{liu2011accelerated}
Ji~Liu and A~Stephen Morse.
\newblock Accelerated linear iterations for distributed averaging.
\newblock \emph{Annual Reviews in Control}, 35\penalty0 (2):\penalty0 160--165,
  2011.

\bibitem[Liu et~al.(2020)Liu, Mroueh, Zhang, Cui, Ross, and
  Das]{liu2020decentralized}
Mingrui~Liu Liu, Youssef Mroueh, Wei Zhang, Xiaodong Cui, Jerret Ross, and
  Payel Das.
\newblock Decentralized parallel algorithm for training generative adversarial
  nets.
\newblock In \emph{NeurIPS}, 2020.

\bibitem[Luo et~al.(2019)Luo, Chen, Li, Xie, and Zhang]{luo2019stochastic}
Luo Luo, Cheng Chen, Yujun Li, Guangzeng Xie, and Zhihua Zhang.
\newblock A stochastic proximal point algorithm for saddle-point problems.
\newblock \emph{arXiv preprint arXiv:1909.06946}, 2019.

\bibitem[Luo et~al.(2020)Luo, Ye, Huang, and Zhang]{luo2020stochastic}
Luo Luo, Haishan Ye, Zhichao Huang, and Tong Zhang.
\newblock Stochastic recursive gradient descent ascent for stochastic
  nonconvex-strongly-concave minimax problems.
\newblock In \emph{NeurIPS}, 2020.

\bibitem[Luo et~al.(2021)Luo, Xie, Zhang, and Zhang]{luo2021near}
Luo Luo, Guangzeng Xie, Tong Zhang, and Zhihua Zhang.
\newblock Near optimal stochastic algorithms for finite-sum unbalanced
  convex-concave minimax optimization.
\newblock \emph{arXiv preprint arXiv:2106.01761}, 2021.

\bibitem[Mokhtari et~al.(2020)Mokhtari, Ozdaglar, and
  Pattathil]{mokhtari2020unified}
Aryan Mokhtari, Asuman Ozdaglar, and Sarath Pattathil.
\newblock A unified analysis of extra-gradient and optimistic gradient methods
  for saddle point problems: Proximal point approach.
\newblock In \emph{AISTATS}, 2020.

\bibitem[Mukherjee and Chakraborty(2020)]{mukherjee2020decentralized}
Soham Mukherjee and Mrityunjoy Chakraborty.
\newblock A decentralized algorithm for large scale min-max problems.
\newblock In \emph{CDC}, 2020.

\bibitem[Nedic et~al.(2017)Nedic, Olshevsky, and Shi]{nedic2017achieving}
Angelia Nedic, Alex Olshevsky, and Wei Shi.
\newblock Achieving geometric convergence for distributed optimization over
  time-varying graphs.
\newblock \emph{SIAM Journal on Optimization}, 27\penalty0 (4):\penalty0
  2597--2633, 2017.

\bibitem[Palaniappan and Bach(2016)]{palaniappan2016stochastic}
Balamurugan Palaniappan and Francis Bach.
\newblock Stochastic variance reduction methods for saddle-point problems.
\newblock In \emph{NIPS}, 2016.

\bibitem[Pan et~al.(2020)Pan, Liu, and Wang]{pan2020d}
Taoxing Pan, Jun Liu, and Jie Wang.
\newblock {D-SPIDER-SFO}: A decentralized optimization algorithm with faster
  convergence rate for nonconvex problems.
\newblock In \emph{AAAI}, 2020.

\bibitem[Pu and Nedi{\'c}(2021)]{pu2021distributed}
Shi Pu and Angelia Nedi{\'c}.
\newblock Distributed stochastic gradient tracking methods.
\newblock \emph{Mathematical Programming}, 187\penalty0 (1):\penalty0 409--457,
  2021.

\bibitem[Qu and Li(2017)]{qu2017harnessing}
Guannan Qu and Na~Li.
\newblock Harnessing smoothness to accelerate distributed optimization.
\newblock \emph{IEEE Transactions on Control of Network Systems}, 5\penalty0
  (3):\penalty0 1245--1260, 2017.

\bibitem[Qu and Li(2019)]{qu2019accelerated}
Guannan Qu and Na~Li.
\newblock Accelerated distributed nesterov gradient descent.
\newblock \emph{IEEE Transactions on Automatic Control}, 65\penalty0
  (6):\penalty0 2566--2581, 2019.

\bibitem[Rockafellar(1970)]{rockafellar1970monotone}
R~Tyrrell Rockafellar.
\newblock Monotone operators associated with saddle-functions and minimax
  problems.
\newblock \emph{Nonlinear functional analysis}, 18\penalty0 (part 1):\penalty0
  397--407, 1970.

\bibitem[Schmidt et~al.(2017)Schmidt, Le~Roux, and Bach]{schmidt2017minimizing}
Mark Schmidt, Nicolas Le~Roux, and Francis Bach.
\newblock Minimizing finite sums with the stochastic average gradient.
\newblock \emph{Mathematical Programming}, 162\penalty0 (1-2):\penalty0
  83--112, 2017.

\bibitem[Shafieezadeh-Abadeh et~al.(2015)Shafieezadeh-Abadeh, Esfahani, and
  Kuhn]{shafieezadeh2015distributionally}
Soroosh Shafieezadeh-Abadeh, Peyman~Mohajerin Esfahani, and Daniel Kuhn.
\newblock Distributionally robust logistic regression.
\newblock \emph{arXiv preprint arXiv:1509.09259}, 2015.

\bibitem[Tang et~al.(2018)Tang, Lian, Yan, Zhang, and Liu]{tang2018d}
Hanlin Tang, Xiangru Lian, Ming Yan, Ce~Zhang, and Ji~Liu.
\newblock {$D^2$}: Decentralized training over decentralized data.
\newblock In \emph{ICML}, 2018.

\bibitem[Tseng(1995)]{tseng1995linear}
Paul Tseng.
\newblock On linear convergence of iterative methods for the variational
  inequality problem.
\newblock \emph{Journal of Computational and Applied Mathematics}, 60\penalty0
  (1-2):\penalty0 237--252, 1995.

\bibitem[Von~Neumann and Morgenstern(2007)]{von2007theory}
John Von~Neumann and Oskar Morgenstern.
\newblock \emph{Theory of games and economic behavior}.
\newblock Princeton university press, 2007.

\bibitem[Wang and Xiao(2017)]{wang2017exploiting}
Jialei Wang and Lin Xiao.
\newblock Exploiting strong convexity from data with primal-dual first-order
  algorithms.
\newblock In \emph{ICML}, 2017.

\bibitem[Xian et~al.(2021)Xian, Huang, Zhang, and Huang]{xian2021faster}
Wenhan Xian, Feihu Huang, Yanfu Zhang, and Heng Huang.
\newblock A faster decentralized algorithm for nonconvex minimax problems.
\newblock \emph{NeurIPS}, 2021.

\bibitem[Xin et~al.(2020)Xin, Kar, and Khan]{xin2020decentralized}
Ran Xin, Soummya Kar, and Usman~A Khan.
\newblock Decentralized stochastic optimization and machine learning: A unified
  variance-reduction framework for robust performance and fast convergence.
\newblock \emph{IEEE Signal Processing Magazine}, 37\penalty0 (3):\penalty0
  102--113, 2020.

\bibitem[Yan et~al.(2019)Yan, Xu, Lin, Zhang, and Yang]{yan2019stochastic}
Yan Yan, Yi~Xu, Qihang Lin, Lijun Zhang, and Tianbao Yang.
\newblock Stochastic primal-dual algorithms with faster convergence than ${O}
  (1/\sqrt t)$ for problems without bilinear structure.
\newblock \emph{arXiv preprint arXiv:1904.10112}, 2019.

\bibitem[Yang et~al.(2020)Yang, Kiyavash, and He]{yang2020global}
Junchi Yang, Negar Kiyavash, and Niao He.
\newblock Global convergence and variance-reduced optimization for a class of
  nonconvex-nonconcave minimax problems.
\newblock In \emph{NIPS}, 2020.

\bibitem[Ye et~al.(2020{\natexlab{a}})Ye, Luo, Zhou, and Zhang]{ye2020multi}
Haishan Ye, Luo Luo, Ziang Zhou, and Tong Zhang.
\newblock Multi-consensus decentralized accelerated gradient descent.
\newblock \emph{arXiv preprint arXiv:2005.00797}, 2020{\natexlab{a}}.

\bibitem[Ye et~al.(2020{\natexlab{b}})Ye, Wei, and Zhang]{ye2020pmgt}
Haishan Ye, Xiong Wei, and Tong Zhang.
\newblock {PMGT-VR}: A decentralized proximal-gradient algorithmic framework
  with variance reduction.
\newblock \emph{arXiv preprint arXiv:2012.15010}, 2020{\natexlab{b}}.

\bibitem[Ye et~al.(2020{\natexlab{c}})Ye, Zhou, Luo, and
  Zhang]{DBLP:conf/nips/YeZL020}
Haishan Ye, Ziang Zhou, Luo Luo, and Tong Zhang.
\newblock Decentralized accelerated proximal gradient descent.
\newblock In \emph{NeurIPS}, 2020{\natexlab{c}}.

\bibitem[Ying et~al.(2016)Ying, Wen, and Lyu]{ying2016stochastic}
Yiming Ying, Longyin Wen, and Siwei Lyu.
\newblock Stochastic online {AUC} maximization.
\newblock \emph{NIPS}, 2016.

\bibitem[Zhang et~al.(2018)Zhang, Lemoine, and Mitchell]{zhang2018mitigating}
Brian~Hu Zhang, Blake Lemoine, and Margaret Mitchell.
\newblock Mitigating unwanted biases with adversarial learning.
\newblock In \emph{AIES}, 2018.

\bibitem[Zhang et~al.(2019)Zhang, Hong, and Zhang]{zhang2019lower}
Junyu Zhang, Mingyi Hong, and Shuzhong Zhang.
\newblock On lower iteration complexity bounds for the saddle point problems.
\newblock \emph{arXiv preprint:1912.07481}, 2019.

\bibitem[Zhang et~al.(2013)Zhang, Mahdavi, and Jin]{zhang2013linear}
Lijun Zhang, Mehrdad Mahdavi, and Rong Jin.
\newblock Linear convergence with condition number independent access of full
  gradients.
\newblock In \emph{NIPS}, 2013.

\bibitem[Zhang and Xiao(2017)]{zhang2017stochastic}
Yuchen Zhang and Lin Xiao.
\newblock Stochastic primal-dual coordinate method for regularized empirical
  risk minimization.
\newblock \emph{Journal of Machine Learning Research}, 18\penalty0
  (1):\penalty0 2939--2980, 2017.

\end{thebibliography}

\newpage
\appendix
Appendices are organized as follows.
In Section \ref{appendix:lemmas}, we provide some useful lemmas for our proofs.
In Section \ref{appendix:MC-SVRE-unconstrained}, we give the detailed proofs for Section \ref{sec:mc-svre-unconstrained}.
In Section \ref{appendix:MC-SVRE-constrained}, we give the detailed proofs for Section \ref{sec:mc-svre-constrained}.
In Section \ref{appendix:MCEG}, we give the detailed proofs for Section \ref{sec:MC-EG}.
We always use $\BT(\cdot)$ to present the procedure of FastMix (Algorithm \ref{alg:fm}), that is
\begin{align*}
    \BT(\vz) = {\rm FastMix}(\vz, K).
\end{align*}
Recall that we have defined
\begin{align*}
\rho=\left(1-\sqrt{1-\lambda_2(W)}\right)^K,
\end{align*}
where $K$ is the number of iterations in Algorithm \ref{alg:fm}.

\section{Some Useful Lemmas}\label{appendix:lemmas}

We first provide some useful lemmas will be used in the analysis of MC-SVRE and MC-EG. 

\begin{lem}\label{lem:norm}
For any $a_1,\dots,a_m\in\BR^d$, we have 
\begin{align*}
\Norm{\frac{1}{m}\sum_{i=1}^m a_i}^2 \leq \frac{1}{m} \sum_{i=1}^m \Norm{a_i}^2.
\end{align*}
\end{lem}
\begin{proof}
We have
\begin{align*}
  &  \frac{1}{m} \sum_{i=1}^m \Norm{a_i}^2 - \Norm{\frac{1}{m}\sum_{i=1}^m a_i}^2 \\ 
= & \frac{1}{m^2}\left(m\sum_{i=1}^m \Norm{a_i}^2 - \sum_{i=1}^m\sum_{j=1}^m \inner{a_i}{a_j} \right) \\
= & \frac{1}{m^2}\left((m-1)\sum_{i=1}^m \Norm{a_i}^2 - 2\sum_{i\neq j} \inner{a_i}{a_j} \right) \\
= & \frac{1}{m^2}\left((m-1)\sum_{i=1}^m \Norm{a_i}^2 - 2\sum_{i\neq j} \inner{a_i}{a_j} \right) \\
= & \frac{1}{m^2}\sum_{i\neq j} \norm{a_i-a_j}^2 \geq 0.
\end{align*}
\end{proof}

\begin{lem}[{\citet[Lemma 3]{ye2020multi}}]\label{lem:avg-norm}
For any matrix $\vz\in\BR^{m\times d}$, we have
$\Norm{\vz - \vone\bz} \leq \Norm{\vz}$ where $\bz=\frac{1}{m}\vone^\top\vz$.
\end{lem}

\begin{lem}\label{lem:smooth-b}
Under Assumption~\ref{asm:smooth}, we have
$\Norm{\vg(\vz)-\vg(\vz')} \leq  L\Norm{\vz-\vz'}$ for any $\vz, \vz'\in\BR^{m\times d}$.
\end{lem}
\begin{proof}
Assumption~\ref{asm:smooth} means $g_i$ is $L$-Lipschitz continuous. Then we have
\begin{align*}
& \Norm{\vg(\vz)-\vg(\vz')} \\
= & \Norm{\begin{bmatrix}
g_1(\vz(1))^\top \\ \vdots \\ g_m(\vz(m))^\top
\end{bmatrix} -
\begin{bmatrix}
g_1(\vz'(1))^\top \\ \vdots \\ g_m(\vz'(m))^\top
\end{bmatrix}}^2  \\
= & \sum_{i=1}^m \Norm{g_i(\vz(i))-g_i(\vz'(i))}^2 \\
\leq & L^2\sum_{i=1}^m \Norm{\vz(i)-\vz'(i)}^2 \\
= & L^2\Norm{\vz-\vz'}^2.
\end{align*}
\end{proof}

\begin{lem}[{\citet[Theorem 1]{rockafellar1970monotone}}]\label{lem:mm}
For any $z, z'\in\BR^{d}$, we have 
\begin{align*}
\inner{g(z)-g(z')}{z-z'} \geq  0    
\end{align*}
under Assumption~\ref{asm:cc} and
\begin{align*}
\inner{g(z)-g(z')}{z-z'} \geq  \mu\Norm{z-z'} 
\end{align*}
under Assumption~\ref{asm:scsc}.
\end{lem}

\begin{lem}\label{lem:stst12}
For Algorithm \ref{alg:dsvre}, we have
$\bs_t=\frac{1}{m}\vone^\top\vg(\vw_t)$ and $\BE_t[\bs_{t+1/2}]=\frac{1}{m}\vone^\top\vg(\vz_{t+1/2})$.
\end{lem}
\begin{proof}
We prove this lemma by induction.
For $t=0$, we have
\begin{align*}
\bs_0 =  \frac{1}{m}\vone^\top\vs_0 = \frac{1}{m}\vone^\top\vg(\vz_0) = \frac{1}{m}\vone^\top\vg(\vw_0) 
\end{align*}
and
\begin{align*}
\BE_0[\bs_{1/2}] = \frac{1}{m}\vone^\top\BE_0[\vs_{1/2}] = \frac{1}{m}\vone^\top\BE_t[\vs_0+\vv_{1/2}-\vv_0] = \frac{1}{m}\vone^\top\BE_t[\vv_{1/2}] = \frac{1}{m}\vone^\top\vg(\vz_{1/2}).
\end{align*}
Suppose the statement holds for $t\leq t'$. Then for $t=t'+1$, Lemma~\ref{lem:FM} and induction base means
\begin{align*}
& \bs_{t+1} \\
= &  \bs_t + \bv_{t+1} - \bv_t \\
= & \frac{1}{m}\vone^\top\vg(\vw_t) + \frac{1}{m}\vone^\top\vg(\vw_{t+1}) - \frac{1}{m}\vone^\top\vg(\vw_t) \\
= & \frac{1}{m}\vone^\top\vg(\vw_{t+1}),
\end{align*}
and
\begin{align*}
  & \BE_t[\bs_{t+1/2}] \\
= & \BE_t[\bs_t + \bv_{t+1/2} - \bv_t]  \\
= & \BE_t\left[\frac{1}{m}\vone^\top\vg(\vw_t) + \bv_{t+1/2} - \frac{1}{m}\vone^\top\vg(\vw_t)\right]  \\
= & \BE_t[\bv_{t+1/2}]  \\
= & \frac{1}{m}\vone^\top\BE_t\begin{bmatrix}
g_1(\vw_t(1))^\top + g_{1,j_1}(\vz_{t+1/2}(1))^\top - g_{1,j_1}(\vw_t(1))^\top \\
\vdots \\
g_m(\vw_t(m))^\top + g_{m,j_m}(\vz_{t+1/2}(m))^\top - g_{m,j_m}(\vw_t(m))^\top
\end{bmatrix} \\
= & \frac{1}{m}\vone^\top\vg(\vz_{t+1/2}).
\end{align*}
\end{proof}

\begin{lem}\label{lem:zpzw}
For Algorithm \ref{alg:dsvre}, we have
\begin{align*}
\Norm{\vone^\top\vz'_t-\bz'_t} \leq \alpha\Norm{\vz_t-\vone\bz_t} + (1-\alpha)\Norm{\vw_t-\vone\bw_t}.    
\end{align*}
\end{lem}
\begin{proof}
The update rule of $\vz'_t = \alpha \vz_t + (1 - \alpha) \vw_t$ means
\begin{align*}
  \Norm{\vz'_t-\vone\bz'_t} 
=  \Norm{\alpha\vz_t+(1-\alpha)\vw_t-\vone(\alpha\bz_t+(1-\alpha))\bw_t)} 
\leq \alpha\Norm{\vz_t-\vone\bz_t} + (1-\alpha)\Norm{\vw_t-\vone\bw_t}.
\end{align*}
\end{proof}

\section{The Proof Details for Section \ref{sec:mc-svre-unconstrained}}\label{appendix:MC-SVRE-unconstrained}

This section provide the detailed proofs for theoretical results of MC-SVRE in unconstrained case.

\subsection{The Proof of Lemma \ref{lem:grad-avg}}
\begin{proof}
Assumption~\ref{asm:smooth} and Lemma \ref{lem:stst12} means
\begin{align*}
  & \Norm{\bs_t - \bg(\bw_t)}^2 \\
= & \Norm{\frac{1}{m}\sum_{i=1}^m \left(g_i(\vw_t(i)) - g_i(\bw_t^\top)\right)}^2   \\
\leq & \frac{1}{m}\sum_{i=1}^m\Norm{g_i(\vw_t(i)) - g_i(\bw_t^\top)}^2 \\
\leq & \frac{L^2}{m}\sum_{i=1}^m\Norm{\vw_t(i) - \bw_t^\top}^2 \\
= & \frac{L^2}{m}\Norm{\vw_t - \vone\bw_t}^2.
\end{align*}
Similarly, we can also prove
$\BE_t[\bv_{t+1/2}]=\frac{1}{m}\vone^\top\vg(\vz_{t+1/2})$. 
\end{proof}

\subsection{The Proof of Lemma \ref{lem:zw}}
\begin{proof}
Lemma \ref{lem:FM} means
\begin{align*}
   \bz_{t+1/2} 
= \frac{1}{m}\vone^\top\BT(\vz'_t-\eta\vs_t) 
= \bz'_t - \eta\bs_t 
\qquad \text{and} \qquad
   \bz_{t+1}  
=  \frac{1}{m}\vone^\top\BT(\vz'_t-\eta\vs_{t+1/2})
= \bz'_t - \eta\bs_{t+1/2}.
\end{align*}
which implies
\begin{align}\label{eq:1}
\begin{split}
& 2\inner{\bz_{t+1} - \bz'_t}{\bz^* - \bz_{t+1}} + 2\inner{\bz_{t+1/2} - \bz'_t}{\bz_{t+1} - \bz_{t+1/2}} \\
& + 2\eta\inner{\bs_{t+1/2}}{\bz^* - \bz_{t+1/2}} + 2\eta \inner{\bs_{t+1/2} - \bs_t}{\bz_{t+1/2} - \bz_{t+1}} = 0.
\end{split}
\end{align}
Using the fact $2\inner{a}{b} = \Norm{a + b}^2 - \Norm{a}^2 - \Norm{b}^2$, we have
\begin{align}\label{eq:1-1}%
\begin{split}
    & 2\inner{\bz_{t+1} - \bz'_t}{\bz^* - \bz_{t+1}} \\
    =& 2 \inner{\bz_{t+1} - \alpha \bz_t - (1 - \alpha) \bw_t}{\bz^* - \bz_{k + 1}} \\
    =& 2\alpha \inner{\bz_{t+1} - \bz_t}{\bz^* - \bz_{t+1}} + 2(1 - \alpha) \inner{\bz_{t+1} - \bw_t}{\bz^* - \bz_{t+1}} \\
    =& \alpha \big( \Norm{\bz_t - \bz^*}^2 - \Norm{\bz_{t+1} - \bz^*}^2 - \Norm{\bz_{t+1} - \bz_t}^2 \big) 
     + (1 - \alpha) \big( \Norm{\bw_t - \bz^*}^2 - \Norm{\bz_{t+1} - \bz^*}^2 - \Norm{\bz_{t+1} - \bw_t}^2 \big) \\
    =& \alpha \Norm{\bz_t - \bz^*}^2 + (1 - \alpha) \Norm{\bw_t - \bz^*}^2 - \Norm{\bz_{t+1} - \bz^*}^2  - \alpha \Norm{\bz_{t+1} - \bz_t}^2 - (1 - \alpha) \Norm{\bz_{t+1} - \bw_t}^2.
\end{split}
\end{align}
and
\begin{align}\label{eq:1-2}
\begin{split}
& 2\inner{\bz_{t+1/2} - \bz'_t}{\bz_{t+1} - \bz_{t+1/2}} \\
=& \alpha \Norm{\bz_t - \bz_{t+1}}^2 + (1 - \alpha) \Norm{\bw_t - \bz_{t+1}}^2 - \Norm{\bz_{t+1/2} - \bz_{t+1}}^2 \\
& - \alpha \Norm{\bz_{t+1/2} - \bz_t}^2 - (1 - \alpha) \Norm{\bz_{t+1/2} - \bw_t}^2.
\end{split}
\end{align}
Additionally, we have
\begin{align}\label{eq:1-3}
\begin{split}
& 2 \BE_t \left[\inner{\bs_{t+1/2}}{\bz^* - \bz_{t+1/2}} \right] \\
= & 2\inner{\bg(\bz_{t+1/2})}{\bz^* - \bz_{t+1/2}} + 2\inner{\BE_t[\bs_{t+1/2}] - g(\bz_{t+1/2})}{\bz^* - \bz_{t+1/2}} \\
\leq & 2 \inner{\bg(\bz^*)}{\bz^* - \bz_{t+1/2}} - 2 \mu \Norm{\bz_{t+1/2} - \bz^*}^2 + 2\inner{\BE_t[\bs_{t+1/2}] - g(\bz_{t+1/2})}{\bz^* - \bz_{t+1/2}} \\
\leq & - \mu \BE_t\Norm{\bz_{t+1} - \bz^*}^2 + 2 \mu \BE_t\Norm{\bz_{t+1/2} - \bz_{t+1}}^2 + \frac{4}{\mu}\Norm{\BE_t[\bs_{t+1/2}] - g(\bz_{t+1/2})}^2 + \frac{\mu}{4}\Norm{\bz_{t+1/2}-\bz^*}^2 \\
\leq & - \mu \BE_t\Norm{\bz_{t+1} - \bz^*}^2 + 2 \mu \BE_t\Norm{\bz_{t+1/2} - \bz_{t+1}}^2 + \frac{4L^2}{m\mu}\Norm{\vz_{t+1/2} - \vone\bz_{t+1/2}}^2 \\
& + \frac{\mu}{2}\Norm{\bz_{t+1/2}-\bz_{t+1}}^2  + \frac{\mu}{2}\Norm{\bz_{t+1}-\bz^*}^2 \\ 
= & - \mu \BE_t\Norm{\bz_{t+1} - \bz^*}^2 + \frac{5\mu}{2} \BE_t\Norm{\bz_{t+1/2} - \bz_{t+1}}^2 + \frac{4L^2}{m\mu}\Norm{\vz_{t+1/2} - \vone\bz_{t+1/2}}^2 + \frac{\mu}{2}\Norm{\bz_{t+1}-\bz^*}^2,
\end{split}
\end{align}
where the first inequality follows Lemma \ref{lem:mm}; the second inequality is according to Lemma \ref{lem:optimal} and Young's inequality; and the last inequality is based on Lemma \ref{lem:grad-avg}.

Moreover, the update rule  $\vs_{t+1}=\BT(\vs_t+\vg(\vw_{t+1})-\vg(\vw_t))$, $\vv_t=\vg(\vw_t)$ and Lemma \ref{lem:FM} implies
\begin{align*}
  &  \BE_t\Norm{\bs_{t+1/2}-\bs_t}^2 \\
= &  \BE_t\Norm{\frac{1}{m}\vone^\top\vs_{t+1/2}-\frac{1}{m}\vone^\top\vs_t}^2 \\
= &  \BE_t\Norm{\frac{1}{m}\vone^\top\left(\vs_{t}+\vv_{t+1/2}-\vv_t\right)-\frac{1}{m}\vone^\top\vs_t}^2 \\
= &  \BE_t\Norm{\frac{1}{m}\vone^\top\left(\vv_{t+1/2}-\vv_t\right)}^2 \\
= &  \BE_t\Norm{\frac{1}{m}\vone^\top
\begin{bmatrix}
g_1(\vw_t(1))^\top + g_{1,j_1}(\vz_{t+1/2}(1))^\top - g_{1,j_1}(\vw_t(1))^\top - g_1(\vw_t(1))^\top \\
\vdots \\
g_m(\vw_t(m))^\top + g_{m,j_m}(\vz_{t+1/2}(m))^\top - g_{m,j_m}(\vw_t(m))^\top - g_m(\vw_t(m))^\top
\end{bmatrix}}^2 \\
\leq &  \BE_t\left[\frac{1}{m}\sum_{i=1}^m \Norm{g_{i,j_i}(\vz_{t+1/2}(i)) - g_{i,j_i}(\vw_t(i))}^2\right] \\
\leq &  \BE_t\left[\frac{L^2}{m}\sum_{i=1}^m \Norm{\vz_{t+1/2}(i) - \vw_t(i)}^2\right] \\
= &  \frac{L^2}{m}\BE_t\Norm{\vz_{t+1/2} - \vw_t}^2 \\
\leq &  \frac{3L^2}{m}\BE_t\left[\Norm{\vz_{t+1/2}-\vone\bz_{t+1/2}}^2 + \Norm{\vone\bz_{t+1/2} - \vone\bw_t}^2 + \Norm{\vw_t-\vone\bw_t}^2\right] \\
= & 3L^2\BE_t\Norm{\bz_{t+1/2} - \bw_t}^2 + \frac{3L^2}{m}\BE_t\Norm{\vz_{t+1/2}-\vone\bz_{t+1/2}}^2  + \frac{3L^2}{m}\BE_t\Norm{\vw_t-\vone\bw_t}^2,
\end{align*}
where the first inequality use Lemma~\ref{lem:norm}; the second inequality use \ref{lem:smooth-b} and the last inequality is due to the fact 
$\Norm{a+b+c}^2\leq3\big(\Norm{a}^2+\Norm{b}^2+\Norm{c}^2\big)$.

By Young's inequality, there also holds
\begin{align}\label{eq:1-4}
\begin{split}
& \BE_t \left[ 2\eta \inner{\bs_{t+1/2} - \bs_t}{\bz_{t+1/2} - \bz_{t+1}} \right] \\
\leq & 2\eta^2 \BE_t\Norm{\bs_{t+1/2} - \bs_t}^2 + \frac{1}{2} \BE_t\Norm{{\bz_{t+1/2} - \bz_{t+1}}}^2 \\
\leq & 6\eta^2L^2\BE_t\Norm{\bz_{t+1/2} - \bw_t}^2 + \frac{6\eta^2L^2}{m}\BE_t\Norm{\vz_{t+1/2}-\vone\bz_{t+1/2}}^2  + \frac{6\eta^2L^2}{m}\BE_t\Norm{\vw_t-\vone\bw_t}^2 \\
& + \frac{1}{2} \BE_t \Norm{{\bz_{t+1/2} - \bz_{t+1}}}^2.
\end{split}
\end{align}
Plugging results of (\ref{eq:1-1}), (\ref{eq:1-2}), (\ref{eq:1-3}) and (\ref{eq:1-4}) into inequality (\ref{eq:1}), we obtain that
\begin{align*}
\begin{split}
& \alpha \Norm{\bz_t - \bz^*}^2 + (1 - \alpha) \Norm{\bw_t - \bz^*}^2 - \BE_t\Norm{\bz_{t+1} - \bz^*}^2 - \BE_t\Norm{\bz_{t+1/2} - \bz_{t+1}}^2 \\
& - \alpha \Norm{\bz_{t+1/2} - \bz_t}^2 - (1 - \alpha) \Norm{\bz_{t+1/2} - \bw_t}^2 \\
& - \eta\mu \BE_t\Norm{\bz_{t+1} - \bz^*}^2 + \frac{5\eta\mu}{2} \BE_t\Norm{\bz_{t+1/2} - \bz_{t+1}}^2 + \frac{4L^2\eta}{m\mu}\Norm{\vz_{t+1/2} - \vone\bz_{t+1/2}}^2 + \frac{\eta\mu}{2}\Norm{\bz_{t+1}-\bz^*}^2 \\
& + 6\eta^2L^2\BE_t\Norm{\bz_{t+1/2} - \bw_t}^2 + \frac{1}{2} \BE_t\Norm{{\bz_{t+1/2} - \bz_{t+1}}}^2 \\
& + \frac{6\eta^2L^2}{m}\BE_t\Norm{\vz_{t+1/2}-\vone\bz_{t+1/2}}^2  + \frac{6\eta^2L^2}{m}\BE_t\Norm{\vw_t-\vone\bw_t}^2 \geq 0,
\end{split}
\end{align*}
that is
\begin{align}\label{eq:2-1}
\begin{split}
& \left(1+\frac{\eta\mu}{2}\right)\Norm{\bz_{t+1}-\bz^*}^2 \\
\leq & \alpha \Norm{\bz_t - \bz^*}^2 + (1 - \alpha) \Norm{\bw_t - \bz^*}^2  - \left(\frac{1}{2}-\frac{5\eta\mu}{2}\right)\BE_t\Norm{\bz_{t+1/2} - \bz_{t+1}}^2 \\
& - \alpha \Norm{\bz_{t+1/2} - \bz_t}^2 - (1 - \alpha-6\eta^2L^2) \Norm{\bz_{t+1/2} - \bw_t}^2 \\
& + \frac{2L\eta(2\kappa + 3L\eta)}{m} \BE_t\Norm{\vz_{t+1/2}-\vone\bz_{t+1/2}}^2  + \frac{6\eta^2L^2}{m}\BE_t\Norm{\vw_t-\vone\bw_t}^2.
\end{split}
\end{align}
On the other hand, by update rule of $\vw_{t+1}$ means
\begin{align}\label{eq:2-2}
\begin{split}
& \BE_t\norm{\bw_{t+1}-\bz^*}^2 \\
= & \BE_t\norm{\frac{1}{m}\vone^\top\left(\vw_{t+1}-\vz^*\right)}^2 \\
= & p\BE_t\norm{\frac{1}{m}\vone^\top\left(\vz_{t+1}-\vz^*\right)}^2 + (1-p)\BE_t\norm{\frac{1}{m}\vone^\top\left(\vw_t-\vz^*\right)}^2 \\
= & p\BE_t\norm{\bz_{t+1}-\bz^*}^2 + (1-p)\BE_t\norm{\bw_t-\bz^*}^2,
\end{split}
\end{align}
where $\vz^*=[z^*,\dots,z^*]^\top\in\BR^{m\times d}$.

Combining the results of (\ref{eq:2-1}), (\ref{eq:2-2}) and the setting $\alpha=1-p$, we have
\begin{align*}
& \left(1 + \frac{\eta \mu}{2}\right) \BE_t \Norm{\bz_{t+1} - \bz^*}^2 + c_1 \E_{t}\Norm{\bw_{t+1} - \bz^*}^2  \\
\leq & (1 - p) \Norm{\bz_t - \bz^*}^2 + p \Norm{\bw_t - \bz^*}^2 - \left(\frac{1}{2}-\frac{5\eta\mu}{2}\right)\BE_t\Norm{\bz_{t+1/2} - \bz_{t+1}}^2 \\
& - (1-p)\Norm{\bz_{t+1/2} - \bz_t}^2 - (p-6\eta^2L^2) \Norm{\bz_{t+1/2} - \bw_t}^2 \\
& + c_1 p \E_{t}\Norm{\bz_{t+1} - \bz^*}^2 + c_1(1 - p) \Norm{\bw_t - \bz^*}^2 \\
& + \frac{2 L\eta\left(2\kappa + 3L\eta\right)}{m}\BE_t\Norm{\vz_{t+1/2}-\vone\bz_{t+1/2}}^2 
 + \frac{6\eta^2L^2}{m}\BE_t\Norm{\vw_t-\vone\bw_t}^2 ,
\end{align*}
which implies
\begin{align} 
\begin{split}
& \left(1 + \frac{\eta \mu}{2} - c_1p\right) \BE_t \Norm{\bz_{t+1} - \bz^*}^2 + c_1 \E_{t}\Norm{\bw_{t+1} - \bz^*}^2  \\
\leq & (1 - p) \Norm{\bz_t - \bz^*}^2 + (p + c_1(1 - p)) \Norm{\bw_t - \bz^*}^2 - \left(\frac{1}{2}-\frac{5\eta\mu}{2}\right)\BE_t\Norm{\bz_{t+1/2} - \bz_{t+1}}^2 \\
& - (1-p)\Norm{\bz_{t+1/2} - \bz_t}^2 - (p-6\eta^2L^2) \Norm{\bz_{t+1/2} - \bw_t}^2 \\
& + \frac{2 L\eta\left(2\kappa + 3L\eta\right)}{m}\BE_t\Norm{\vz_{t+1/2}-\vone\bz_{t+1/2}}^2 
 + \frac{6\eta^2L}{m}\BE_t\Norm{\vw_t-\vone\bw_t}^2 \\
\leq & (1 - p) \Norm{\bz_t - \bz^*}^2 + (p + c_1(1 - p)) \Norm{\bw_t - \bz^*}^2 - \left(\frac{1}{2}-\frac{5\eta\mu}{2}\right)\BE_t\Norm{\bz_{t+1/2} - \bz_{t+1}}^2 \\
& - (1-p)\Norm{\bz_{t+1/2} - \bz_t}^2 - (p-6\eta^2L^2) \Norm{\bz_{t+1/2} - \bw_t}^2 \\
& + \frac{2 L\eta\left(2\kappa + 3L\eta\right)}{m}\BE_t\left[\Norm{\vz_{t+1/2}-\vone\bz_{t+1/2}}^2 
 + \BE_t\Norm{\vw_t-\vone\bw_t}^2\right] \\
\leq & (1 - p) \Norm{\bz_t - \bz^*}^2 + (p + c_1(1 - p)) \Norm{\bw_t - \bz^*}^2 - \frac{1}{3n}\delta_t \\
& + \frac{8 L\eta\left(2\kappa + 3L\eta\right)}{m}\left(\Norm{\vz_t - \vone\bz_t}^2 + \Norm{\vw_t - \vone\bw_t}^2 + \eta\Norm{\vs_t - \vone\bs_t}^2\right),
\end{split}
\end{align}
where the last step use $\eta=1/\big(6\sqrt{n}L\big)$ and Lemma \ref{lem:diff-z12} which leads to
\begin{align*}
& \Norm{\vz_{t+1/2} - \vone\bz_{t+1/2}}^2 + \Norm{\vw_t-\vone\bw_t}^2 \\
\leq & 3\rho^2\left(\Norm{\vz_t - \vone\bz_t}^2 + \Norm{\vw_t - \vone\bw_t}^2 + \eta^2\Norm{\vs_t - \vone\bs_t}^2\right) + \Norm{\vw_t - \vone\bw_t}^2 \\
\leq & 4\left(\Norm{\vz_t - \vone\bz_t}^2 + \Norm{\vw_t - \vone\bw_t}^2 + \eta^2\Norm{\vs_t - \vone\bs_t}^2\right).
\end{align*}
\end{proof}

\subsection{The Proof of Lemma \ref{lem:error}}

We first introduce some lemmas for the consensus error of each variables.

\begin{lem}\label{lem:diff-z12}
Suppose Assumption \ref{asm:unconstrained} holds. For Algorithm \ref{alg:dsvre}, we have
\begin{align*}
   \Norm{\vz_{t+1/2} - \vone\bz_{t+1/2}} 
\leq \rho\left(\Norm{\vz_t - \vone\bz_t} + \Norm{\vw_t - \vone\bw_t} + \eta\Norm{\vs_t - \vone\bs_t}\right).
\end{align*}
\end{lem}
\begin{proof}
Using Lemma \ref{lem:FM} and the update rule $\vz'_t=\alpha\vz_t+(1-\alpha)\vw_t$, we have
\begin{align*}
   & \Norm{\vz_{t+1/2} - \vone\bz_{t+1/2}} \\
= & \Norm{\BT(\vz'_t - \eta\vs_t) - \frac{1}{m}\vone\vone^\top\BT(\vz'_t - \eta\vs_t)} \\
\leq  & \rho \Norm{(\vz'_t - \eta\vs_t)  - \frac{1}{m}\vone\vone^\top\left(\vz'_t - \eta\vs_t\right)} \\
=  & \rho \Norm{(\vz'_t - \eta\vs_t)  - \vone(\bz'_t - \eta\bs_t)} \\
\leq  & \rho \Norm{\vz'_t - \vone\bz'_t} + \rho\eta\Norm{\vs_t - \vone\bs_t} \\
=  & \rho\left(\alpha\Norm{\vz_t - \vone\bz_t} + (1-\alpha)\Norm{\vw_t - \vone\bw_t} + \eta\Norm{\vs_t - \vone\bs_t}\right) \\
\leq  & \rho\left(\Norm{\vz_t - \vone\bz_t} + \Norm{\vw_t - \vone\bw_t} + \eta\Norm{\vs_t - \vone\bs_t}\right).
\end{align*}
\end{proof}

\begin{lem}\label{lem:errs12}
Suppose Assumption \ref{asm:smooth} holds. For Algorithm \ref{alg:dsvre}, we have
\begin{align*}
   \Norm{\vs_{t+1/2} - \vone\bs_{t+1/2}} 
\leq  \rho\left(\Norm{\vs_t - \vone\bs_t} + L\Norm{\vz_{t+1/2} - \vone\bz_{t+1/2}} + L\sqrt{m}\Norm{\bz_{t+1/2} - \bw_t} + L\Norm{\vw_t - \vone\bw_t}\right).
\end{align*}
\end{lem}
\begin{proof}
We have
\begin{align}\label{eq:s12-1}
\begin{split}
   & \Norm{\vs_{t+1/2} - \vone\bs_{t+1/2}} \\
= & \Norm{\BT(\vs_t + \vv_{t+1/2} - \vv_t) - \frac{1}{m}\vone\vone^\top\BT(\vs_t + \vv_{t+1/2} - \vv_t)} \\
\leq & \rho\Norm{\vs_t + \vv_{t+1/2} - \vv_t - \frac{1}{m}\vone\vone^\top(\vs_t + \vv_{t+1/2} - \vv_t)} \\
\leq & \rho\Norm{\vs_t - \frac{1}{m}\vone\vone^\top\vs_t} + \rho\Norm{\vv_{t+1/2} - \vv_t - \frac{1}{m}\vone\vone^\top(\vv_{t+1/2} - \vv_t)} \\
\leq & \rho\left(\Norm{\vs_t - \vone\bs_t} + \Norm{\vv_{t+1/2} - \vv_t}\right),
\end{split}
\end{align}
where the first inequality uses Lemma \ref{lem:FM} and the last one uses Lemma \ref{lem:avg-norm}. 

Then we bound the last term as follows
\begin{align}\label{eq:s12-2}
\begin{split}
  &  \Norm{\vv_{t+1/2}-\vv_t}^2 \\
= &  \Norm{\begin{bmatrix}
g_1(\vw_t(1))^\top + g_{1,j_1}(\vz_{t+1/2}(1))^\top - g_{1,j_1}(\vw_t(1))^\top - g_1(\vw_t(1))^\top \\
\vdots \\
g_m(\vw_t(m))^\top + g_{m,j_m}(\vz_{t+1/2}(m))^\top - g_{m,j_m}(\vw_t(m))^\top - g_m(\vw_t(m))^\top
\end{bmatrix}}^2 \\
= &  \sum_{i=1}^m \Norm{g_{i,j_i}(\vz_{t+1/2}(i)) - g_{i,j_i}(\vw_t(i))}^2 \\
\leq & L^2\sum_{i=1}^m \Norm{\vz_{t+1/2}(i) - \vw_t(i)}^2 \\
\leq & 2L^2\sum_{i=1}^m \left(\Norm{\vz_{t+1/2}(i) - z^*}^2 + \Norm{\vw_t(i)- z^*}^2\right)  \\
= & 2L^2\left(\Norm{\vz_{t+1/2} - \vone z^*}^2 + \Norm{\vw_t- \vone z^*}^2\right)  \\
\leq & L\left(\Norm{\vz_{t+1/2} - \vone\bz_{t+1/2}} + \Norm{\vone\bz_{t+1/2} - \vone\bw_t} + \Norm{\vw_t - \vone\bw_t}\right) \\
= & L\left(\Norm{\vz_{t+1/2} - \vone\bz_{t+1/2}} + \sqrt{m}\Norm{\bz_{t+1/2} - \bw_t} + \Norm{\vw_t - \vone\bw_t}\right),
\end{split}
\end{align}
where the first inequality use the Lipschitz continuity of $g_{ij}$ and the second inequality use Young's inequality.

Combing the results of (\ref{eq:s12-1}) and (\ref{eq:s12-2}), we have
\begin{align*}
   \Norm{\vs_{t+1/2} - \vone\bs_{t+1/2}} 
\leq  \rho\left(\Norm{\vs_t - \vone\bs_t} + L\Norm{\vz_{t+1/2} - \vone\bz_{t+1/2}} + L\sqrt{m}\Norm{\bz_{t+1/2} - \bw_t} + L\Norm{\vw_t - \vone\bw_t}\right).
\end{align*}
\end{proof}

\begin{lem}\label{lem:errz1}
Suppose Assumption \ref{asm:unconstrained} holds.
For Algorithm \ref{alg:dsvre}, we have
\begin{align*}
   \Norm{\vz_{t+1} - \vone\bz_{t+1}} 
\leq  \rho\left(\Norm{\vz_t - \vone\bz_t} + \Norm{\vw_t - \vone\bw_t} + \eta\Norm{\vs_{t+1/2} - \vone\bs_{t+1/2}}\right). 
\end{align*}
\end{lem}
\begin{proof}
Using Lemma \ref{lem:FM}, Lemma \ref{lem:avg-norm} and the update rule of $\vz_t'$, we have
\begin{align*}
   & \Norm{\vz_{t+1} - \vone\bz_{t+1}} \\
= & \Norm{\BT(\vz'_t - \eta\vs_{t+1/2}) - \frac{1}{m}\vone\vone^\top\BT(\vz'_t - \eta\vs_{t+1/2})} \\
\leq  & \rho \Norm{(\vz'_t - \eta\vs_{t+1/2})  - \frac{1}{m}\vone\vone^\top\left(\vz'_t - \eta\vs_{t+1/2}\right)} \\
\leq  & \rho \Norm{\vz'_t - \eta\vs_{t+1/2}  - \vone(\bz'_t - \eta\bs_{t+1/2})} \\
\leq  & \rho\left(\Norm{\vz'_t - \vone\bz'_t} + \eta\Norm{\vs_{t+1/2} - \vone\bs_{t+1/2}}\right) \\
\leq  & \rho\left(\alpha\Norm{\vz_t - \vone\bz_t} + (1-\alpha)\Norm{\vw_t - \vone\bw_t} + \eta\Norm{\vs_{t+1/2} - \vone\bs_{t+1/2}}\right) \\
\leq  & \rho\left(\Norm{\vz_t - \vone\bz_t} + \Norm{\vw_t - \vone\bw_t} + \eta\Norm{\vs_{t+1/2} - \vone\bs_{t+1/2}}\right). 
\end{align*}
\end{proof}

\begin{lem}\label{lem:errw1}
For Algorithm \ref{alg:dsvre}, we have
\begin{align*}
\BE_t\norm{\vw_{t+1}-\vone\bw_{t+1}} =  \frac{1}{2n}\BE_t\norm{\vz_{t+1}-\vone\bz_{t+1}} + \left(1-\frac{1}{2n}\right)\BE_t\norm{\vw_t-\vone\bw_t} 
\end{align*}
and
\begin{align*}
\BE_t\norm{\vw_{t+1}-\vone\bw_{t+1}}^2 =  \frac{1}{2n}\BE_t\norm{\vz_{t+1}-\vone\bz_{t+1}}^2 + \left(1-\frac{1}{2n}\right)\BE_t\norm{\vw_t-\vone\bw_t}^2. 
\end{align*}
\end{lem}
\begin{proof}
Using the update rule of $\vw_{t+1}$ and $p=1/(2n)$, we have
\begin{align*}
  &  \BE_t\norm{\vw_{t+1}-\vone\bw_{t+1}} \\
= & p\BE_t\norm{\vz_{t+1}-\vone\bz_{t+1}} + (1-p)\BE_t\norm{\vw_t-\vone\bw_t}  \\
= & \frac{1}{2n}\BE_t\norm{\vz_{t+1}-\vone\bz_{t+1}} + \left(1-\frac{1}{2n}\right)\BE_t\norm{\vw_t-\vone\bw_t}.
\end{align*}
and
\begin{align*}
  &  \BE_t\norm{\vw_{t+1}-\vone\bw_{t+1}}^2 \\
= & p\BE_t\norm{\vz_{t+1}-\vone\bz_{t+1}}^2 + (1-p)\BE_t\norm{\vw_t-\vone\bw_t}^2  \\
= & \frac{1}{2n}\BE_t\norm{\vz_{t+1}-\vone\bz_{t+1}}^2 + \left(1-\frac{1}{2n}\right)\BE_t\norm{\vw_t-\vone\bw_t}^2.
\end{align*}
\end{proof}

\begin{lem}\label{lem:err-s}
Suppose Assumption \ref{asm:smooth} holds. For Algorithm \ref{alg:dsvre}, we have
\begin{align*}
  &  \BE_t\Norm{\vs_{t+1} - \vone\bs_{t+1}} \\
\leq & \rho\BE_t\Norm{\vs_t - \vone\bs_t} + \rho L \BE_t\left[\Norm{\vz_{t+1} - \vone\bz_{t+1}} + 2\Norm{\vz_{t+1/2} - \vone\bz_{t+1/2}} + \Norm{\vw_t - \vone\bw_t}\right] \\
& + \rho L\sqrt{m}\BE_t\left[\Norm{\bz_{t+1} - \bz_{t+1/2}} +  \Norm{\bz_{t+1/2} - \bw_t} \right].
\end{align*}
\end{lem}
\begin{proof}
Using Lemma \ref{lem:FM}, \ref{lem:smooth-b}, \ref{lem:avg-norm}  and the update rule of $\vz_t'$ and $\vw_{t+1}$, we have
\begin{align*}
   & \BE_t\Norm{\vs_{t+1} - \vone\bs_{t+1}} \\
= & \BE_t\Norm{\BT(\vs_t + \vg(\vw_{t+1}) - \vg(\vw_t)) - \frac{1}{m}\vone\vone^\top\BT(\vs_t + \vg(\vw_{t+1}) - \vg(\vw_t))} \\
\leq & \rho\BE_t\Norm{\vs_t + \vg(\vw_{t+1}) - \vg(\vw_t) - \frac{1}{m}\vone\vone^\top(\vs_t + \vg(\vw_{t+1}) - \vg(\vw_t))} \\
\leq & \rho\BE_t\Norm{\vs_t - \vone\bs_t} + \rho\BE_t\Norm{\vg(\vw_{t+1}) - \vg(\vw_t) - \frac{1}{m}\vone\vone^\top(\vg(\vw_{t+1}) - \vg(\vw_t))} \\
\leq & \rho\BE_t\Norm{\vs_t - \vone\bs_t} + \rho\Norm{\vg(\vw_{t+1}) - \vg(\vw_t)} \\
\leq & \rho\BE_t\Norm{\vs_t - \vone\bs_t} + \rho L\BE_t\Norm{\vw_{t+1} - \vw_t} \\
\leq & \rho\BE_t\Norm{\vs_t - \vone\bs_t} + \rho L p\BE_t\Norm{\vz_{t+1} - \vw_t} + \rho L (1-p)\BE_t\Norm{\vw_{t} - \vw_t} \\
\leq & \rho\BE_t\Norm{\vs_t - \vone\bs_t} + \rho L p\BE_t\Norm{\vz_{t+1} - \vz_{t+1/2}} + \rho L p\BE_t\Norm{\vz_{t+1/2} - \vw_t} \\
\leq & \rho\BE_t\Norm{\vs_t - \vone\bs_t} + \rho L p\BE_t\left[\Norm{\vz_{t+1} - \vone\bz_{t+1}} + \Norm{\vone\bz_{t+1} - \vone\bz_{t+1/2}} + \Norm{\vone\bz_{t+1/2} - \vz_{t+1/2}}\right] \\
& + \rho L p\BE_t\left[\Norm{\vz_{t+1/2} - \vone\bz_{t+1/2}} + \Norm{\vone\bz_{t+1/2} - \vone\bw_t} + \Norm{\vw_t - \vone\bw_t}\right] \\
= & \rho\BE_t\Norm{\vs_t - \vone\bs_t} + \rho L \BE_t\left[\Norm{\vz_{t+1} - \vone\bz_{t+1}} + 2\Norm{\vz_{t+1/2} - \vone\bz_{t+1/2}} + \Norm{\vw_t - \vone\bw_t}\right] \\
& + \rho L\sqrt{m}\BE_t\left[\Norm{\bz_{t+1} - \bz_{t+1/2}} +  \Norm{\bz_{t+1/2} - \bw_t} \right]. 
\end{align*}
\end{proof}

We use the notations of 
\begin{align*}
r_t = \begin{bmatrix}
\norm{\vz_t-\vone\bz_t} \\[0.1cm] \norm{\vw_t-\vone\bw_t} \\[0.1cm]
\eta\norm{\vs_t-\vone\bs_t}
\end{bmatrix}
\qquad \text{and} \qquad
r_t^2 = \begin{bmatrix}
\norm{\vz_t-\vone\bz_t}^2 \\[0.1cm] \norm{\vw_t-\vone\bw_t}^2 \\[0.1cm]
\eta^2\norm{\vs_t-\vone\bs_t}^2
\end{bmatrix}.
\end{align*}
Then we provide the proof of Lemma \ref{lem:error}.
\begin{proof}
Using Lemma \ref{lem:diff-z12}, we have
\begin{align*}
 &  \Norm{\vz_{t+1/2} - \vone\bz_{t+1/2}} 
\leq \rho\begin{bmatrix}
1 & 1 & 1 
\end{bmatrix} r_t.
\end{align*}
Using Lemma \ref{lem:errs12}, we have
\begin{align*}
   & \Norm{\vs_{t+1/2} - \vone\bs_{t+1/2}} \\
\leq & \rho\left(\begin{bmatrix} 0 & L & 1/\eta \end{bmatrix} r_t
+ L \Norm{\vz_{t+1/2} - \vone\bz_{t+1/2}}  +  L\sqrt{m}\Norm{\bz_{t+1/2} - \bw_t}\right) \\
= & \rho\left(\begin{bmatrix} 0 & L & 1/\eta \end{bmatrix} r_t
+ \rho L\begin{bmatrix} 1 & 1 & 1 \end{bmatrix} r_t  
+  L\sqrt{m}\Norm{\bz_{t+1/2} - \bw_t}\right) \\
\leq & \rho\begin{bmatrix} L & 2L & L+1/\eta \end{bmatrix} r_t  
+  \rho L\sqrt{m}\Norm{\bz_{t+1/2} - \bw_t} \\
\leq & \rho\begin{bmatrix} L & 2L & 3/(2\eta) \end{bmatrix} r_t  
+  \rho L\sqrt{m}\Norm{\bz_{t+1/2} - \bw_t}.
\end{align*}
Note that $\eta\leq 1/(4L)$. Using Lemma \ref{lem:errz1}, we have
\begin{align*}
 &  \Norm{\vz_{t+1} - \vone\bz_{t+1}} \\
= & 2\rho\left(
\begin{bmatrix} 1 & 1 & 0 \end{bmatrix} r_t
 + \eta\Norm{\vs_{t+1/2} - \vone\bs_{t+1/2}}\right) \\
\leq & 2\rho\left(
\begin{bmatrix} 1 & 1 & 0 \end{bmatrix} r_t
 + \eta\rho\begin{bmatrix} L & 2L & 3/(2\eta) \end{bmatrix} r_t  
+  \eta\rho L\sqrt{m}\Norm{\bz_{t+1/2} - \bw_t}\right) \\
\leq & 2\rho\left(
\begin{bmatrix} 1 & 1 & 0 \end{bmatrix} r_t
 + \begin{bmatrix} 0.25 & 0.5 & 1.5 \end{bmatrix} r_t  
+  0.25\rho  \sqrt{m}\Norm{\bz_{t+1/2} - \bw_t}\right) \\
\leq & \rho\begin{bmatrix} 2.5 & 3 & 3 \end{bmatrix} r_t  
+  0.5\rho \sqrt{m}\Norm{\bz_{t+1/2} - \bw_t} 
\end{align*}
and
\begin{align*}
&  \Norm{\vz_{t+1} - \vone\bz_{t+1}}^2  \\
\leq & 4\rho^2\begin{bmatrix} 6.25 & 9 & 9 \end{bmatrix} r_t^2  
+  \rho^2m\Norm{\bz_{t+1/2} - \bw_t}^2 \\
\leq & \rho^2\begin{bmatrix} 25 & 36 & 36 \end{bmatrix} r_t^2  
+  \rho^2m\Norm{\bz_{t+1/2} - \bw_t}^2.
\end{align*}
Using Lemma \ref{lem:errw1}, we have
\begin{align*}
  &  \BE_t\norm{\vw_{t+1}-\vone\bw_{t+1}}^2 \\
= & p\BE_t\norm{\vz_{t+1}-\vone\bz_{t+1}}^2 + (1-p)\norm{\vw_t-\vone\bw_t}^2  \\
\leq & \frac{1}{2n}\left(
\rho^2\begin{bmatrix} 25 & 36 & 36 \end{bmatrix} r_t^2  
+  \rho^2 m\Norm{\bz_{t+1/2} - \bw_t}^2\right)
+ \left(1-\frac{1}{2n}\right)\begin{bmatrix}
0 & 1 & 0
\end{bmatrix} r_t^2  \\
\leq & \begin{bmatrix}
\dfrac{25\rho^2}{2n} & 1-\dfrac{1-36\rho^2}{2n} & \dfrac{18\rho^2}{n}
\end{bmatrix} r_t^2 + \frac{\rho^2m}{2n}\Norm{\bz_{t+1/2} - \bw_t}^2.
\end{align*}
Using Lemma \ref{lem:err-s}, we have
\begin{align*}
  &  \BE_t\Norm{\vs_{t+1} - \vone\bs_{t+1}} \\
\leq & \rho\begin{bmatrix}
 0 & L & 1/\eta
\end{bmatrix}r_t + \rho L\left(\rho\begin{bmatrix} 2.5 & 3 & 3 \end{bmatrix} r_t  
+  0.5\rho\sqrt{m}\Norm{\bz_{t+1/2} - \bw_t}\right) \\
& + 2\rho^2 L\begin{bmatrix}1 & 1 & 2 \end{bmatrix} r_t
 + \rho L\sqrt{m}\BE_t\left[\Norm{\bz_{t+1} - \bz_{t+1/2}} +  \Norm{\bz_{t+1/2} - \bw_t} \right] \\
\leq & \rho\begin{bmatrix}
 2.5L & 4L & 1/\eta+3L
\end{bmatrix}r_t +  0.5\rho^3 L\sqrt{m}\Norm{\bz_{t+1/2} - \bw_t} \\
& + \rho \begin{bmatrix}2L & 2L & 4L \end{bmatrix} r_t
 + \rho L\sqrt{m}\BE_t\left[\Norm{\bz_{t+1} - \bz_{t+1/2}} +  \Norm{\bz_{t+1/2} - \bw_t} \right] \\ 
\leq & \rho\begin{bmatrix}
 4.5L & 6L & 1/\eta+7L
\end{bmatrix}r_t  + 1.5\rho L\sqrt{m}\BE_t\left[\Norm{\bz_{t+1} - \bz_{t+1/2}} +  \Norm{\bz_{t+1/2} - \bw_t} \right], 
\end{align*}
that is
\begin{align*}
    \eta\BE_t\Norm{\vs_{t+1} - \vone\bs_{t+1}} 
\leq  \rho\begin{bmatrix}
 9/8 & 3/2 & 11/4
\end{bmatrix}r_t  + \frac{3}{8}\rho\sqrt{m}\BE_t\left[\Norm{\bz_{t+1} - \bz_{t+1/2}} +  \Norm{\bz_{t+1/2} - \bw_t} \right], 
\end{align*}
which implies
\begin{align*}
 & \eta^2\BE_t\Norm{\vs_{t+1} - \vone\bs_{t+1}} \\
\leq & 4\rho^2\begin{bmatrix}
 (9/8)^2 & (3/2)^2 & (11/4)^2
\end{bmatrix}r_t  + \frac{36}{64}\rho^2m\BE_t\left[\Norm{\bz_{t+1} - \bz_{t+1/2}} +  \Norm{\bz_{t+1/2} - \bw_t} \right]^2 \\
\leq & \rho^2\begin{bmatrix}
 \dfrac{81}{16} & 9 & \dfrac{121}{4}
\end{bmatrix}r_t  + \frac{9}{8}\rho^2m\BE_t\left[\Norm{\bz_{t+1} - \bz_{t+1/2}}^2 +  \Norm{\bz_{t+1/2} - \bw_t}^2 \right].
\end{align*}
Combing all above results, we have
\begin{align*}
    \BE_t\left[r_{t+1}^2\right]
\leq & \begin{bmatrix}
 25\rho^2 & 36\rho^2 & 36\rho^2 \\[0.3cm]
 \dfrac{25\rho^2}{2n} & 1-\dfrac{1-36\rho^2}{2n} & \dfrac{18\rho^2}{n} \\[0.3cm]
 \dfrac{81\rho^2}{16} & 9\rho^2 & \dfrac{121\rho^2}{4}
\end{bmatrix} r_t^2 +
\begin{bmatrix}  
\rho^2m\Norm{\bz_{t+1/2} - \bw_t}^2  \\[0.25cm] 
\dfrac{\rho^2m}{2n}\Norm{\bz_{t+1/2} - \bw_t}^2 \\[0.25cm] 
\dfrac{9}{8}\rho^2m\BE_t\left[\Norm{\bz_{t+1} - \bz_{t+1/2}}^2 +  \Norm{\bz_{t+1/2} - \bw_t}^2 \right]
\end{bmatrix}.
\end{align*}
Consider that $\rho\leq 1/\sqrt{291}$, we have
\begin{align*}
    \BE_t\left[\vone^\top r_{t+1}^2\right]
\leq & \left(1-\frac{1}{3n}\right) \vone^\top r_t^2 +
4\rho^2m \BE_t\left[\Norm{\bz_{t+1} - \bz_{t+1/2}}^2 +  \Norm{\bz_{t+1/2} - \bw_t}^2 \right].
\end{align*}
\end{proof}

\subsection{The Proof of Lemma \ref{lem:scsc-unconstrained}}
\begin{proof}
Using Lemma~\ref{lem:zw} and Lemma~\ref{lem:error}, we have
\begin{align*} 
\begin{split}
& \left(1 + \frac{\eta \mu}{2} - c_1p\right) \BE_t \Norm{\bz_{t+1} - \bz^*}^2 + c_1 \E_{t}\Norm{\bw_{t+1} - \bz^*}^2  \\
\leq & \left(1 - p\right) \Norm{\bz_t - \bz^*}^2 + \left(p + c_1\left(1 - p\right)\right) \Norm{\bw_t - \bz^*}^2 - \frac{1}{3n}\delta_t  + C\vone^\top\vr^2_t
\end{split}
\end{align*}
and
\begin{align*}
    \vone^\top r_{t+1}^2
\leq  \left(1-\frac{1}{3n}\right) \vone^\top r_t^2 +
4\rho^2m \delta_t =  \left(1-\frac{2p}{3}\right) \vone^\top r_t^2 +
4\rho^2m \delta_t.
\end{align*}
Hence, we combining above results, we obtain
\begin{align*}
 & \left(1 + \frac{\eta \mu}{2} - c_1p\right) \E \Norm{\bz_{t+1} - \bz^*}^2 + c_1 \E\Norm{\bw_{t+1} - \bz^*}^2 + c_2C\BE_t\left[\vone^\top r_{t+1}\right] \\
\leq & \left(1 - p\right) \Norm{\bz_t - \bz^*}^2 + \left(p + c_1\left(1 - p\right)\right) \Norm{\bw_t - \bz^*}^2 - \frac{1}{3n}\delta_t  + C\vone^\top\vr^2_t 
+ c_2C\left(1-\frac{2p}{3}\right) \vone^\top r_t^2 +4c_2C\rho^2m \delta_t \\
= & \left(1 - p\right) \Norm{\bz_t - \bz^*}^2 + \left(p + c_1\left(1 - p\right)\right) \Norm{\bw_t - \bz^*}^2  
+ \left(1+ c_2\left(1-\frac{2p}{3}\right)\right)C \vone^\top r_t^2 - \left(\frac{1}{3n}-4c_2C\rho^2m\right) \delta_t \\
\leq & \left(1 - p\right) \Norm{\bz_t - \bz^*}^2 + \left(p + c_1\left(1 - p\right)\right) \Norm{\bw_t - \bz^*}^2  
+ \left(1+ c_2\left(1-\frac{2p}{3}\right)\right)C \vone^\top r_t^2
\end{align*}
where the last step is due to the choice of $K$ leads to
$\rho \leq 1/\sqrt{12mnc_2C}$ and $1/(3n)\geq 4c_2C\rho^2m$.

We rewrite above inequality as
\begin{align}\label{ieq:rel-V}
\begin{split}
 & \BE_t \Norm{\bz_{t+1} - \bz^*}^2 + \frac{c_1}{1 + \frac{\eta \mu}{2} - c_1p} \E\Norm{\bw_{t+1} - \bz^*}^2 + \frac{c_2}{1 + \frac{\eta \mu}{2} - c_1p}C\BE_t\left[\vone^\top r_{t+1}\right] \\
\leq & \frac{1 - p}{1 + \frac{\eta \mu}{2} - c_1p} \Norm{\bz_t - \bz^*}^2 + \frac{p + c_1\left(1 - p\right)}{1 + \frac{\eta \mu}{2} - c_1p}\Norm{\bw_t - \bz^*}^2  
+ \frac{1+ c_2\left(1-\frac{2p}{3}\right)}{1 + \frac{\eta \mu}{2} - c_1p}C \vone^\top r_t^2.
\end{split}
\end{align}
The settings $c_1 =\frac{2\eta \mu + 4 p}{\eta \mu + 4 p}$ and $\eta=\frac{1}{6\sqrt{n}L}$ mean
\begin{align*}
  &  \frac{1-p}{1+\frac{\eta\mu}{2}-c_1p} 
= 1 - \frac{\frac{\eta\mu}{2}-p(c_1-1)}{1+\frac{\eta\mu}{2}-c_1p}
= 1 - \frac{\frac{\eta\mu}{2}-\frac{\eta\mu p}{\eta\mu+4p}}{1+\frac{\eta\mu}{2}-\frac{p(2\eta \mu + 4 p)}{\eta \mu + 4 p}} 
= 1 - \frac{\frac{\eta\mu(\eta \mu + 4 p)}{2} - \eta\mu p}{(1+\frac{\eta\mu}{2})(\eta \mu + 4 p)-p(2\eta \mu + 4 p)} \\
= & 1 - \frac{\eta^2\mu^2 + 2\eta\mu p}{(2+\eta\mu)\eta \mu + (2+\eta\mu)4 p - 2p(2\eta \mu + 4 p)} 
= 1 - \frac{\eta^2\mu^2 + 2\eta\mu p}{(2+\eta\mu)\eta \mu + 8p(1 - p)} 
\leq 1 - \frac{\eta^2\mu^2 + 2\eta\mu p}{3\eta\mu + 8p} \\
\leq & 1 - \frac{\eta\mu(\eta\mu + 2p)}{3\eta\mu + 8p}
\leq 1 - \frac{\eta\mu(\eta\mu + 2p)}{4\eta\mu + 8p}
= 1 - \frac{\eta\mu}{4} = 1-\frac{1}{24\kappa\sqrt{n}}
\end{align*}
and
\begin{align*}
\frac{p+c_1(1-p)}{c_1} 
= 1 - \frac{p(c_1-1)}{c_1}
= 1 - \frac{\frac{p\eta \mu}{\eta \mu + 4 p}}{\frac{2\eta \mu + 4 p}{\eta \mu + 4 p}}
= 1 - \frac{p\eta \mu}{2\eta \mu + 4 p}
= 1 - \frac{1}{\frac{2}{p} + \frac{4}{\eta\mu}}
= 1 - \frac{1}{n + 24\kappa\sqrt{n}}.
\end{align*}
Additionally, the value of $c_2=3/p=6n$ means
\begin{align*}
\frac{1+c_2\left(1-\frac{2p}{3}\right)}{c_2}
= \frac{1+\frac{3}{p}\left(1-\frac{2p}{3}\right)}{c_2}
= \frac{\frac{3}{p}\left(\frac{p}{3}+\left(1-\frac{2p}{3}\right)\right)}{\frac{3}{p}}
= 1-\frac{p}{3}=1-\frac{1}{6n}.
\end{align*}
Hence, the inequality (\ref{ieq:rel-V}) implies
\begin{align*}
    \BE_t\left[V_{t+1}\right]
\leq \max\left\{1-\frac{1}{24\kappa\sqrt{n}},1-\frac{1}{n+24\kappa\sqrt{n}},1-\frac{1}{6n}\right\}  V_t
\leq \left(1-\frac{1}{6(n+4\kappa\sqrt{n})}\right)  V_t.
\end{align*}
\end{proof}

\subsection{The Proof of Theorem \ref{thm:scsc-unconstrained}}
\begin{proof}
Note that 
\begin{align*}
1 + \frac{\eta \mu}{\eta \mu + 4 p}
\leq c_1 = \frac{2\eta \mu + 4 p}{\eta \mu + 4 p}
\leq \frac{2\eta \mu + 8 p}{\eta \mu + 4 p} = 2
\end{align*}
which means
\begin{align*}
    \frac{c_1}{1 + \frac{\eta \mu}{2} - c_1p}
\leq \frac{2}{1 + \frac{\eta \mu}{2} - 2p}
= \frac{2}{1 + \frac{1}{12\kappa\sqrt{n}} - \frac{1}{n}}
\leq 4
\end{align*}
We also have
\begin{align*}
C = \frac{8\eta\left(2\kappa + 3L\eta\right)}{m} 
= \frac{\frac{8}{6L\sqrt{n}}\left(2\kappa + \frac{3L}{6L\sqrt{n}}\right)}{m} 
=\frac{2\left(4\kappa\sqrt{n} + 1\right)}{3mnL}
\end{align*}
and
\begin{align*}
\frac{c_2 C}{1 + \frac{\eta \mu}{2} - c_1p}
\leq \frac{\frac{12n\left(4\kappa\sqrt{n} + 1\right)}{3mnL}}{1 + \frac{1}{12\kappa\sqrt{n}} - 2p}
\leq \frac{\frac{12\left(4\kappa\sqrt{n} + 1\right)}{m}}{1 + \frac{1}{12\kappa\sqrt{n}} - \frac{1}{n}}
\leq \frac{24\left(4\kappa\sqrt{n} + 1\right)}{m}
\leq \frac{120\kappa\sqrt{n}}{m}
\end{align*}

Using above inequalities, we have
\begin{align*}
V_t = \Norm{\bz_{t} - \bz^*}^2 + \frac{c_1}{1 + \frac{\eta \mu}{2} - c_1p} \Norm{\bw_{t} - \bz^*}^2 + \frac{c_2}{1 + \frac{\eta \mu}{2} - c_1p}C\vone^\top r_{t}^2.
\end{align*}

Since $\vz_0=\vw_0=\vone\bz_0=\vone\bw_0$, we have
\begin{align*}
 \vone^\top r_0^2 
=  \norm{\vz_0-\vone\bz_0}^2 + \norm{\vw_0-\vone\bw_0}^2 + \eta^2\norm{\vs_0-\vone\bs_0}^2 
=  \frac{1}{36nL^2}\Norm{\vs_0-\vone\bs_0}^2 
\end{align*}
and
\begin{align*}
    \frac{c_2}{1 + \frac{\eta \mu}{2} - c_1p}C\vone^\top r_{t}^2
\leq \frac{120\kappa\sqrt{n}}{36mnL^2}\Norm{\vs_0-\vone\bs_0}^2 
= \frac{10\kappa}{3m\sqrt{n}L^2}\Norm{\vs_0-\vone\bs_0}^2 
\end{align*}
which implies
\begin{align*}
V_0 = & \Norm{\bz_{0} - \bz^*}^2 + \frac{c_1}{1 + \frac{\eta \mu}{2} - c_1p} \Norm{\bw_{0} - \bz^*}^2 + \frac{c_2}{1 + \frac{\eta \mu}{2} - c_1p}C\vone^\top r_{0}^2 \\
\leq & 5\Norm{\bz_{0} - \bz^*}^2 + \frac{10\kappa}{3m\sqrt{n}L^2}\Norm{\vs_0-\vone\bs_0}^2 \\
= & 5\Norm{\bz_{0} - \bz^*}^2 + \frac{10\kappa}{3m\sqrt{n}L^2}\Norm{\vg(\vz_0)-\frac{1}{m}\vone\vone^\top\vg(\vz_0)}^2  \\
\leq & 5\Norm{\bz_{0} - \bz^*}^2 + \frac{10\kappa\rho}{3m\sqrt{n}L^2}\Norm{\vg(\vz_0)-\frac{1}{m}\vone\vone^\top\vg(\vz_0)}^2.
\end{align*}
Then, by setting
\begin{align*}
K_0 \geq \sqrt{\chi}\log\left(\frac{10\kappa}{3m\sqrt{n}L^2\eps}\Norm{\vg(\vz_0)-\frac{1}{m}\vone\vone^\top\vg(\vz_0)}^2 \right) = \fO\left(\sqrt{\chi}\log\left(\frac{\kappa}{mnL^2\eps} \right)\right)
\end{align*}
we have
\begin{align*}
V_0 \leq \Norm{\bz_0-\bz^*}^2 + 5\Norm{\bw_0-\bz^*}^2 + \eps \leq 6\Norm{\bz_0-\bz^*}^2 + \eps.
\end{align*}
Hence, we require
\begin{align*}
T = 6(n+4\kappa\sqrt{n})\log\left(\frac{6\Norm{\bz_0-\bz^*}^2 + \eps}{\eps}\right) 
= \fO\left(\left(n+\kappa\sqrt{n}\right)\log\left(\frac{1}{\eps}\right)\right)    
\end{align*}
to obtain
\begin{align*}
& \BE_t\left[\Norm{\bz_{T}-\bz^*}^2\right]
\leq \BE_t\left[V_T\right] \\
\leq & \left(1-\frac{1}{6(n+4\kappa\sqrt{n})}\right)^TV_0 \\
\leq & \left(1-\frac{1}{6(n+4\kappa\sqrt{n})}\right)^T\left(6\Norm{\bz_0-\bz^*}^2 + \eps\right) \\
\leq & \eps.
\end{align*}
Since each iteration requires $1+np=\fO(1)$ SFO calls in expectation, the total complexity is
\begin{align*}
\fO((1+np)T)=\fO\left(\left(n+\kappa\sqrt{n}\right)\log\left(\frac{1}{\eps}\right)\right)
\end{align*}
in expectation. The number of communication rounds is 
\begin{align*}
KT+K_0 = \fO\left(\left(n+\kappa\sqrt{n}\right)\sqrt{\chi}\log(\kappa n)\log\left(\frac{1}{\eps}\right)\right).
\end{align*}
\end{proof}

\section{The Proof Details for Section \ref{sec:mc-svre-constrained}}\label{appendix:MC-SVRE-constrained}

We first give two lemmas to address the constraint in the problem.

\begin{lem}[{\citet[Lemma 6]{luo2021near}}]\label{lem:proj}
For any $u\in\BR^{d}$ and $v\in\fX\times\fY$, we have
$\inner{\fP_\fZ(u)-u}{\fP_\fZ(u)-v} \leq 0$.
\end{lem}

\begin{lem}\label{lem:optimal}
Under Assumption \ref{asm:smooth} and \ref{asm:cc}, we have $\inner{g(z^*)}{z-z^*}\geq 0$ for any $z\in\fX\times\fY$.
\end{lem}
\begin{proof}
Since the objective function is convex-concave, for any $x\in\fX$ and $y\in\fY$, we have
\begin{align*}
\inner{\nabla_x f(x^*,y^*)}{x-x^*} \geq 0 \qquad \text{and} \qquad
\inner{-\nabla_y f(x^*,y^*)}{y-y^*} \geq 0
\end{align*}
which means $\inner{g(z^*)}{z-z^*}\geq 0$.
\end{proof}

\begin{lem}[{\citet[Lemma 11]{DBLP:conf/nips/YeZL020}}]\label{lem:proj-dc}
For any $\vu\in\BR^{m\times d}$, it holds that
\begin{align*}
\Norm{\fP_{\FZ}\left(\frac{1}{m}\vone\vone^\top\vu\right) - \frac{1}{m}\vone\vone^\top\fP_{\FZ}\left(\vu\right)} \leq \Norm{\vu - \vone\bu}.
\end{align*}
\end{lem}


Then we provide the detailed proofs for theoretical results of MC-SVRE in constrained case.

\subsection{The Proof of Lemma \ref{lem:rel-constrained}}
\begin{proof}
Lemma \ref{lem:FM} means
\begin{align*}
   \bz_{t+1/2} 
= & \frac{1}{m}\vone^\top\BT(\fP(\vz_t'-\eta\vs_t)) 
=  \frac{1}{m}\vone^\top\fP(\vz_t'-\eta\vs_t) \\
= & \fP(\bz'_t - \eta\bs_t)  + \underbrace{\frac{1}{m}\vone^\top\fP(\vz_t'-\eta\vs_t) - \fP(\bz'_t - \eta\bs_t)}_{\Delta_t}  \\
= & \fP(\bz'_t - \eta\bs_t)  + {\Delta_t} 
\end{align*}
and
\begin{align*}
   \bz_{t+1}  
= & \frac{1}{m}\vone^\top\BT(\fP(\vz'_t-\eta\vv_{t+1/2})) 
=  \frac{1}{m}\vone^\top\fP(\vz'_t-\eta\vv_{t+1/2}) \\
= & \fP(\bz'_t - \eta\bv_{t+1/2}) + \underbrace{\frac{1}{m}\vone^\top\fP(\vz'_t-\eta\vs_{t+1/2}) - \fP(\bz'_t - \eta\bs_{t+1/2})}_{\Delta_{t+1/2}} \\
= & \fP(\bz'_t - \eta\bs_{t+1/2}) + \Delta_{t+1/2}.
\end{align*}
which implies
\begin{align*}
\fP(\bz'_t - \eta\bg_t)  =  \bz_{t+1/2} - \Delta_t \qquad\text{and}\qquad
\fP(\bz'_t - \eta\bg_{t+1/2}) =   \bz_{t+1}  - \Delta_{t+1/2}.
\end{align*}
Using Lemma \ref{lem:proj} with $u=\bz'_t-\eta \bg_{t+1/2}$ and $v=\bz^*$, we have
\begin{align*}
    \inner{\bz_{t+1}-\Delta_{t+1/2}-\bz'_t+\eta \bs_{t+1/2}}{\bz_{t+1}-\Delta_{t+1/2}-\bz^*} \leq 0.
\end{align*}
Using Lemma \ref{lem:proj} with $u=\bz'_t-\eta \bg_t$ and $v=\bz_{t+1}$, we have
\begin{align*}
    \inner{\bz_{t+1/2} - \Delta_{t} - \bz'_t + \eta\bs_t}{\bz_{t+1/2} - \Delta_{t} -\bz_{t+1}} \leq 0.
\end{align*}
Summing over above inequalities, we have
\begin{align*}
& \inner{\bz_{t+1}-\bz'_t+\eta \bs_{t+1/2}}{\bz^*-\bz_{t+1}} + \inner{\bz_{t+1}-\bz'_t+\eta \bs_{t+1/2}+\bz_{t+1}-\bz^*}{\Delta_{t+1/2}} - \Norm{\Delta_{t+1/2}}^2 \\
& + \inner{\bz_{t+1/2}- \bz'_t + \eta\bs_t}{\bz_{t+1} - \bz_{t+1/2}} + \inner{\bz_{t+1/2} - \bz'_t + \eta\bs_t+\bz_{t+1/2} -\bz_{t+1}}{\Delta_t} - \Norm{\Delta_t}^2 \geq 0,
\end{align*}
which means
\begin{align}\label{eq:proj-ieq0}
\begin{split}
& \inner{\bz_{t+1}-\bz'_t}{\bz^*-\bz_{t+1}} + \inner{\bz_{t+1/2}- \bz'_t}{\bz_{t+1} - \bz_{t+1/2}}   \\
& + \eta \inner{\bs_{t+1/2}}{\bz^*-\bz_{t+1/2}} + \eta\inner{\bs_{t+1/2}-\bs_t}{\bz_{t+1/2}-\bz_{t+1}} \\
& + \inner{\bz_{t+1}-\bz'_t+\eta \bs_{t+1/2}+\bz_{t+1}-\bz^*}{\Delta_{t+1/2}} \\
& + \inner{\bz_{t+1/2} - \bz'_t + \eta\bs_t+\bz_{t+1/2} -\bz_{t+1}}{\Delta_t}  \geq 0.
\end{split}
\end{align}
In constrained case. It is easy to verify the results of (\ref{eq:1-1}), (\ref{eq:1-2}) and (\ref{eq:1-4}) also hold. Additionally, we also have (\ref{eq:1-3}) due to Lemma~\ref{lem:optimal}.

Plugging results of (\ref{eq:1-1}), (\ref{eq:1-2}), (\ref{eq:1-3}) and (\ref{eq:1-4}) into (\ref{eq:proj-ieq0}), we obtain that
\begin{align}\label{ieq:constraind-zw}
\begin{split}
& \alpha \Norm{\bz_t - \bz^*}^2 + (1 - \alpha) \Norm{\bw_t - \bz^*}^2 - \BE_t\Norm{\bz_{t+1} - \bz^*}^2 - \BE_t\Norm{\bz_{t+1/2} - \bz_{t+1}}^2 \\
& - \alpha \Norm{\bz_{t+1/2} - \bz_t}^2 - (1 - \alpha) \Norm{\bz_{t+1/2} - \bw_t}^2 \\
& - \eta\mu \BE_t\Norm{\bz_{t+1} - \bz^*}^2 + \frac{5\eta\mu}{2} \BE_t\Norm{\bz_{t+1/2} - \bz_{t+1}}^2 + \frac{4L^2\eta}{m\mu}\Norm{\vz_{t+1/2} - \vone\bz_{t+1/2}}^2 + \frac{\eta\mu}{2}\Norm{\bz_{t+1}-\bz^*}^2 \\
& + 6\eta^2L^2\BE_t\Norm{\bz_{t+1/2} - \bw_t}^2 + \frac{1}{2} \BE_t\Norm{{\bz_{t+1/2} - \bz_{t+1}}}^2 \\
& + \frac{6\eta^2L^2}{m}\BE_t\Norm{\vz_{t+1/2}-\vone\bz_{t+1/2}}^2  + \frac{6\eta^2L^2}{m}\Norm{\vw_t-\vone\bw_t}^2 \\
 & + \BE_t\left[\inner{\bz_{t+1}-\bz'_t+\eta \bs_{t+1/2}+\bz_{t+1}-\bz^*}{\Delta_{t+1/2}}\right] \\
& + \BE_t\left[\inner{\bz_{t+1/2} - \bz'_t + \eta\bs_t+\bz_{t+1/2} -\bz_{t+1}}{\Delta_t}\right]\geq 0,
\end{split}
\end{align}
Combining inequality (\ref{ieq:constraind-zw}) with
\begin{align*}
\inner{\bz_{t+1}-\bz'_t+\eta \bs_{t+1/2}+\bz_{t+1}-\bz^*}{\Delta_{t+1/2}}
\leq \Norm{\bz_{t+1}-\bz'_t+\eta \bs_{t+1/2}+\bz_{t+1}-\bz^*}\Norm{\Delta_{t+1/2}}
\end{align*}
and
\begin{align*}
\inner{\bz_{t+1/2} - \bz'_t + \eta\bs_t+\bz_{t+1/2} -\bz_{t+1}}{\Delta_t}
\leq \Norm{\bz_{t+1/2} - \bz'_t + \eta\bs_t+\bz_{t+1/2} -\bz_{t+1}}\Norm{\Delta_t},
\end{align*}
we have
\begin{align*} 
\begin{split}
& \left(1 + \frac{\eta \mu}{2} - c_1p\right) \BE_t \Norm{\bz_{t+1} - \bz^*}^2 + c_1 \E_{t}\Norm{\bw_{t+1} - \bz^*}^2  \\
\leq & \left(1 - p\right) \Norm{\bz_t - \bz^*}^2 + \left(p + c_1\left(1 - p\right)\right) \Norm{\bw_t - \bz^*}^2  \\
& + \frac{2 L\eta\left(2\kappa + 3L\eta\right)}{m}\left(\BE_t\Norm{\vz_{t+1/2} - \vone\bz_{t+1/2}}^2 + \Norm{\vw_t - \vone\bw_t}^2\right) \\
& + \BE_t\left[\Norm{\bz_{t+1}-\bz'_t+\eta \bs_{t+1/2}+\bz_{t+1}-\bz^*}\Norm{\Delta_{t+1/2}}\right]
 + \BE_t\left[\Norm{\bz_{t+1/2} - \bz'_t + \eta\bs_t+\bz_{t+1/2} -\bz_{t+1}}\Norm{\Delta_t}\right] \\
= & \left(1 - p\right) \Norm{\bz_t - \bz^*}^2 + \left(p + c_1\left(1 - p\right)\right) \Norm{\bw_t - \bz^*}^2 + \zeta_t.
\end{split}
\end{align*}
\end{proof}

\subsection{The Proof of Lemma \ref{lem:zeta}}

We first give a lemma for consensus error.

\begin{lem}\label{lem:constrained-err}
Suppose Assumption~\ref{asm:smooth} and \ref{asm:constrained} holds and we have $\Norm{\vs_0 - \vone\bs_0}\leq\delta'$ for some $\delta'>0$. 
Then Algorithm \ref{alg:dsvre} with $\eta=1/\left(6\sqrt{n}L\right)$ and
\begin{align*}
K \geq \sqrt{\chi}\log\left(\frac{2\left(\sqrt{m}LD+\delta'\right)}{\delta'}\right),
\end{align*}
holds that {\small$\Norm{\vz_t - \vone\bz_t}\leq\delta'$}, {\small$\Norm{\vz_{t+1/2} - \vone\bz_{t+1/2}}\leq\delta'$}, {\small$\Norm{\vw_t - \vone\bw_t}\leq\delta'$}, {\small$\Norm{\vs_t - \vone\bs_t}\leq\delta'/\eta$} and {\small$\Norm{\vs_{t+1/2} - \vone\bs_{t+1/2}}\leq\delta'/\eta$} for any $t\geq 0$.
\end{lem}
\begin{proof}
We prove this lemma by induction. 
It is obvious the statement holds for $t=0$.
Suppose we have $\Norm{\vz_t - \vone\bz_t}\leq\delta'$, $\Norm{\vw_t - \vone\bw_t}\leq\delta'$, $\Norm{\vs_t - \vone\bs_t}\leq\delta'/\eta$ for $t\geq 1$, then
Lemma \ref{lem:FM} and
Lemma \ref{lem:errs12} means
\begin{align*}
   & \Norm{\vs_{t+1/2} - \vone\bs_{t+1/2}} \\
\leq & \rho\left(\Norm{\vs_t - \vone\bs_t} + L\Norm{\vz_{t+1/2} - \vone\bz_{t+1/2}} + L\sqrt{m}\Norm{\bz_{t+1/2} - \bw_t} + L\Norm{\vw_t - \vone\bw_t}\right) \\
\leq & \rho\left(\Norm{\vs_t - \vone\bs_t} + L\sqrt{m}D + L\sqrt{m}D + L\sqrt{m}D\right) \\
\leq & \rho\left(\delta'/\eta + 3\sqrt{m}LD\right) 
\leq \delta'/\eta.
\end{align*}
We also have
\begin{align*}
   & \Norm{\vs_{t+1} - \vone\bs_{t+1}} \\
= & \Norm{\BT(\vs_t + \vg(\vw_{t+1}) - \vg(\vw_t)) - \frac{1}{m}\vone\vone^\top\BT(\vs_t + \vg(\vw_{t+1}) - \vg(\vw_t))} \\
\leq & \rho\Norm{\vs_t + \vg(\vw_{t+1}) - \vg(\vw_t) - \frac{1}{m}\vone\vone^\top(\vs_t + \vg(\vw_{t+1}) - \vg(\vw_t))} \\
\leq & \rho\Norm{\vs_t - \vone\bs_t} + \rho\Norm{\vg(\vw_{t+1}) - \vg(\vw_t) - \frac{1}{m}\vone\vone^\top(\vg(\vw_{t+1}) - \vg(\vw_t))} \\
\leq & \rho\Norm{\vs_t - \vone\bs_t} + \rho\Norm{\vg(\vw_{t+1}) - \vg(\vw_t)} \\
\leq & \rho\Norm{\vs_t - \vone\bs_t} + \rho L\Norm{\vw_{t+1} - \vw_t} \\
\leq & \rho(\delta'/\eta + \sqrt{m}LD) \leq \delta'/\eta.
\end{align*}
We modify the proof of Lemma \ref{lem:diff-z12} as follows
\begin{align*}
   & \Norm{\vz_{t+1/2} - \vone\bz_{t+1/2}} \\
= & \Norm{\BT(\fP_\FZ(\vz'_t - \eta\vs_t)) - \frac{1}{m}\vone\vone^\top\BT(\fP_\FZ(\vz'_t - \eta\vs_t))} \\
\leq  & \rho \Norm{\fP_\FZ(\vz'_t - \eta\vs_t)  - \frac{1}{m}\vone\vone^\top\fP_\FZ\left(\vz'_t - \eta\vs_t\right)} \\
\leq  & \rho \Norm{\fP_\FZ(\vz'_t - \eta\vs'_t)  - \fP_\FZ\left(\vone(\bz'_t - \eta\bs_t)\right)} + \rho \Norm{\fP_\FZ\left(\vone(\bz'_t - \eta\bs_t)\right)  - \frac{1}{m}\vone\vone^\top\fP_\FZ\left(\vz'_t - \eta\vs_t\right)} \\
\leq  & \rho \Norm{\vz'_t - \eta\vs_t  - \vone(\bz'_t - \eta\bs_t)} + \rho \Norm{(\vz'_t - \eta\vs_t) - \vone(\bz'_t - \eta\bs_t)} \\
\leq  & 2\rho \Norm{\vz'_t - \vone\bz'_t} + 2\rho\eta\Norm{\vs_t - \vone\bs_t} \\
\leq  & \rho \left(2\sqrt{m}D + 2\delta'\right) \leq \delta', 
\end{align*}
where the first inequality use Lemma \ref{lem:FM} and third inequality use the non-expansiveness of projection and Lemma \ref{lem:proj-dc} with $\bu=\vz_t-\eta\vv_t$.

Similarly, we modify the proof of Lemma \ref{lem:errz1} as follows
\begin{align*}
   & \Norm{\vz_{t+1} - \vone\bz_{t+1}} \\
= & \Norm{\BT(\fP_\FZ(\vz'_t - \eta\vs_{t+1/2})) - \frac{1}{m}\vone\vone^\top\BT(\fP_\FZ(\vz'_t - \eta\vs_{t+1/2}))} \\
\leq  & \rho \Norm{\fP_\FZ(\vz'_t - \eta\vs_{t+1/2})  - \frac{1}{m}\vone\vone^\top\fP_\FZ\left(\vz'_t - \eta\vs_{t+1/2}\right)} \\
\leq  & \rho \Norm{\fP_\FZ(\vz'_t - \eta\vs_{t+1/2})  - \fP_\FZ\left(\vone(\bz'_t - \eta\bs_{t+1/2})\right)} + \rho \Norm{\fP_\FZ\left(\vone(\bz'_t - \eta\bs_{t+1/2})\right)  - \frac{1}{m}\vone\vone^\top\fP_\FZ\left(\vz'_t - \eta\vs_{t+1/2}\right)} \\
\leq  & \rho \Norm{\vz'_t - \eta\vs_{t+1/2}  - \vone(\bz'_t - \eta\bs_{t+1/2})} + \rho \Norm{(\vz'_t - \eta\vs_{t+1/2}) - \vone(\bz'_t - \eta\bs_{t+1/2})} \\
\leq  & 2\rho\left(\Norm{\vz'_t - \vone\bz'_t} + \eta\Norm{\vs_{t+1/2} - \vone\bs_{t+1/2}}\right) \\
\leq  & \rho\left(2\sqrt{m}D + 2\delta'\right) \leq \delta'.
\end{align*}
Finally, we have
\begin{align*}
 \Norm{\vw_{t+1} - \vone\bw_{t+1}} 
\leq  \max\left\{\Norm{\vz_{t+1} - \vone\bz_{t+1}}, \Norm{\vw_t - \vone\bw_t}\right\} \leq \delta'.
\end{align*}
\end{proof}

Then we provide the proof of of Lemma \ref{lem:zeta}.
\begin{proof}
Using the non-expansiveness of projection, the update rule $\vz'_t=\alpha\vz_t+(1-\alpha)\vw_t$ and Lemma \ref{lem:constrained-err}, we have
\begin{align*}
\Norm{\Delta_t} = & \Norm{\frac{1}{m}\vone^\top\fP_\FZ(\vz_t'-\eta\vs_t) - \fP_\fZ(\bz_t' - \eta\bs_t)}  \\
= & \sqrt{\frac{1}{m}\sum_{i=1}^m\Norm{\fP_\FZ(\vz'_t(i)-\eta\vs_t(i)) - \fP_\fZ(\bz^{\prime\top}_t - \eta\bs_t^\top)}^2} \\
\leq  & \sqrt{\frac{1}{m}\sum_{i=1}^m\Norm{(\vz'_t(i)-\eta\vs_t(i)) - (\bz^{\prime\top}_t - \eta\bs_t^\top)}^2} \\
\leq  & \sqrt{\frac{1}{m}\sum_{i=1}^m 2\left(\Norm{\vz'_t(i)-\bz^{\prime\top}_t}^2 + \eta^2\Norm{\vs_t(i) - \bs_t^\top}^2\right)} \\
\leq  & \sqrt{\frac{1}{m}}\sqrt{2\left(\Norm{\vz'_t-\vone\bz'_t}^2 + \eta^2\Norm{\vs_t - \vone\bs_t}^2\right)} \\
\leq  & 2\sqrt{\frac{1}{m}}\left(\Norm{\vz_t-\vone\bz_t}^2 + \Norm{\vw_t-\vone\bw_t}^2 + \eta\Norm{\vs_t - \vone\bs_t}\right) \\
\leq &  \frac{6\delta'}{\sqrt{m}}
\end{align*}
Similarly, we obtain
\begin{align*}
\Norm{\Delta_{t+1/2}} \leq \frac{6\delta'}{\sqrt{m}}.
\end{align*}
We also have
\begin{align*}
    & \Norm{\bz_{t+1/2} - \bz'_t + \eta\bs_t+\bz_{t+1/2} -\bz_{t+1}} \\
\leq & \Norm{\bz_{t+1/2} - \bz'_t} + \Norm{\bz_{t+1/2} -\bz_{t+1}} + \eta\Norm{\bs_t} \\
\leq & 2D + \eta\Norm{\frac{1}{m}\vone^\top\vg(\vw_t)} \\
= & 2D + \eta\Norm{\frac{1}{m}\sum_{i=1}^m g_i(\vw_t(i))} \\
\leq & 2D + \eta\sqrt{\frac{1}{m}\sum_{i=1}^m \Norm{g_i(\vw_t(i))}^2} \\
\leq & 2D + \eta\sqrt{\frac{2}{m}\sum_{i=1}^m \left(\Norm{g_i(\vw_t(i))-g_i(z^*)}^2 + \Norm{g_i(z^*)}^2\right)}  \\
\leq & 2D + \eta\sqrt{\frac{2}{m}\sum_{i=1}^m \left(L^2\Norm{\vw_t(i)-z^*}^2 + \Norm{g_i(z^*)}^2\right)}  \\
\leq & 2D + \frac{1}{6\sqrt{n}L}\sqrt{2L^2D^2 + \frac{2}{m}\sum_{i=1}^m\Norm{g_i(z^*)}^2} \\
\triangleq & C_1.
\end{align*}
Similarly, we have
\begin{align*}
    & \Norm{\bz_{t+1}-\bz'_t+\eta \bs_{t+1/2}+\bz_{t+1}-\bz^*} \\
\leq & \Norm{\bz_{t+1}-\bz'_t} + \Norm{\bz_{t+1}-\bz^*} + \eta\Norm{\bs_{t+1/2}} \\
\leq & 2D + \eta\Norm{\bv_{t+1/2}} \\
\leq & 2D + \eta\Norm{\frac{1}{m}\sum_{i=1}^m\left(g_i(\vw_t(i)) + g_{i,j_i}(\vz_{t+1/2}(i)) - g_{i,j_i}(\vw_t(i))\right)} \\
\leq & 2D + \eta\sqrt{\frac{1}{m}\sum_{i=1}^m\Norm{g_i(\vw_t(i)) + g_{i,j_i}(\vz_{t+1/2}(i)) - g_{i,j_i}(\vw_t(i))}^2} \\
= & 2D + \eta\sqrt{\frac{1}{m}\sum_{i=1}^m\Norm{g_i(\vw_t(i)) - g_i(z^*) + g_{i,j_i}(\vz_{t+1/2}(i)) - g_{i,j_i}(\vw_t(i)) + g_i(z^*)}^2} \\
\leq & 2D + \eta\sqrt{\frac{3}{m}\sum_{i=1}^m\left(\Norm{g_i(\vw_t(i)) - g_i(z^*)}^2 + \Norm{g_{i,j_i}(\vz_{t+1/2}(i)) - g_{i,j_i}(\vw_t(i))}^2 + \Norm{g_i(z^*)}^2\right)} \\
\leq & 2D + \eta\sqrt{\frac{3}{m}\sum_{i=1}^m\left(L^2\Norm{\vw_t(i) - z^*}^2 + L^2\Norm{\vz_{t+1/2}(i) - \vw_t(i)}^2 + \Norm{g_i(z^*)}^2\right)} \\
\leq & 2D + \frac{1}{6\sqrt{n}L}\sqrt{6L^2D^2+\frac{3}{m}\sum_{i=1}^m\Norm{g_i(z^*)}^2} \\
\triangleq & C_{1/2}.
\end{align*}
Plunging all above results into the expression of $\zeta_t$, we have
\begin{align*}
\zeta_t \leq  \frac{4L\eta\left(2\kappa + 3L\eta\right)\delta'^2}{m} + \frac{6(C_1+C_{1/2})\delta'}{\sqrt{m}}.
\end{align*}
\end{proof}

\subsection{The Proof of Theorem \ref{thm:scsc-constrained}}
\begin{proof}
The setting of $K_0$ means $\Norm{\vs_0 - \vone\bs_0}\leq\delta'$. 
Then Lemma \ref{lem:rel-constrained} and Lemma \ref{lem:zeta} imply
\begin{align*}
& \Norm{\bz_{t+1} - \bz^*}^2 + \frac{c_1}{1 + \frac{\eta \mu}{2} - c_1p} \Norm{\bw_{t+1} - \bz^*}^2 \\
\leq & \frac{1 - p}{1 + \frac{\eta \mu}{2} - c_1p} \Norm{\bz_t - \bz^*}^2 + \frac{p + c_1(1 - p)}{1 + \frac{\eta \mu}{2} - c_1p} \Norm{\bw_t - \bz^*}^2 + \frac{\zeta_t}{1 + \frac{\eta \mu}{2} - c_1p} \\
\leq & \left(1-\frac{1}{6(n+4\kappa\sqrt{n})}\right)\left(\Norm{\bz_{t} - \bz^*}^2 + \frac{c_1}{1 + \frac{\eta \mu}{2} - c_1p} \Norm{\bw_{t} - \bz^*}^2\right) \\
& + \frac{1}{1 + \frac{\eta \mu}{2} - c_1p}\left(\frac{4L\eta\left(2\kappa + 3L\eta\right)\delta'^2}{m} + \frac{6(C_1+C_{1/2})\delta'}{\sqrt{m}}\right),
\end{align*}
where the last step follows the proof of Lemma \ref{lem:scsc-unconstrained}.
Recall that
\begin{align*}
c_1 =\frac{2\eta \mu + 4 p}{\eta \mu + 4 p} \leq 2 
\qquad\text{and}\qquad
    \frac{c_1}{1 + \frac{\eta \mu}{2} - c_1p}
\leq \frac{2}{1 + \frac{\eta \mu}{2} - 2p}
= \frac{2}{1 + \frac{1}{12\kappa\sqrt{n}} - \frac{1}{2}}
\leq 4.
\end{align*}
Then we have
\begin{align*}
& \Norm{\bz_T - \bz^*}^2 + \frac{c_1}{1 + \frac{\eta \mu}{2} - c_1p} \Norm{\bw_T - \bz^*}^2 \\
\leq & \left(1-\frac{1}{6(n+4\kappa\sqrt{n})}\right)^T\left(\Norm{\bz_{0} - \bz^*}^2 + \frac{c_1}{1 + \frac{\eta \mu}{2} - c_1p} \Norm{\bw_{0} - \bz^*}^2\right) \\
& + \frac{6(n+4\kappa\sqrt{n})}{1 + \frac{\eta \mu}{2} - c_1p}\left(\frac{4L\eta\left(2\kappa + 3L\eta\right)\delta'^2}{m} + \frac{6(C_1+C_{1/2})\delta'}{\sqrt{m}}\right) \\
\leq & 5\left(1-\frac{1}{6(n+4\kappa\sqrt{n})}\right)^T\Norm{\bz_{0} - \bz^*}^2 + 12(n+4\kappa\sqrt{n})\left(\frac{4L\eta\left(2\kappa + 3L\eta\right)\delta'^2}{m} + \frac{(C_1+C_{1/2})\delta'}{\sqrt{m}}\right). 
\end{align*}
The value of $T$ and $\delta'$ means
we have
\begin{align*}
5\left(1-\frac{1}{6(n+4\kappa\sqrt{n})}\right)^T\Norm{\bz_{0} - \bz^*}^2 \leq \frac{\eps}{2} 
\end{align*}
and
\begin{align*}
12(n+4\kappa\sqrt{n})\left(\frac{4L\eta\left(2\kappa + 3L\eta\right)\delta'^2}{m} + \frac{(C_1+C_{1/2})\delta'}{\sqrt{m}}\right) \leq \frac{\eps}{2},
\end{align*}
which implies
\begin{align*}
\BE_t\Norm{\bz_T - \bz^*}^2
\leq \BE_t\Norm{\bz_T - \bz^*}^2 + \frac{c_1}{1 + \frac{\eta \mu}{2} - c_1p}
\leq \frac{\eps}{2} + \frac{\eps}{2} = \eps.
\end{align*}
Since each iteration requires $1+np=\fO(1)$ SFO calls in expectation, the total complexity is
\begin{align*}
\fO((1+np)T)=\fO\left(\left(n+\kappa\sqrt{n}\right)\log\left(\frac{1}{\eps}\right)\right)
\end{align*}
in expectation. The number of communication rounds is 
\begin{align*}
KT+K_0 = \fO\left(\left(n+\kappa\sqrt{n}\right)\sqrt{\chi}\log\left(\frac{\kappa n}{\eps}\right)\log\left(\frac{1}{\eps}\right)\right).
\end{align*}
\end{proof}

\subsection{The Proof of Lemma \ref{thm:cc-constrained}}
\begin{proof}
Since the objective function is non-strongly-convex and non-strongly-concave, we first modify (\ref{eq:1-3}) as follows
\begin{align}\label{eq:1-3b}
\begin{split}
& 2 \BE_t \left[\inner{\bs_{t+1/2}}{\bz^* - \bz_{t+1/2}} \right] \\
= & 2\inner{\bg(\bz_{t+1/2})}{\bz^* - \bz_{t+1/2}} + 2\inner{\BE_t[\bs_{t+1/2}] - \bg(\bz_{t+1/2})}{\bz^* - \bz_{t+1/2}} \\
\leq & 2 \inner{\bg(\bz^*)}{\bz^* - \bz_{t+1/2}} + 2\inner{\BE_t[\bs_{t+1/2}] - \bg(\bz_{t+1/2})}{\bz^* - \bz_{t+1/2}} \\
\leq & \beta\Norm{\BE_t[\bs_{t+1/2}] - \bg(\bz_{t+1/2})}^2 + \frac{1}{\beta}\Norm{\bz_{t+1/2}-\bz^*}^2 
\end{split}
\end{align}
where we use Lemma~\ref{lem:mm} and \ref{lem:optimal}.

Plugging results of (\ref{eq:1-1}), (\ref{eq:1-2}), (\ref{eq:1-3b}) and (\ref{eq:1-4}) into (\ref{eq:proj-ieq0}), we obtain that
\begin{align}\label{ieq:cc-relation1}
\begin{split}
& \alpha \Norm{\bz_t - \bz^*}^2 + (1 - \alpha) \Norm{\bw_t - \bz^*}^2 - \BE_t\Norm{\bz_{t+1} - \bz^*}^2 - \BE_t\Norm{\bz_{t+1/2} - \bz_{t+1}}^2 \\
& - \alpha \Norm{\bz_{t+1/2} - \bz_t}^2 - (1 - \alpha) \Norm{\bz_{t+1/2} - \bw_t}^2 + \beta\eta\Norm{\BE_t[\bs_{t+1/2}] - g(\bz_{t+1/2})}^2 + \frac{\eta}{\beta}\Norm{\bz_{t+1/2}-\bz^*}^2  \\
& + 6\eta^2L^2\BE_t\Norm{\bz_{t+1/2} - \bw_t}^2 + \frac{1}{2} \BE_t\Norm{{\bz_{t+1/2} - \bz_{t+1}}}^2 \\
& + \frac{6\eta^2L^2}{m}\BE_t\Norm{\vz_{t+1/2}-\vone\bz_{t+1/2}}^2  + \frac{6\eta^2L^2}{m}\BE_t\Norm{\vw_t-\vone\bw_t}^2 \\
 & + \BE_t\left[\inner{\bz_{t+1}-\bz'_t+\eta \bs_{t+1/2}+\bz_{t+1}-\bz^*}{\Delta_{t+1/2}} 
+ \inner{\bz_{t+1/2} - \bz'_t + \eta\bs_t+\bz_{t+1/2} -\bz_{t+1}}{\Delta_t}\right] \\
& \geq 2\eta\BE_t\left[\inner{\bg(\bz_{t+1/2})}{\bz_{t+1/2}-\bz^*}\right],
\end{split}
\end{align}
Using the update rule of $\vw_t$, we have
\begin{align}\label{ieq:cc-relation2}
\BE_t\Norm{\bw_{t+1}-\bz^*}^2  = (1-p)\BE_t\Norm{\bw_{t}-\bz^*}^2 + p\BE_t\Norm{\bz_{t+1}-\bz^*}^2.
\end{align}
Connecting the inequalities (\ref{ieq:cc-relation1}) and (\ref{ieq:cc-relation2}) implies
\begin{align*}
&  2\eta\BE_t\inner{\bg(\bz_{t+1/2})}{\bz_{t+1/2}-\bz^*} + \BE_t\Norm{\bw_{t+1}-\bz^*} \\
\leq & (1-p) \Norm{\bz_t - \bz^*}^2 - \BE_t\Norm{\bz_{t+1} - \bz^*}^2 + p\Norm{\bw_t - \bz^*}^2 + \frac{\eta}{\beta}\Norm{\bz_{t+1/2}-\bz^*}^2   
+ \beta\eta\Norm{\BE_t[\bs_{t+1/2}] - g(\bz_{t+1/2})}^2   \\
& + \frac{6\eta^2L^2}{m}\BE_t\Norm{\vz_{t+1/2}-\vone\bz_{t+1/2}}^2  + \frac{6\eta^2L^2}{m}\BE_t\Norm{\vw_t-\vone\bw_t}^2 \\
 & + \BE_t\left[\inner{\bz_{t+1}-\bz'_t+\eta \bs_{t+1/2}+\bz_{t+1}-\bz^*}{\Delta_{t+1/2}} 
 + \inner{\bz_{t+1/2} - \bz'_t + \eta\bs_t+\bz_{t+1/2} -\bz_{t+1}}{\Delta_t}\right] \\
& + p\Norm{\bz_{t+1}-\bz^*}^2 + (1-p)\Norm{\bw_{t}-\bz^*}^2 \\
\leq & (1-p) \Norm{\bz_t - \bz^*}^2 - \BE_t\Norm{\bz_{t+1} - \bz^*}^2 + p\Norm{\bw_t - \bz^*}^2 + \frac{\eta D^2}{\beta}  
 + \frac{\beta L^2\eta}{m}\Norm{\vz_{t+1/2} - \vone\bz_{t+1/2}}^2  \\
& + \frac{6\eta^2L^2}{m}\BE_t\Norm{\vz_{t+1/2}-\vone\bz_{t+1/2}}^2  + \frac{6\eta^2L^2}{m}\BE_t\Norm{\vw_t-\vone\bw_t}^2 \\
 & + \BE_t\left[\Norm{\bz_{t+1}-\bz'_t+\eta \bs_{t+1/2}+\bz_{t+1}-\bz^*}\Norm{\Delta_{t+1/2}} 
 + \Norm{\bz_{t+1/2} - \bz'_t + \eta\bs_t+\bz_{t+1/2} -\bz_{t+1}}\Norm{\Delta_t}\right] \\
& + p\Norm{\bz_{t+1}-\bz^*}^2 + (1-p)\Norm{\bw_{t}-\bz^*}^2,
\end{align*}
that is
\begin{align}\label{ieq:cc-relation3}
\begin{split}
&  2\eta\BE_t\inner{\bg(\bz_{t+1/2})}{\bz_{t+1/2}-\bz^*} \\
\leq & (1-p)\left(\Norm{\bz_t - \bz^*}^2 - \BE_t\Norm{\bz_{t+1} - \bz^*}^2\right) + \left(\Norm{\bw_t - \bz^*}^2-\BE_t\Norm{\bw_{t+1} - \bz^*}^2\right) +\zeta'_t
\end{split}
\end{align}
where 
\begin{align*}
\zeta'_t 
= & \frac{\eta D^2}{\beta} + \frac{6(\eta^2+\beta\eta)L^2}{m}\BE_t\Norm{\vz_{t+1/2}-\vone\bz_{t+1/2}}^2  + \frac{6\eta^2L^2}{m}\BE_t\Norm{\vw_t-\vone\bw_t}^2 \\
 & + \BE_t\left[\Norm{\bz_{t+1}-\bz'_t+\eta \bs_{t+1/2}+\bz_{t+1}-\bz^*}\Norm{\Delta_{t+1/2}} 
 + \Norm{\bz_{t+1/2} - \bz'_t + \eta\bs_t+\bz_{t+1/2} -\bz_{t+1}}\Norm{\Delta_t}\right] \\
\leq & \frac{\eta D^2}{\beta} + \frac{6(2\eta^2+\beta\eta )L^2\delta'^2}{m}  + \frac{6\left(C_1 + C_{1/2}\right)\delta'}{\sqrt{m}}.
\end{align*}
The upper bound of $\zeta'_t$ follows the proof of Lemma \ref{lem:zeta}.

Summing over (\ref{ieq:cc-relation3}) with $t=0,\dots,T-1$, we obtain
\begin{align*}
\begin{split}
& \BE_t\left[f(\hat x,y^*)-f(x^*,\hat y)\right] \\
\leq &  \frac{1}{T}\sum_{t=0}^{T-1}\BE_t\left[f(\bx_{t+1/2},y^*)-f(x^*,\by_{t+1/2})\right]  \\
\leq &  \frac{1}{T}\sum_{t=0}^{T-1}\BE_t\inner{g(\bz_{t+1/2})}{\bz_{t+1/2}-\bz^*}  \\
\leq & \frac{(1-p)\left(\Norm{\bz_0 - \bz^*}^2-\BE_T\Norm{\bz_T - \bz^*}^2\right) + \Norm{\bw_0 - \bz^*}^2 - \BE_T\Norm{\bw_T - \bz^*}^2 + \sum_{t=0}^{T-1}\zeta'_t}{2\eta T} \\
\leq & \frac{\Norm{\bz_0 - \bz^*}^2}{\eta T} + \frac{1}{2\eta}\left(\frac{\eta D^2}{\beta} + \frac{6(2\eta^2+\beta\eta )L^2\delta'^2}{m}  + \frac{6\left(C_1 + C_{1/2}\right)\delta'}{\sqrt{m}}\right) \\
\leq & \frac{\eps}{2}+ \frac{\eps}{4} + \frac{\eps}{8} + \frac{\eps}{8} = \eps.
\end{split}
\end{align*}
where the first inequality use Jensen's inequality; the second inequality use the objective function is convex-concave; the third inequality use the upper bound of $\zeta'_t$; and the last one is based on the value of parameter settings.

Since each iteration requires $1+np=\fO(1)$ SFO calls in expectation and we the algorithm needs to compute the full gradient at first, the total SFO complexity is
\begin{align*}
\fO(n + (1+np)T) = \fO\left(n + \frac{\sqrt{n}L}{\eps}\right)
\end{align*}
in expectation. The number of communication rounds is 
\begin{align*}
KT+K_0 = \fO\left(\frac{\sqrt{n\chi}L}{\eps} \log\left(\frac{nL}{\eps}\right)\right).
\end{align*}
\end{proof}

\section{The Proof Details for Section \ref{sec:MC-EG}}\label{appendix:MCEG}

This section provide the detailed proofs for theoretical results of MC-EG.

\subsection{The Proof of Lemma \ref{lem:eg-relation}}
\begin{proof}
Using the fact $2\inner{a}{b} = \Norm{a + b}^2 - \Norm{a}^2 - \Norm{b}^2$, we have
\begin{align}
 2\inner{\bz_t-\bz_{t+1}}{\bz_{t+1}-\bz^*} 
= & \Norm{\bz_t-\bz^*}^2 - \Norm{\bz_t-\bz_{t+1}}^2 - \Norm{\bz_{t+1}-\bz^*}^2, \label{ieq:eg-00} 
\end{align}
and
\begin{align}
  2\inner{\bz_t-\bz_{t+1/2}}{\bz_{t+1/2}-\bz_{t+1}} 
= & \Norm{\bz_t-\bz_{t+1}}^2 - \Norm{\bz_t-\bz_{t+1/2}}^2 - \Norm{\bz_{t+1/2}-\bz_{t+1}}^2. \label{ieq:eg-01}
\end{align}
Hence, we have
\begin{align}\label{ieq:eg-1}
\begin{split}
& 2\eta\inner{\bs_{t+1/2}}{\bz_{t+1/2}-\bz^*} \\
= & 2\eta\inner{\bs_{t+1/2}}{\bz_{t+1}-\bz^*} + 2\eta\inner{\bs_{t+1/2}}{\bz_{t+1/2}-\bz_{t+1}} \\
= & 2\inner{\bz_t-\bz_{t+1}}{\bz_{t+1}-\bz^*} 
+ 2\inner{\bz_t-\bz_{t+1/2}}{\bz_{t+1/2}-\bz_{t+1}} 
+ 2\inner{\bz_{t+1/2}-\bz_{t+1}}{\bz_{t+1/2}-\bz_{t+1}} \\
= & \Norm{\bz_t-\bz^*}^2 - \Norm{\bz_t-\bz_{t+1}}^2 - \Norm{\bz_{t+1}-\bz^*}^2 + \Norm{\bz_t-\bz_{t+1}}^2 - \Norm{\bz_t-\bz_{t+1/2}}^2 - \Norm{\bz_{t+1/2}-\bz_{t+1}}^2 \\
& + 2\eta\inner{\bs_{t+1/2}-\bs_t}{\bz_{t+1/2}-\bz_{t+1}} \\
\leq & \Norm{\bz_t-\bz^*}^2 - \Norm{\bz_{t+1}-\bz^*}^2 
 - \Norm{\bz_t-\bz_{t+1/2}}^2 - \Norm{\bz_{t+1/2}-\bz_{t+1}}^2 \\
& + 2\eta^2\Norm{\bs_{t+1/2}-\bs_t}^2 + \frac{1}{2}\Norm{\bz_{t+1/2}-\bz_{t+1}}^2 \\
\leq & \Norm{\bz_t-\bz^*}^2 - \Norm{\bz_{t+1}-\bz^*}^2 
- \frac{1}{2}\Norm{\bz_{t+1/2}-\bz_{t+1}}^2
 -(1-6\eta^2L^2)\Norm{\bz_{t+1/2}-\bz_t}^2 \\
 & + \frac{6L^2\eta^2}{m}\left(\Norm{\vz_{t+1/2} - \vone\bz_{t+1/2}}^2 + \Norm{\vz_t - \vone\bz_t}^2\right),
\end{split}
\end{align}
where the second equality is based on the update rule; the third one is based on (\ref{ieq:eg-00}), (\ref{ieq:eg-01}) and Lemma \ref{lem:FM}; the inequality is due to $2\inner{a}{b} \leq \Norm{a}^2+\Norm{b}^2$; the last step is because of
\begin{align}\label{ieq:eg-3}
\begin{split}
      &  \Norm{\bs_{t+1/2}-\bs_t}^2 \\
\leq & 3\Norm{\bs_{t+1/2}-\bg(\bz_{t+1/2})}^2 + 3\Norm{\bg(\bz_{t+1/2})-\bg(\bz_t)}^2 + 3\Norm{\bg(\bz_t)-\bs_t}^2 \\
\leq &  3\left(\frac{L^2}{m}\Norm{\vz_{t+1/2} - \vone\bz_{t+1/2}}^2 
+  L^2\Norm{\bz_{t+1/2}-\bz_t}^2  + \frac{L^2}{m}\Norm{\vz_t - \vone\bz_t}^2\right).
\end{split}
\end{align}
We also have
\begin{align}\label{ieq:eg-2}
\begin{split}
& 2\eta\inner{\bs_{t+1/2}}{\bz_{t+1/2}-\bz^*} \\
= & 2\eta\inner{\bg(\bz_{t+1/2})}{\bz_{t+1/2}-\bz^*} + 2\eta\inner{\bs_{t+1/2} - \bg(\bz_{t+1/2})}{\bz_{t+1/2}-\bz^*}  \\
\geq & 2\eta\inner{\bg(\bz^*)}{\bz_{t+1/2}-\bz^*} + 2\eta\mu\Norm{\bz_{t+1/2}-\bz^*}^2 - \frac{4\eta}{\mu}\Norm{\bs_{t+1/2} - \bg(\bz_{t+1/2})}^2 - \frac{\mu\eta}{4}\Norm{\bz_{t+1/2}-\bz^*}^2 \\
\geq & \eta\mu\Norm{\bz_t-\bz^*}^2 - 2\eta\mu\Norm{\bz_t-\bz_{t+1/2}}^2 - \frac{4\eta}{\mu}\Norm{\bs_{t+1/2} - \bg(\bz_{t+1/2})}^2 - \frac{\mu\eta}{4}\Norm{\bz_{t+1/2}-\bz^*}^2,
\end{split}
\end{align}
where the first inequality follows Lemma \ref{lem:mm} and Young's inequality; the second inequality is according to Lemma~\ref{lem:optimal}.

Connecting inequalities (\ref{ieq:eg-1}) and (\ref{ieq:eg-2}), we have
\begin{align*}
& \eta\mu\Norm{\bz_t-\bz^*}^2 - 2\eta\mu\Norm{\bz_t-\bz_{t+1/2}}^2  \\
\leq & 2\eta\inner{\bs_{t+1/2}}{\bz_{t+1/2}-\bz^*} + \frac{4\eta}{\mu}\Norm{\bs_{t+1/2} - \bg(\bz_{t+1/2})}^2 + \frac{\mu\eta}{4}\Norm{\bz_{t+1/2}-\bz^*}^2  \\
\leq & \Norm{\bz_t-\bz^*}^2 - \Norm{\bz_{t+1}-\bz^*}^2 - \frac{1}{2}\Norm{\bz_{t+1/2}-\bz_{t+1}}
  -(1-6\eta^2L^2)\Norm{\bz_{t+1/2}-\bz_t}^2 \\
 & + \frac{6L^2\eta^2}{m}\left(\Norm{\vz_{t+1/2} - \vone\bz_{t+1/2}}^2 + \Norm{\vz_t - \vone\bz_t}^2\right) \\
& + \frac{4\eta}{\mu}\Norm{\bs_{t+1/2} - g(\bz_{t+1/2})}^2  + \frac{\mu\eta}{2}\Norm{\bz_{t+1/2}-\bz_t}^2 + \frac{\mu\eta}{2}\Norm{\bz_t-\bz^*}^2 \\
\leq & \Norm{\bz_t-\bz^*}^2 - \Norm{\bz_{t+1}-\bz^*}^2 - \frac{1}{2}\Norm{\bz_{t+1/2}-\bz_{t+1}}
  -(1-6\eta^2L^2)\Norm{\bz_{t+1/2}-\bz_t}^2 \\
 & + \frac{6L^2\eta^2}{m}\left(\Norm{\vz_{t+1/2} - \vone\bz_{t+1/2}}^2 + \Norm{\vz_t - \vone\bz_t}^2\right) \\
& + \frac{4\eta L^2}{m\mu}\Norm{\vz_{t+1/2} - \vone\bz_{t+1/2}}^2  + \frac{\mu\eta}{2}\Norm{\bz_{t+1/2}-\bz_t}^2 + \frac{\mu\eta}{2}\Norm{\bz_t-\bz^*}^2,  
\end{align*}
where the last step follows the proof of Lemma \ref{lem:grad-avg}.
By rearranging above result, we have
\begin{align*}
& \Norm{\bz_{t+1}-\bz^*}^2    \\
\leq & \left(1- \frac{\mu\eta}{2}\right)\Norm{\bz_t-\bz^*}^2 - \frac{1}{2}\Norm{\bz_{t+1/2}-\bz_{t+1}}
  -(1-6\eta^2L^2)\Norm{\bz_{t+1/2}-\bz_t}^2 \\
 & + \frac{6L^2\eta^2}{m}\left(\Norm{\vz_{t+1/2} - \vone\bz_{t+1/2}}^2 + \Norm{\vz_t - \vone\bz_t}^2\right) + \frac{4\eta L^2}{m\mu}\Norm{\vz_{t+1/2} - \vone\bz_{t+1/2}}^2 \\
& + \frac{\mu\eta}{2}\Norm{\bz_{t+1/2}-\bz_t}^2 + 2\eta\mu\Norm{\bz_t-\bz_{t+1/2}}^2 \\
= & \left(1- \frac{\mu\eta}{2}\right)\Norm{\bz_t-\bz^*}^2 - \frac{1}{2}\Norm{\bz_{t+1/2}-\bz_{t+1}}
-\left(1-\frac{5\mu\eta}{2}-6\eta^2L^2\right)\Norm{\bz_{t+1/2}-\bz_t}^2 \\
& + \frac{2L\eta\left(3L\eta+ 2\kappa\right)}{m}\Norm{\vz_{t+1/2} - \vone\bz_{t+1/2}}^2 + \frac{6L^2\eta^2}{m}\Norm{\vz_t - \vone\bz_t}^2 \\
\leq & \left(1- \frac{\mu\eta}{2}\right)\Norm{\bz_t-\bz^*}^2 - \frac{1}{2}\Norm{\bz_{t+1/2}-\bz_{t+1}}
-\left(1-\frac{5\mu\eta}{2}-6\eta^2L^2\right)\Norm{\bz_{t+1/2}-\bz_t}^2 \\
& + \frac{2L\eta\left(3L\eta+ 2\kappa\right)}{m}\left(\rho\Norm{\vz_t - \vone\bz_t} + \rho\eta\Norm{\vs_t - \vone\bs_t}\right) + \frac{6L^2\eta^2}{m}\Norm{\vz_t - \vone\bz_t}^2 \\
\leq & \left(1- \frac{\mu\eta}{2}\right)\Norm{\bz_t-\bz^*}^2 - \frac{1}{2}\Norm{\bz_{t+1/2}-\bz_{t+1}}
-\left(1-\frac{5\mu\eta}{2}-6\eta^2L^2\right)\Norm{\bz_{t+1/2}-\bz_t}^2 \\
& + \frac{4L\eta\left(3L\eta+ 2\kappa\right)\rho}{m}\vone^\top\hr_t^2,
\end{align*}
where the second inequality use Lemma \ref{lem:diff-z12b}.
\end{proof}

\subsection{The Proof of Lemma \ref{lem:errorb}}

We first introduce some lemmas for the consensus error of each variables.
We use the notations of 
\begin{align*}
\hr_t = \begin{bmatrix}
\norm{\vz_t-\vone\bz_t} \\[0.1cm] \eta\norm{\vs_t-\vone\bs_t}
\end{bmatrix}
\qquad \text{and} \qquad
\hr_t^2 = \begin{bmatrix}
\norm{\vz_t-\vone\bz_t}^2 \\[0.1cm] \eta^2\norm{\vs_t-\vone\bs_t}^2
\end{bmatrix}.
\end{align*}

\begin{lem}\label{lem:diff-z12b}
Suppose Assumption \ref{asm:unconstrained} holds. For Algorithm \ref{alg:deg}, we have
\begin{align*}
    \Norm{\vz_{t+1/2} - \vone\bz_{t+1/2}} 
\leq \rho\left(\Norm{\vz_t - \vone\bz_t} + \eta\Norm{\vs_t - \vone\bs_t}\right).
\end{align*}
\end{lem}
\begin{proof}
Similar to the proof of Lemma \ref{lem:diff-z12}, we have
\begin{align*}
   & \Norm{\vz_{t+1/2} - \vone\bz_{t+1/2}} \\
= & \Norm{\BT(\vz'_t - \eta\vs_t) - \frac{1}{m}\vone\vone^\top\BT(\vz'_t - \eta\vs_t)} \\
\leq  & \rho \Norm{(\vz'_t - \eta\vs_t)  - \frac{1}{m}\vone\vone^\top\left(\vz'_t - \eta\vs_t\right)} \\
=  & \rho \Norm{(\vz'_t - \eta\vs_t)  - \vone(\bz'_t - \eta\bs_t)} \\
\leq  & \rho \Norm{\vz'_t - \vone\bz'_t} + \rho\eta\Norm{\vs_t - \vone\bs_t}.
\end{align*}
\end{proof}

\begin{lem}\label{lem:errs12b}
Suppose Assumption \ref{asm:unconstrained} holds. For Algorithm \ref{alg:deg}, we have
\begin{align*}
   \Norm{\vs_{t+1/2} - \vone\bs_{t+1/2}} 
\leq \rho\left(L\Norm{\vz_{t} - \vone\bz_{t}} + L\Norm{\vz_{t+1/2} - \vone\bz_{t+1/2}} + \Norm{\vs_t - \vone\bs_t} + L\sqrt{m}\Norm{\bz_{t+1/2} - \bz_t}\right).
\end{align*}
\end{lem}
\begin{proof}
Similar to the proof of Lemma \ref{lem:errs12}, we have
\begin{align*}
   & \Norm{\vs_{t+1/2} - \vone\bs_{t+1/2}} \\
= & \Norm{\BT(\vs_t + \vg(\vz_{t+1/2}) - \vg(\vz_t)) - \frac{1}{m}\vone\vone^\top\BT(\vs_t + \vg(\vz_{t+1/2}) - \vg(\vz_t))} \\
\leq & \rho\Norm{\vs_t + \vg(\vz_{t+1/2}) - \vg(\vz_t) - \frac{1}{m}\vone\vone^\top(\vs_t + \vg(\vz_{t+1/2}) - \vg(\vz_t))} \\
\leq & \rho\Norm{\vs_t - \frac{1}{m}\vone\vone^\top\vs_t} + \rho\Norm{\vg(\vz_{t+1/2}) - \vg(\vz_t) - \frac{1}{m}\vone\vone^\top(\vg(\vz_{t+1/2}) - \vg(\vz_t))} \\
\leq & \rho\left(\Norm{\vs_t - \vone\bs_t} + \Norm{\vg(\vz_{t+1/2}) - \vg(\vz_t)}\right) \\
\leq & \rho\left(\Norm{\vs_t - \vone\bs_t} + L\Norm{\vz_{t+1/2} - \vz_t}\right) \\
\leq & \rho\left(\Norm{\vs_t - \vone\bs_t} + L\Norm{\vz_{t+1/2} - \vone\bz_{t+1/2}} + L\sqrt{m}\Norm{\bz_{t+1/2} - \bz_t} + L\Norm{\vz_{t} - \vone\bz_{t}}\right).
\end{align*}
\end{proof}

\begin{lem}\label{lem:diff-z11b}
Suppose Assumption \ref{asm:unconstrained} holds. For Algorithm \ref{alg:deg}, we have
\begin{align*}
   \Norm{\vz_{t+1} - \vone\bz_{t+1}} 
\leq \rho\left(\Norm{\vz_t - \vone\bz_t} + \eta\Norm{\vs_{t+1/2} - \vone\bs_{t+1/2}}\right).
\end{align*}
\end{lem}
\begin{proof}
Similar to the proof of Lemma \ref{lem:errz1}, we have
\begin{align*}
   & \Norm{\vz_{t+1} - \vone\bz_{t+1}} \\
= & \Norm{\BT(\vz_t - \eta\vs_{t+1/2}) - \frac{1}{m}\vone\vone^\top\BT(\vz_t - \eta\vs_{t+1/2})} \\
\leq  & \rho \Norm{(\vz_t - \eta\vs_{t+1/2})  - \frac{1}{m}\vone\vone^\top\left(\vz_t - \eta\vs_{t+1/2}\right)} \\
\leq  & \rho \Norm{\vz_t - \eta\vs_{t+1/2}  - \vone(\bz_t - \eta\bs_{t+1/2})} \\
\leq  & \rho\left(\Norm{\vz_t - \vone\bz_t} + \eta\Norm{\vs_{t+1/2} - \vone\bs_{t+1/2}}\right) 
\end{align*}
\end{proof}

\begin{lem}\label{lem:err-sb}
Suppose Assumption \ref{asm:smooth} holds. For Algorithm \ref{alg:deg}, we have
\begin{align*}
    \Norm{\vs_{t+1} - \vone\bs_{t+1}} 
\leq & \rho\Norm{\vs_t - \vone\bs_t} + \rho L\left(\Norm{\vz_{t+1} - \vone\bz_{t+1}}  +  2\Norm{\vz_{t+1/2} - \vone\bz_{t+1/2}} + \Norm{\vz_t - \vone\bz_t}\right) \\
& + \rho L\left(\sqrt{m}\Norm{\bz_{t+1} - \bz_{t+1/2}} + \sqrt{m}\Norm{\bz_{t+1/2} - \bz_t} \right). 
\end{align*}
\end{lem}
\begin{proof}
Similar to the proof of Lemma \ref{lem:err-s}, we have
\begin{align*}
   & \Norm{\vs_{t+1} - \vone\bs_{t+1}} \\
= & \Norm{\BT(\vs_t + \vg(\vz_{t+1}) - \vg(\vz_t)) - \frac{1}{m}\vone\vone^\top\BT(\vs_t + \vg(\vz_{t+1}) - \vg(\vz_t))} \\
\leq & \rho\Norm{\vs_t + \vg(\vz_{t+1}) - \vg(\vz_t) - \frac{1}{m}\vone\vone^\top(\vs_t + \vg(\vz_{t+1}) - \vg(\vz_t))} \\
\leq & \rho\Norm{\vs_t - \vone\bs_t} + \rho\Norm{\vg(\vz_{t+1}) - \vg(\vz_t) - \frac{1}{m}\vone\vone^\top(\vg(\vz_{t+1}) - \vg(\vz_t))} \\
\leq & \rho\Norm{\vs_t - \vone\bs_t} + \rho\Norm{\vg(\vz_{t+1}) - \vg(\vz_t)} \\
\leq & \rho\Norm{\vs_t - \vone\bs_t} + \rho L\Norm{\vz_{t+1} - \vz_t} \\
\leq & \rho\Norm{\vs_t - \vone\bs_t} + \rho L\left(\Norm{\vz_{t+1} - \vz_{t+1/2}} + \Norm{\vz_{t+1/2} - \vz_t}\right) \\
\leq & \rho\Norm{\vs_t - \vone\bs_t} + \rho L\left(\Norm{\vz_{t+1} - \vone\bz_{t+1}} + \Norm{\vone\bz_{t+1} - \vone\bz_{t+1/2}} +  \Norm{\vz_{t+1/2} - \vone\bz_{t+1/2}}\right) \\
& + \rho L\left(\Norm{\vz_{t+1/2} - \vone\bz_{t+1/2}} + \Norm{\vone\bz_{t+1/2} - \vone\bz_t} + \Norm{\vz_t - \vone\bz_t}\right) \\
= & \rho\Norm{\vs_t - \vone\bs_t} + \rho L\left(\Norm{\vz_{t+1} - \vone\bz_{t+1}} + \sqrt{m}\Norm{\bz_{t+1} - \bz_{t+1/2}} +  \Norm{\vz_{t+1/2} - \vone\bz_{t+1/2}}\right) \\
& + \rho L\left(\Norm{\vz_{t+1/2} - \vone\bz_{t+1/2}} + \sqrt{m}\Norm{\bz_{t+1/2} - \bz_t} + \Norm{\vz_t - \vone\bz_t}\right) \\
= & \rho\Norm{\vs_t - \vone\bs_t} + \rho L\left(\Norm{\vz_{t+1} - \vone\bz_{t+1}}  +  2\Norm{\vz_{t+1/2} - \vone\bz_{t+1/2}} + \Norm{\vz_t - \vone\bz_t}\right) \\
& + \rho L\left(\sqrt{m}\Norm{\bz_{t+1} - \bz_{t+1/2}} + \sqrt{m}\Norm{\bz_{t+1/2} - \bz_t} \right). 
\end{align*}
\end{proof}

Then we give the proof of Lemma \ref{lem:errorb}.
\begin{proof}
The analysis is similar to the proof of Lemma \ref{lem:error}.
Using Lemma \ref{lem:diff-z12b}, we have
\begin{align*}
\Norm{\vz_{t+1/2} - \vone\bz_{t+1/2}} \leq \rho\begin{bmatrix}
1 & 1 
\end{bmatrix} \hr_t.
\end{align*}
Using Lemma \ref{lem:errs12b}, we have
\begin{align*}
   & \Norm{\vs_{t+1/2} - \vone\bs_{t+1/2}} \\
\leq & \rho\left(\begin{bmatrix} L & 1/\eta \end{bmatrix} \hr_t
+ L \Norm{\vz_{t+1/2} - \vone\bz_{t+1/2}}  +  L\sqrt{m}\Norm{\bz_{t+1/2} - \bz_t}\right) \\
= & \rho\left(\begin{bmatrix} L & 1/\eta \end{bmatrix} \hr_t
+ \rho L\begin{bmatrix} 1 & 1 \end{bmatrix} \hr_t  
+  L\sqrt{m}\Norm{\bz_{t+1/2} - \bz_t}\right) \\
\leq & \rho\begin{bmatrix} L & L+1/\eta \end{bmatrix} \hr_t  
+  \rho L\sqrt{m}\Norm{\bz_{t+1/2} - \bz_t} \\
\leq & \rho\begin{bmatrix} L & 3/(2\eta) \end{bmatrix} \hr_t  
+  \rho L\sqrt{m}\Norm{\bz_{t+1/2} - \bz_t}.
\end{align*}
Using Lemma \ref{lem:diff-z12b}, we have
\begin{align*}
 &  \Norm{\vz_{t+1} - \vone\bz_{t+1}} \\
= & 2\rho\left(
\begin{bmatrix} 1 & 0 \end{bmatrix} \hr_t
 + \eta\Norm{\vs_{t+1/2} - \vone\bs_{t+1/2}}\right) \\
\leq & 2\rho\left(
\begin{bmatrix} 1 & 0 \end{bmatrix} \hr_t
 + \eta\rho\begin{bmatrix} L & 3/(2\eta) \end{bmatrix} \hr_t  
+  \eta\rho L\sqrt{m}\Norm{\bz_{t+1/2} - \bz_t}\right) \\
\leq & 2\rho\left(
\begin{bmatrix} 1 & 0 \end{bmatrix} \hr_t
 + \begin{bmatrix} 0.25 & 1.5 \end{bmatrix} \hr_t  
+  0.25\rho  \sqrt{m}\Norm{\bz_{t+1/2} - \bz_t}\right) \\
\leq & \rho\begin{bmatrix} 2.5 & 3 \end{bmatrix} \hr_t  
+  0.5\rho \sqrt{m}\Norm{\bz_{t+1/2} - \bz_t} 
\end{align*}
and
\begin{align*}
&  \Norm{\vz_{t+1} - \vone\bz_{t+1}}^2  \\
\leq & 4\rho^2\begin{bmatrix} 6.25 &  9 \end{bmatrix} \hr_t^2  
+  \rho^2m\Norm{\bz_{t+1/2} - \bz_t}^2 \\
\leq & \rho^2\begin{bmatrix} 25 & 36 \end{bmatrix} \hr_t^2  
+  \rho^2m\Norm{\bz_{t+1/2} - \bz_t}^2. 
\end{align*}
Note that $\eta\leq 1/4L$. Using Lemma \ref{lem:err-sb}, we have
\begin{align*}
  &  \Norm{\vs_{t+1} - \vone\bs_{t+1}} \\
\leq & \rho\begin{bmatrix}
 L & 1/\eta
\end{bmatrix}\hr_t + \rho L\left(\rho\begin{bmatrix} 2.5 & 3  \end{bmatrix} \hr_t  
+  0.5\rho\sqrt{m}\Norm{\bz_{t+1/2} - \bz_t}\right) \\
& + 2\rho^2 L\begin{bmatrix}1 & 1 \end{bmatrix} \hr_t
 + \rho L\sqrt{m}\left[\Norm{\bz_{t+1} - \bz_{t+1/2}} +  \Norm{\bz_{t+1/2} - \bz_t} \right] \\
\leq & \rho\begin{bmatrix}
 3.5L & 1/\eta+3L
\end{bmatrix}\hr_t +  0.5\rho^3 L\sqrt{m}\Norm{\bz_{t+1/2} - \bz_t} \\
& + \rho \begin{bmatrix}2L & 2L \end{bmatrix} \hr_t
 + \rho L\sqrt{m}\left[\Norm{\bz_{t+1} - \bz_{t+1/2}} +  \Norm{\bz_{t+1/2} - \bz_t} \right] \\ 
\leq & \rho\begin{bmatrix}
 5.5L & 1/\eta+5L
\end{bmatrix}\hr_t  + 1.5\rho L\sqrt{m}\left[\Norm{\bz_{t+1} - \bz_{t+1/2}} +  \Norm{\bz_{t+1/2} - \bz_t} \right] 
\end{align*}
that is
\begin{align*}
    \eta\Norm{\vs_{t+1} - \vone\bs_{t+1}} 
\leq  \rho\begin{bmatrix}
 11/8 & 9/4
\end{bmatrix}\hr_t  + \frac{3}{8}\rho\sqrt{m}\left(\Norm{\bz_{t+1} - \bz_{t+1/2}} +  \Norm{\bz_{t+1/2} - \bz_t} \right), 
\end{align*}
which implies
\begin{align*}
 & \eta^2\Norm{\vs_{t+1} - \vone\bs_{t+1}}^2 \\
\leq & 3\rho^2\begin{bmatrix}
 (11/8)^2 & (9/4)^2
\end{bmatrix}\hr_t^2  + \frac{27}{64}\rho^2m\left(\Norm{\bz_{t+1} - \bz_{t+1/2}} +  \Norm{\bz_{t+1/2} - \bz_t} \right)^2 \\
\leq & \rho^2\begin{bmatrix}
 \dfrac{363}{64} & \dfrac{243}{16}
\end{bmatrix}\hr_t^2  + \frac{27}{32}\rho^2m\left(\Norm{\bz_{t+1} - \bz_{t+1/2}}^2 +  \Norm{\bz_{t+1/2} - \bz_t}^2 \right).
\end{align*}
In summary, we have
\begin{align*}
    \hr_{t+1}^2
\leq \begin{bmatrix}
 25\rho^2 & 36\rho^2 \\[0.3cm]
\dfrac{363\rho^2}{64} & \dfrac{243\rho^2}{16}
\end{bmatrix} \hr_t^2 +
\begin{bmatrix}  
\rho^2m\Norm{\bz_{t+1/2} - \bz_t}^2  \\[0.25cm] 
\dfrac{27}{32}\rho^2m\left(\Norm{\bz_{t+1} - \bz_{t+1/2}}^2 +  \Norm{\bz_{t+1/2} - \bz_t}^2 \right).
\end{bmatrix} 
\end{align*}
Hence, we have
\begin{align*}
    \vone^\top\hr_{t+1}^2
\leq & 52\rho^2 \vone^\top\hr_t^2 +
2\rho^2m \left[\Norm{\bz_{t+1} - \bz_{t+1/2}}^2 +  \Norm{\bz_{t+1/2} - \bz_t}^2 \right].
\end{align*}
\end{proof}

\subsection{The Proof of Lemma \ref{lem:deg-unconstrained}}
\begin{proof}
Using Lemma~\ref{lem:eg-relation} and Lemma~\ref{lem:errorb}, we have
\begin{align*}
& \Norm{\bz_{t+1}-\bz^*}^2 + \frac{1}{m}\vone^\top\hr_{t+1}^2 \\
\leq & \left(1- \frac{\mu\eta}{2}\right)\Norm{\bz_t-\bz^*}^2 - \frac{1}{2}\Norm{\bz_{t+1/2}-\bz_{t+1}}
-\left(1-\frac{5\mu\eta}{2}-6\eta^2L^2\right)\Norm{\bz_{t+1/2}-\bz_t}^2 \\
& + \frac{4L\eta\left(3L\eta+ 2\kappa\right)\rho}{m}\vone^\top\hr_t^2
+ \frac{52\rho^2}{m}\vone^\top\hr_t^2 + 2\rho^2\left[\Norm{\bz_{t+1} - \bz_{t+1/2}}^2 +  \Norm{\bz_{t+1/2} - \bz_t}^2 \right] \\
= & \left(1- \frac{\mu\eta}{2}\right)\left(\Norm{\bz_t-\bz^*}^2 + \frac{4L\eta\left(3L\eta + 2\kappa\right)\rho+52\rho^2}{m(1-\frac{\mu\eta}{2})}\vone^\top\hr_t^2\right) \\
& - \left(\frac{1}{2}-2\rho^2\right)\Norm{\bz_{t+1/2}-\bz_{t+1}}
-\left(1-\frac{5\mu\eta}{2}-6\eta^2L^2-2\rho^2\right)\Norm{\bz_{t+1/2}-\bz_t}^2 \\
\leq & \left(1- \frac{\mu\eta}{2}\right)\left(\Norm{\bz_t-\bz^*}^2 + \frac{1}{m}\vone^\top\hr_t^2\right) 
= \left(1- \frac{1}{12\kappa}\right)\left(\Norm{\bz_t-\bz^*}^2 + \frac{1}{m}\vone^\top\hr_t^2\right)
\end{align*}
where the last inequality is due to the value of $K$ implies
\begin{align*}
\rho \leq \min\left\{
\frac{1}{2},~
\sqrt{\frac{1}{2}\left(1-\frac{5\mu\eta}{2}-6\eta^2L^2\right)},~
\frac{1-\frac{\mu\eta}{2}}{4L\eta\left(3L\eta + 2\kappa\right)+52}\right\}
=\min\left\{
\frac{1}{2},
\sqrt{\frac{1}{2}\left(\frac{5}{6}-\frac{5\mu\eta}{2}\right)},
\frac{3\left(2-\mu\eta\right)}{2(4\kappa+157)}\right\}.
\end{align*}
\end{proof}

\subsection{The Proof of Theorem \ref{thm:eg-unconstrained}}
\begin{proof}
The value of $K_0$ implies $\vone^\top r_0^2 = \Norm{\vs_0-\vone\bs_0}^2 \leq m\eps$. Then Lemma \ref{lem:deg-unconstrained} implies
\begin{align*}
& \Norm{\bz_T-\bz^*}^2 \\
\leq & \Norm{\bz_T-\bz^*}^2 + \frac{1}{m}\vone^\top\hr_T^2  \\
\leq & \left(1- \frac{1}{12\kappa}\right)^T\left(\Norm{\bz_0-\bz^*}^2 + \frac{1}{m}\vone^\top\hr_0^2\right) \\
\leq & \left(1- \frac{1}{12\kappa}\right)^T\left(\Norm{\bz_0-\bz^*}^2 + \eps\right).
\end{align*}
Hence, the value of $T$ means $\Norm{\bz_T-\bz^*}^2\leq\eps$.

Since each iteration requires $\fO(n)$ SFO calls, the total complexity is
\begin{align*}
\fO(nT)=\fO\left(\kappa n\log\left(\frac{1}{\eps}\right)\right)
\end{align*}
and the number of communication round is 
\begin{align*}
KT+K_0 = \fO\left(\kappa\sqrt{\chi}\log\kappa\log\left(\frac{1}{\eps}\right)\right).
\end{align*}
\end{proof}

\subsection{The Proof of Theorem \ref{thm:eg-scsc-constarined}}

We first provide some lemmas for our proof.

\begin{lem}\label{lem:rel-constrainedb}
Suppose Assumption~\ref{asm:smooth} and \ref{asm:scsc}  hold. For Algorithm \ref{alg:deg} with $\eta = 1/\left(6L\right)$, we have, we have
\begin{align*}
\Norm{\bz_{t+1}-\bz^*}^2  \leq \left(1- \frac{\mu\eta}{2}\right)\Norm{\bz_t-\bz^*}^2 + \hzeta_t
\end{align*}
where
\begin{align*}
\hzeta_t = &\frac{4L\eta\left(3L\eta+ 2\kappa\right)}{m}\left(\Norm{\vz_{t+1/2} - \vone\bz_{t+1/2}}^2 + \Norm{\vz_t - \vone\bz_t}^2\right) \\
& + \Norm{\bz_{t+1}-\bz_t+\eta \bs_{t+1/2}+\bz_{t+1}-\bz^*}\Norm{\hDelta_{t+1/2}} 
+ \Norm{\bz_{t+1/2} - \bz_t + \eta\bs_t+\bz_{t+1/2} -\bz_{t+1}}\Norm{\hDelta_t}
\end{align*}
\end{lem}
\begin{proof}
Lemma \ref{lem:FM} means
\begin{align*}
   \bz_{t+1/2} 
= & \frac{1}{m}\vone^\top\BT(\fP(\vz_t-\eta\vs_t)) 
=  \frac{1}{m}\vone^\top\fP(\vz_t-\eta\vs_t) \\
= & \fP(\bz_t - \eta\bs_t)  + \underbrace{\frac{1}{m}\vone^\top\fP(\vz_t-\eta\vs_t) - \fP(\bz_t - \eta\bs_t)}_{\hDelta_t}  \\
= & \fP(\bz_t - \eta\bs_t)  + {\hDelta_t} 
\end{align*}
and
\begin{align*}
   \bz_{t+1}  
= & \frac{1}{m}\vone^\top\BT(\fP(\vz_t-\eta\vv_{t+1/2})) 
=  \frac{1}{m}\vone^\top\fP(\vz_t-\eta\vv_{t+1/2}) \\
= & \fP(\bz_t - \eta\bv_{t+1/2}) + \underbrace{\frac{1}{m}\vone^\top\fP(\vz_t-\eta\vs_{t+1/2}) - \fP(\bz_t - \eta\bs_{t+1/2})}_{\hDelta_{t+1/2}} \\
= & \fP(\bz_t - \eta\bs_{t+1/2}) + \hDelta_{t+1/2}.
\end{align*}
which implies
\begin{align*}
\fP(\bz_t - \eta\bg_t)  =  \bz_{t+1/2} - \hDelta_t \qquad\text{and}\qquad
\fP(\bz_t - \eta\bg_{t+1/2}) =   \bz_{t+1}  - \hDelta_{t+1/2}.
\end{align*}
Using Lemma \ref{lem:proj} with $u=\bz_t-\eta \bg_{t+1/2}$ and $v=\bz^*$, we have
\begin{align*}
    \inner{\bz_{t+1}-\hDelta_{t+1/2}-\bz_t+\eta \bs_{t+1/2}}{\bz_{t+1}-\hDelta_{t+1/2}-\bz^*} \leq 0.
\end{align*}
Using Lemma \ref{lem:proj} with $u=\bz_t-\eta \bg_t$ and $v=\bz_{t+1}$, we have
\begin{align*}
    \inner{\bz_{t+1/2} - \hDelta_{t} - \bz_t + \eta\bs_t}{\bz_{t+1/2} - \hDelta_{t} -\bz_{t+1}} \leq 0.
\end{align*}
Summing over above inequalities, we have
\begin{align*}
& \inner{\bz_{t+1}-\bz_t+\eta \bs_{t+1/2}}{\bz^*-\bz_{t+1}} + \inner{\bz_{t+1}-\bz_t+\eta \bs_{t+1/2}+\bz_{t+1}-\bz^*}{\hDelta_{t+1/2}} - \Norm{\hDelta_{t+1/2}}^2 \\
& + \inner{\bz_{t+1/2}- \bz_t + \eta\bs_t}{\bz_{t+1} - \bz_{t+1/2}} + \inner{\bz_{t+1/2} - \bz_t + \eta\bs_t+\bz_{t+1/2} -\bz_{t+1}}{\hDelta_t} - \Norm{\hDelta_t}^2 \geq 0,
\end{align*}
which means
\begin{align}\label{eq:proj-ieq0b}
\begin{split}
& \inner{\bz_{t+1}-\bz_t}{\bz^*-\bz_{t+1}} + \inner{\bz_{t+1/2}- \bz_t}{\bz_{t+1} - \bz_{t+1/2}}   \\
& + \eta \inner{\bs_{t+1/2}}{\bz^*-\bz_{t+1/2}} + \eta\inner{\bs_{t+1/2}-\bs_t}{\bz_{t+1/2}-\bz_{t+1}} \\
& + \inner{\bz_{t+1}-\bz_t+\eta \bs_{t+1/2}+\bz_{t+1}-\bz^*}{\hDelta_{t+1/2}} \\
& + \inner{\bz_{t+1/2} - \bz_t + \eta\bs_t+\bz_{t+1/2} -\bz_{t+1}}{\hDelta_t}  \geq 0.
\end{split}
\end{align}
In constrained case, it is easy to verify the results of (\ref{ieq:eg-00}), (\ref{ieq:eg-01}), (\ref{ieq:eg-2}) and (\ref{ieq:eg-3}) also hold. Hence, we have
\begin{align}\label{ieq:eg-1c}
\begin{split}
& 2\eta\inner{\bs_{t+1/2}}{\bz_{t+1/2}-\bz^*} \\
= & 2\eta\inner{\bs_{t+1/2}}{\bz_{t+1}-\bz^*} + 2\eta\inner{\bs_{t+1/2}}{\bz_{t+1/2}-\bz_{t+1}} \\
= & 2\inner{\bz_t-\bz_{t+1}}{\bz_{t+1}-\bz^*} 
+ 2\inner{\bz_t-\bz_{t+1/2}}{\bz_{t+1/2}-\bz_{t+1}} 
+ 2\inner{\bz_{t+1/2}-\bz_{t+1}}{\bz_{t+1/2}-\bz_{t+1}} \\
= & \Norm{\bz_t-\bz^*}^2 - \Norm{\bz_t-\bz_{t+1}}^2 - \Norm{\bz_{t+1}-\bz^*}^2 + \Norm{\bz_t-\bz_{t+1}}^2 - \Norm{\bz_t-\bz_{t+1/2}}^2 - \Norm{\bz_{t+1/2}-\bz_{t+1}}^2 \\
& + 2\eta\inner{\bs_{t+1/2}-\bs_t}{\bz_{t+1/2}-\bz_{t+1}} + \inner{\bz_{t+1}-\bz_t+\eta \bs_{t+1/2}+\bz_{t+1}-\bz^*}{\hDelta_{t+1/2}} \\
& + \inner{\bz_{t+1/2} - \bz_t + \eta\bs_t+\bz_{t+1/2} -\bz_{t+1}}{\hDelta_t} \\
\leq & \Norm{\bz_t-\bz^*}^2 - \Norm{\bz_{t+1}-\bz^*}^2 
 - \Norm{\bz_t-\bz_{t+1/2}}^2 - \Norm{\bz_{t+1/2}-\bz_{t+1}}^2 \\
& + 2\eta^2\Norm{\bs_{t+1/2}-\bs_t}^2 + \frac{1}{2}\Norm{\bz_{t+1/2}-\bz_{t+1}}^2 + \inner{\bz_{t+1}-\bz_t+\eta \bs_{t+1/2}+\bz_{t+1}-\bz^*}{\hDelta_{t+1/2}} \\
& + \inner{\bz_{t+1/2} - \bz_t + \eta\bs_t+\bz_{t+1/2} -\bz_{t+1}}{\hDelta_t} \\
\leq & \Norm{\bz_t-\bz^*}^2 - \Norm{\bz_{t+1}-\bz^*}^2 
- \frac{1}{2}\Norm{\bz_{t+1/2}-\bz_{t+1}}^2
 -(1-6\eta^2L^2)\Norm{\bz_{t+1/2}-\bz_t}^2 \\
 & + \frac{6L^2\eta^2}{m}\left(\Norm{\vz_{t+1/2} - \vone\bz_{t+1/2}}^2 + \Norm{\vz_t - \vone\bz_t}^2\right) + \Norm{\bz_{t+1}-\bz_t+\eta \bs_{t+1/2}+\bz_{t+1}-\bz^*}\Norm{\hDelta_{t+1/2}} \\
& + \Norm{\bz_{t+1/2} - \bz_t + \eta\bs_t+\bz_{t+1/2} -\bz_{t+1}}\Norm{\hDelta_t}.
\end{split}
\end{align}
Combining (\ref{ieq:eg-1c}) with (\ref{ieq:eg-2}), we have
\begin{align*}
& \Norm{\bz_{t+1}-\bz^*}^2    \\
\leq & \left(1- \frac{\mu\eta}{2}\right)\Norm{\bz_t-\bz^*}^2 - \frac{1}{2}\Norm{\bz_{t+1/2}-\bz_{t+1}}
-\left(1-\frac{5\mu\eta}{2}-6\eta^2L^2\right)\Norm{\bz_{t+1/2}-\bz_t}^2 \\ 
& + \frac{4L\eta\left(3L\eta+ 2\kappa\right)}{m}\left(\Norm{\vz_{t+1/2} - \vone\bz_{t+1/2}}^2 + \Norm{\vz_t - \vone\bz_t}^2\right) \\
& + \Norm{\bz_{t+1}-\bz_t+\eta \bs_{t+1/2}+\bz_{t+1}-\bz^*}\Norm{\hDelta_{t+1/2}} 
+ \Norm{\bz_{t+1/2} - \bz_t + \eta\bs_t+\bz_{t+1/2} -\bz_{t+1}}\Norm{\hDelta_t} \\
\leq & \left(1- \frac{\mu\eta}{2}\right)\Norm{\bz_t-\bz^*}^2 + \hzeta_t.
\end{align*}
\end{proof}

\begin{lem}\label{lem:eg-err}
Suppose Assumption~\ref{asm:smooth} and \ref{asm:constrained} and we have $\Norm{\vs_0 - \vone\bs_0}\leq\delta'$.
For Algorithm \ref{alg:deg} with $\eta=1/(6L)$ and 
\begin{align*}
K \geq \sqrt{\chi}\log\left(\frac{2\left(\sqrt{m}LD+\delta'\right)}{\delta'}\right),
\end{align*}
holds that $\Norm{\vz_t - \vone\bz_t}\leq\hdelta'$, $\Norm{\vz_{t+1/2} - \vone\bz_{t+1/2}}\leq\hdelta'$, $\Norm{\vs_t - \vone\bs_t}\leq\hdelta'/\eta$ and $\Norm{\vs_{t+1/2} - \vone\bs_{t+1/2}}\leq\hdelta'/\eta$ for any $t\geq 0$.
\end{lem}
\begin{proof}
The analysis is similar to the proof of Lemma \ref{lem:constrained-err}.
It is obvious the statement holds for $t=0$.
Suppose we have that $\Norm{\vz_t - \vone\bz_t}\leq\hdelta'$, $\Norm{\vw_t - \vone\bw_t}\leq\hdelta'$, $\Norm{\vs_t - \vone\bs_t}\leq\hdelta'/\eta$ for $t\geq 1$, then Lemma \ref{lem:diff-z12b} means
\begin{align*}
   & \Norm{\vz_{t+1/2} - \vone\bz_{t+1/2}} \\
= & \Norm{\BT(\fP_\FZ(\vz_t - \eta\vs_t)) - \frac{1}{m}\vone\vone^\top\BT(\fP_\FZ(\vz_t - \eta\vs_t))} \\
\leq  & \rho \Norm{\fP_\FZ(\vz_t - \eta\vs_t)  - \frac{1}{m}\vone\vone^\top\fP_\FZ\left(\vz_t - \eta\vs_t\right)} \\
\leq  & \rho \Norm{\fP_\FZ(\vz_t - \eta\vs_t)  - \fP_\FZ\left(\vone(\bz_t - \eta\bs_t)\right)} + \rho \Norm{\fP_\FZ\left(\vone(\bz_t - \eta\bs_t)\right)  - \frac{1}{m}\vone\vone^\top\fP_\FZ\left(\vz_t - \eta\vs_t\right)} \\
\leq  & \rho \Norm{\vz_t - \eta\vs_t  - \vone(\bz_t - \eta\bs_t)} + \rho \Norm{(\vz_t - \eta\vs_t) - \vone(\bz_t - \eta\bs_t)} \\
\leq  & 2\rho \Norm{\vz_t - \vone\bz_t} + 2\rho\eta\Norm{\vs_t - \vone\bs_t} \\
\leq  & \rho \left(2\sqrt{m}D + 2\hdelta'\right) \leq \hdelta'.
\end{align*}
Using Lemma \ref{lem:errs12b} we have
\begin{align*}
  & \Norm{\vs_{t+1/2} - \vone\bs_{t+1/2}} \\
\leq & \rho\left(L\Norm{\vz_{t} - \vone\bz_{t}} + L\Norm{\vz_{t+1/2} - \vone\bz_{t+1/2}} + \Norm{\vs_t - \vone\bs_t} + L\sqrt{m}\Norm{\bz_{t+1/2} - \bz_t}\right) \\
\leq & \rho\left(L\sqrt{m}D + L\sqrt{m}D + \Norm{\vs_t - \vone\bs_t} + L\sqrt{m}D\right) \\
\leq & \rho\left(3L\sqrt{m}D + \hdelta'/\eta\right) \leq \hdelta'/\eta
\end{align*}
Using Lemma \ref{lem:err-sb}, we have
\begin{align*}
   & \Norm{\vz_{t+1} - \vone\bz_{t+1}} \\
= & \Norm{\BT(\fP_\FZ(\vz_t - \eta\vs_{t+1/2})) - \frac{1}{m}\vone\vone^\top\BT(\fP_\FZ(\vz_t - \eta\vs_{t+1/2}))} \\
\leq  & \rho \Norm{\fP_\FZ(\vz_t - \eta\vs_{t+1/2})  - \frac{1}{m}\vone\vone^\top\fP_\FZ\left(\vz_t - \eta\vs_{t+1/2}\right)} \\
\leq  & \rho \Norm{\fP_\FZ(\vz_t - \eta\vs_{t+1/2})  - \fP_\FZ\left(\vone(\bz_t - \eta\bs_{t+1/2})\right)} + \rho \Norm{\fP_\FZ\left(\vone(\bz_t - \eta\bs_{t+1/2})\right)  - \frac{1}{m}\vone\vone^\top\fP_\FZ\left(\vz_t - \eta\vs_{t+1/2}\right)} \\
\leq  & \rho \Norm{\vz_t - \eta\vs_{t+1/2}  - \vone(\bz_t - \eta\bs_{t+1/2})} + \rho \Norm{(\vz_t - \eta\vs_{t+1/2}) - \vone(\bz_t - \eta\bs_{t+1/2})} \\
\leq  & 2\rho\left(\Norm{\vz_t - \vone\bz_t} + \eta\Norm{\vs_{t+1/2} - \vone\bs_{t+1/2}}\right) \\
\leq  & \rho\left(2\sqrt{m}D + 2\hdelta'\right) \leq \hdelta'
\end{align*}
Using Lemma \ref{lem:diff-z11b}, we have
\begin{align*}
   & \Norm{\vs_{t+1} - \vone\bs_{t+1}} \\
= & \Norm{\BT(\vs_t + \vg(\vz_{t+1}) - \vg(\vz_t)) - \frac{1}{m}\vone\vone^\top\BT(\vs_t + \vg(\vz_{t+1}) - \vg(\vz_t))} \\
\leq & \rho\Norm{\vs_t + \vg(\vz_{t+1}) - \vg(\vz_t) - \frac{1}{m}\vone\vone^\top(\vs_t + \vg(\vz_{t+1}) - \vg(\vz_t))} \\
\leq & \rho\Norm{\vs_t - \vone\bs_t} + \rho\Norm{\vg(\vz_{t+1}) - \vg(\vz_t) - \frac{1}{m}\vone\vone^\top(\vg(\vz_{t+1}) - \vg(\vz_t))} \\
\leq & \rho\Norm{\vs_t - \vone\bs_t} + \rho\Norm{\vg(\vz_{t+1}) - \vg(\vz_t)} \\
\leq & \rho\Norm{\vs_t - \vone\bs_t} + \rho L\Norm{\vz_{t+1} - \vz_t} \\
\leq & \rho(\hdelta'/\eta + \sqrt{m}LD) \leq \hdelta'/\eta.
\end{align*}
\end{proof}

\begin{lem}\label{lem:hzeta}
Under the settings of Lemma \ref{lem:rel-constrainedb} and Lemma \ref{lem:eg-err},
we have
\begin{align*}
\hzeta_t \leq \frac{4L\eta\left(2\kappa + 3L\eta\right)\hdelta'^2}{m} + \frac{8\hC_1\hdelta'}{\sqrt{m}}
\end{align*}
\end{lem}
\begin{proof}
The analysis is similar to the proof of Lemma \ref{lem:zeta}.
Using the non-expansiveness of projection, the update rule $\vz'_t=\alpha\vz_t+(1-\alpha)\vw_t$ and Lemma \ref{lem:eg-err}, we have
\begin{align*}
\Norm{\hDelta_t} = & \Norm{\frac{1}{m}\vone^\top\fP_\FZ(\vz_t-\eta\vs_t) - \fP_\FZ(\bz_t - \eta\bs_t)}  \\
= & \sqrt{\frac{1}{m}\sum_{i=1}^m\Norm{\fP_\FZ(\vz_t(i)-\eta\vs_t(i)) - \fP_\fZ(\bz_t^\top - \eta\bs_t^\top)}^2} \\
\leq  & \sqrt{\frac{1}{m}\sum_{i=1}^m\Norm{(\vz_t(i)-\eta\vs_t(i)) - (\bz_t^\top - \eta\bs_t^\top)}^2} \\
\leq  & \sqrt{\frac{1}{m}\sum_{i=1}^m 2\left(\Norm{\vz_t(i)-\bz_t^\top}^2 + \eta^2\Norm{\vs_t(i) - \bs_t^\top}^2\right)} \\
\leq  & \sqrt{\frac{1}{m}}\sqrt{2\left(\Norm{\vz_t-\vone\bz_t}^2 + \eta^2\Norm{\vs_t - \vone\bs_t}^2\right)} \\
\leq  & 2\sqrt{\frac{1}{m}}\left(\Norm{\vz_t-\vone\bz_t}^2 +  \eta\Norm{\vs_t - \vone\bs_t}\right) \\
\leq &  \frac{4\hdelta'}{\sqrt{m}}.
\end{align*}
Similarly, we obtain
\begin{align*}
\Norm{\hDelta_{t+1/2}} \leq  \frac{4\hdelta'}{\sqrt{m}}.
\end{align*}
We also have
\begin{align*}
    & \Norm{\bz_{t+1/2} - \bz_t + \eta\bs_t+\bz_{t+1/2} -\bz_{t+1}} \\
= & \Norm{\bz_{t+1/2} - \bz_t} + \Norm{\bz_{t+1/2} -\bz_{t+1}} + \eta\Norm{\bs_t} \\
\leq & 2D + \eta\Norm{\frac{1}{m}\vone^\top\vg(\vz_t)} \\
= & 2D + \eta\Norm{\frac{1}{m}\sum_{i=1}^m g_i(\vz_t(i))} \\
\leq & 2D + \eta\sqrt{\frac{1}{m}\sum_{i=1}^m \Norm{g_i(\vz_t(i))}^2} \\
\leq & 2D + \eta\sqrt{\frac{2}{m}\sum_{i=1}^m \left(\Norm{g_i(\vz_t(i))-g_i(z^*)}^2 + \Norm{g_i(z^*)}^2\right)}  \\
\leq & 2D + \eta\sqrt{\frac{2}{m}\sum_{i=1}^m \left(L^2\Norm{\vz_t(i)-z^*}^2 + \Norm{g_i(z^*)}^2\right)}  \\
\leq & 2D + \frac{1}{6L}\sqrt{2L^2D^2 + \frac{2}{m}\sum_{i=1}^m\Norm{g_i(z^*)}^2} \\
\triangleq & \hC_1
\end{align*}
and
\begin{align*}
\Norm{\bz_{t+1}-\bz'_t+\eta \bs_{t+1/2}+\bz_{t+1}-z^*} \leq  \hC_1.
\end{align*}
Combing above results, we have
\begin{align*}
\hzeta_t = & \frac{2L\eta\left(2\kappa + 3L\eta\right)}{m}\left(\Norm{\vz_{t+1/2} - \vone\bz_{t+1/2}}^2 + \Norm{\vz_t - \vone\bz_t}^2\right) \\
& + \Norm{\bz_{t+1}-\bz'_t+\eta \bs_{t+1/2}+\bz_{t+1}-z^*}\Norm{\Delta_{t+1/2}} + \Norm{\bz_{t+1/2} - \bz'_t + \eta\bs_t+\bz_{t+1/2} -\bz_{t+1}}\Norm{\Delta_t} \\
\leq & \frac{4L\eta\left(2\kappa + 3L\eta\right)\hdelta'^2}{m} + \frac{8\hC_1\hdelta'}{\sqrt{m}}.
\end{align*}
\end{proof}

Then we provide the proof of Theorem \ref{thm:eg-scsc-constarined}.

\begin{proof}
The setting of $K_0$ means $\Norm{\vs_0 - \vone\bs_0}\leq\hdelta'$. 
Then Lemma \ref{lem:rel-constrainedb} and Lemma \ref{lem:hzeta} imply
\begin{align*}
 \Norm{\bz_{t+1}-\bz^*}^2  
\leq \left(1-\frac{1}{12\kappa}\right)\Norm{\bz_t-\bz^*}^2 + \frac{4L\eta\left(2\kappa + 3L\eta\right)\hdelta'^2}{m} + \frac{8\hC_1\hdelta'}{\sqrt{m}}.
\end{align*}
Then we have
\begin{align*}
\Norm{\bz_T-\bz^*}^2  
\leq & \left(1- \frac{1}{12\kappa}\right)^T\Norm{\bz_t-\bz^*}^2 + 12\kappa\left(\frac{4L\eta\left(2\kappa + 3L\eta\right)\hdelta'^2}{m} + \frac{8\hC_1\hdelta'}{\sqrt{m}}\right) \\
= & \left(1- \frac{1}{12\kappa}\right)^T\Norm{\bz_t-\bz^*}^2 + \left(\frac{4\kappa\left(4\kappa + 1\right)\hdelta'^2}{m} + \frac{96\kappa\hC_1\hdelta'}{\sqrt{m}}\right) \\
\leq & \frac{\eps}{2} + \frac{\eps}{4} + \frac{\eps}{4} = \eps.
\end{align*}
where the last step is due to the value of $\eta$, $T$ and $\hdelta'$.

Since each iteration requires $\fO(n)$ SFO calls in expectation and we the algorithm needs to compute the full gradient at first, the total complexity is
\begin{align*}
\fO(nT) = \fO\left(\kappa n\log\left(\frac{1}{\eps}\right)\right).
\end{align*}
The number of communication rounds is 
\begin{align*}
\fO(nT) = \fO\left(\kappa\sqrt{\chi} \log\left(\frac{\kappa}{\eps}\right)\log\left(\frac{1}{\eps}\right)\right).
\end{align*}
\end{proof}

\subsection{The Proof of Theorem \ref{thm:eg-cc-constrained}}

\begin{proof}
Since the objective function is non-strongly-convex and non-strongly-concave, we first modify (\ref{ieq:eg-2}) as follows
\begin{align}\label{ieq:eg-2c}
\begin{split}
& 2\eta\inner{\bs_{t+1/2}}{\bz_{t+1/2}-\bz^*} \\
= & 2\eta\inner{g(\bz_{t+1/2})}{\bz_{t+1/2}-\bz^*} + 2\eta\inner{\bs_{t+1/2} - g(\bz_{t+1/2})}{\bz_{t+1/2}-\bz^*}  \\
\geq & 2\eta\inner{g(\bz^*)}{\bz_{t+1/2}-\bz^*} - \beta\eta\Norm{\bs_{t+1/2} - g(\bz_{t+1/2})}^2 - \frac{\eta}{\beta}\Norm{\bz_{t+1/2}-\bz^*}^2 
\end{split}
\end{align}
where we use Lemma \ref{lem:mm}, \ref{lem:optimal} and $\beta=2D^2/\eps$.

It is easy to verify (\ref{ieq:eg-1}) and (\ref{ieq:eg-3}) still holds. Connecting the with and (\ref{ieq:eg-2c}), we have
\begin{align*}
& 2\eta\inner{\bg(\bz_{t+1/2})}{\bz_{t+1/2}-\bz^*} - \beta\eta\Norm{\bs_{t+1/2} - \bg(\bz_{t+1/2})}^2 - \frac{\eta}{\beta}\Norm{\bz_{t+1/2}-\bz^*}^2 \\
\leq & 2\eta\inner{\bs_{t+1/2}}{\bz_{t+1/2}-\bz^*} \\
\leq & \Norm{\bz_t-\bz^*}^2 - \Norm{\bz_{t+1}-\bz^*}^2 
- \frac{1}{2}\Norm{\bz_{t+1/2}-\bz_{t+1}}^2
 -(1-6\eta^2L^2)\Norm{\bz_{t+1/2}-\bz_t}^2 \\
 & + \frac{6L^2\eta^2}{m}\left(\Norm{\vz_{t+1/2} - \vone\bz_{t+1/2}}^2 + \Norm{\vz_t - \vone\bz_t}^2\right) \\
& + \Norm{\bz_{t+1}-\bz_t+\eta \bs_{t+1/2}+\bz_{t+1}-\bz^*}\Norm{\hDelta_{t+1/2}} + \Norm{\bz_{t+1/2} - \bz_t + \eta\bs_t+\bz_{t+1/2} -\bz_{t+1}}\Norm{\hDelta_t},
\end{align*}
which implies
\begin{align}\label{ieq:cc-rel}
\begin{split}
& 2\eta\inner{\bg(\bz_{t+1/2})}{\bz_{t+1/2}-\bz^*}  \\
\leq & \Norm{\bz_t-\bz^*}^2 - \Norm{\bz_{t+1}-\bz^*}^2 + \beta\eta\Norm{\bs_{t+1/2} - \bg(\bz_{t+1/2})}^2 + \frac{\eta}{\beta}\Norm{\bz_{t+1/2}-\bz^*}^2 \\
 & + \frac{6L^2\eta^2}{m}\left(\Norm{\vz_{t+1/2} - \vone\bz_{t+1/2}}^2 + \Norm{\vz_t - \vone\bz_t}^2\right) \\
& + \Norm{\bz_{t+1}-\bz_t+\eta \bs_{t+1/2}+\bz_{t+1}-\bz^*}\Norm{\hDelta_{t+1/2}} + \Norm{\bz_{t+1/2} - \bz_t + \eta\bs_t+\bz_{t+1/2} -\bz_{t+1}}\Norm{\hDelta_t} \\
\leq & \Norm{\bz_t-\bz^*}^2 - \Norm{\bz_{t+1}-\bz^*}^2 + \frac{\beta L^2\eta}{m}\Norm{\vz_{t+1/2} - \vone\bz_{t+1/2}}^2 + \frac{\eta D^2}{\beta} \\
 & + \frac{6L^2\eta^2}{m}\left(\Norm{\vz_{t+1/2} - \vone\bz_{t+1/2}}^2 + \Norm{\vz_t - \vone\bz_t}^2\right) \\
& + \Norm{\bz_{t+1}-\bz_t+\eta \bs_{t+1/2}+\bz_{t+1}-\bz^*}\Norm{\hDelta_{t+1/2}} + \Norm{\bz_{t+1/2} - \bz_t + \eta\bs_t+\bz_{t+1/2} -\bz_{t+1}}\Norm{\hDelta_t} \\
= & \Norm{\bz_t-\bz^*}^2 - \Norm{\bz_{t+1}-\bz^*}^2 + \frac{\beta L^2\eta}{m}\Norm{\vz_{t+1/2} - \vone\bz_{t+1/2}}^2 + \hzeta'_t,
\end{split}
\end{align}
where 
\begin{align*}
\hzeta'_t 
= &  \frac{\eta D^2}{\beta} + \frac{6(\eta^2+\beta\eta)L^2}{m}\left(\Norm{\vz_{t+1/2} - \vone\bz_{t+1/2}}^2 + \Norm{\vz_t - \vone\bz_t}^2\right) \\
& + \Norm{\bz_{t+1}-\bz_t+\eta \bs_{t+1/2}+\bz_{t+1}-\bz^*}\Norm{\hDelta_{t+1/2}} + \Norm{\bz_{t+1/2} - \bz_t + \eta\bs_t+\bz_{t+1/2} -\bz_{t+1}}\Norm{\hDelta_t}
\end{align*}
and the second inequality is based on the fact
\begin{align*}
  & \Norm{\bs_{t+1/2} - g(\bz_{t+1/2})}^2 \\
= & \Norm{\frac{1}{m}\sum_{i=1}^m \left(g_i(\vz_{t+1/2}(i)) - g_i(\bz_{t+1/2}^\top)\right)}^2   \\
\leq & \frac{1}{m}\sum_{i=1}^m\Norm{g_i(\vz_{t+1/2}(i)) - g_i(\bz_{t+1/2}^\top)}^2 \\
\leq & \frac{L^2}{m}\sum_{i=1}^m\Norm{\vz_{t+1/2}(i) - \bz_{t+1/2}^\top}^2 \\
= & \frac{L^2}{m}\Norm{\vz_{t+1/2} - \vone\bz_{t+1/2}}^2.
\end{align*}
Following the proof of Lemma \ref{lem:hzeta}, we have
\begin{align*}
\hzeta'_t 
\leq  \frac{\eta D^2}{\beta}  + \frac{12(\eta^2+\beta\eta)L^2\hdelta'^2}{m} + \frac{8\hC_1\hdelta'}{\sqrt{m}}.
\end{align*}
Summing over (\ref{ieq:cc-rel}) with $t=0,\dots,T-1$, we obtain
\begin{align*}
\begin{split}
& f(\hat x,y^*)-f(x^*,\hat y) \\
\leq &  \frac{1}{T}\sum_{t=0}^{T-1}(f(\bx_{t+1/2},y^*)-f(x^*,\by_{t+1/2}))  \\
\leq &  \frac{1}{T}\sum_{t=0}^{T-1}\inner{g(\bz_{t+1/2})}{\bz_{t+1/2}-z^*}  \\
\leq & \frac{\left(\Norm{\bz_0 - \bz^*}^2-\Norm{\bz_T - z^*}^2\right) + \sum_{t=0}^{T-1}\zeta'_t}{2\eta T} \\
\leq & \frac{\Norm{\bz_0 - z^*}^2}{\eta T} + \frac{1}{2\eta}\left(\frac{\eta D^2}{\beta}  + \frac{12(\eta^2+\beta\eta)L^2\hdelta'^2}{m} + \frac{8\hC_1\hdelta'}{\sqrt{m}}\right) \\
\leq & \frac{\eps}{2} + \frac{\eps}{4} + \frac{\eps}{8} + \frac{\eps}{8} = \eps.
\end{split}
\end{align*}
where the first inequality use Jensen's inequality; the second inequality use the objective function is convex-concave; the third inequality use the upper bound of $\zeta'_t$; and the last one is based on the value of parameter settings.

Since each iteration requires $\fO(n)$ SFO calls in expectation and we the algorithm needs to compute the full gradient at first, the total complexity is
\begin{align*}
\fO(nT) = \fO\left(\frac{nL}{\eps}\right).
\end{align*}
The number of communication round is 
\begin{align*}
KT+K_0 = \fO\left(\frac{L\sqrt{\chi}}{\eps}\log\left(\frac{L}{\eps}\right)\right).
\end{align*}
\end{proof}

\end{document}